\documentclass[a4paper,12pt]{amsart}

\usepackage{amssymb}
\usepackage{amsmath}
\usepackage{tikz}

\newtheorem{theorem}{Theorem}[section] 
\newtheorem{lemma}[theorem]{Lemma} 
\newtheorem{corollary}[theorem]{Corollary} 
\newtheorem{proposition}[theorem]{Proposition}
 
\theoremstyle{definition}
\newtheorem*{ack}{Acknowledgments}
\newtheorem{definition}[theorem]{Definition} 
\newtheorem{remark}[theorem]{Remark}

\title[]{Classification of degenerate Verma modules for $E(5,10)$.}

\author{Nicoletta Cantarini}\author{Fabrizio Caselli}\author{Victor Kac}
\address{Fabrizio Caselli and Nicoletta Cantarini, Dipartimento di matematica, Universit\`a di Bologna, Piazza di Porta San Donato 5, 40126 Bologna, Italy}

\email{fabrizio.caselli@unibo.it}
\email{nicoletta.cantarini@unibo.it}
\address {Victor Kac, Department of Mathematics, MIT, 77 Massachusetts Avenue, Cambridge, MA 02139, USA}

\email{kac@math.mit.edu}

 \DeclareMathOperator{\Sym}{Sym}
  \DeclareMathOperator{\Hom}{Hom}
\newcommand{\N}{\mathbb{N}}
\newcommand{\C}{\mathbb{C}}
\newcommand{\Z}{\mathbb{Z}}
\newcommand{\F}{\mathbb{C}}
\newcommand{\de}{\partial}
\newcommand{\slc}{\mathfrak{sl}_5}
\newcommand{\inlinewedge}{\textrm{\raisebox{0.6mm}{\footnotesize $\bigwedge$}}}
\newcommand{\displaywedge}{\textrm{\raisebox{0.6mm}{\tiny $\bigwedge$}}}
\newcommand{\g}{\mathfrak {g}}
\newcommand{\deb}{\bold \partial}
\begin{document}
	\maketitle
	\begin{abstract}
	Given a Lie superalgebra $\g$ with a subalgebra $\g_{\geq 0}$, and a finite-dimensional irreducible $\g_{\geq 0}$-module $F$, the induced $\g$-module $M(F)=\mathcal{U}(\g)\otimes_{\mathcal{U}(\g_{\geq 0})}F$ is called a finite Verma
	module. In the present paper we classify the non-irreducible finite Verma modules over the largest exceptional linearly compact Lie superalgebra $\g=E(5,10)$ with the subalgebra $\g_{\geq 0}$ of minimal codimension. This is done via classification of all singular vectors in the modules $M(F)$. Besides known singular vectors of degree 1,2,3,4 and 5, we discover two new singular vectors, of degrees 7 and 11. We show that the corresponding morphisms of finite Verma modules of degree 1,4,7, and 11 can be arranged in an infinite number of bilateral infinite complexes, which may be viewed as “exceptional" de Rham complexes for $E(5,10)$. 
\end{abstract}
\section{Introduction} Recall that a linearly compact Lie (super)algebra $\g$ is defined by the property that, viewed as a vector space, $\g$ is linearly compact. According to E.\ Cartan's classification, the list of infinite-dimensional simple linearly compact Lie algebras consists of four Lie--Cartan series: $W_n$, $S_n$, $H_n$, and $K_n$.

The infinite-dimensional simple linearly compact Lie superalgebras were classified in \cite{K} and explicitely described in \cite{CK}; all their maximal open subalgebras were classified in \cite{CaK}. The complete list consists of ten “classical" series (which include the Lie--Cartan series), and five exceptional
examples, denoted by $E(1,6)$, $E(3,6)$, $E(3,8)$, $E(4,4)$, and $E(5,10)$.
With the exception of $E(4,4)$, these Lie superalgebras carry a $\Z$--gradation, compatible with the parity:
$$\g=\bigoplus_{j\geq -d} \g_j,$$
where $d=2$ for $E(1,6)$, $E(3,6)$ and $E(5,10)$, and $d=3$ for $E(3,8)$.
Then $\g_{\geq 0}:=\oplus_{j\geq 0}\g_j$ is a maximal open subalgebra of $\g$
of minimal codimension.
In the case of $\g=E(3,6)$ and $E(3,8)$ the subalgebra $\g_0$ is isomorphic to $\mathfrak{sl}_3\oplus \mathfrak{sl}_2\oplus\C$, and for $\g=E(5,10)$, $\g_0$ is isomorphic to $\mathfrak{sl}_5$, which hints to connections to particle physycs \cite{KR2}.

Let $F$ be an irreducible finite-dimensional $\g_0$-module, extend it to $\g_{\geq 0}$ by letting all $\g_j$ with $j>0$ act by 0, and consider the {\em finite Verma $\g$-module}
$$M(F)=\mathcal{U}(\g)\otimes_{\mathcal{U}(\g_{\geq 0})}F,$$
where $M(F)$ is viewed as a vector space with discrete topology. These modules are especially interesting since their topological duals are linearly compact.

The first problem of representation theory of linearly compact Lie superalgebras is to classify their degenerate (i.e., non-irreducible) finite Verma modules and
morphisms between them. This is equivalent to classification of {\em singular vectors} in these modules, i.e., those which are annihilated by $\g_j$ with $j\geq 1$. This problem was solved for Lie algebras $W_n$, $S_n$ and $H_n$ by Rudakov
\cite{R1}, \cite{R0}. In particular, he showed that the degenerate finite $W_n$-modules form the de Rham complex in a formal neighborhood of 0 in $\C^n$ (rather its  topological dual).

In a series of papers \cite{KR1}, \cite{KR2}, \cite{KR3} this problem was solved for the exceptional linearly compact Lie superalgebra $E(3,6)$. It turned out
that all the morphisms between the degenerate finite Verma modules over $E(3,6)$ can be arranged in an infinite number of complexes, and cohomology of these complexes was computed in \cite{KR2} as well. The most difficult technical part of this work is \cite{KR3}, where all singular vectors have been classified.

In the subsequent paper \cite{KR} a solution to this problem was announced for $E(3,8)$, and a conjecture on classification of degenerate finite Verma modules for $E(5,10)$ was posed, motivated by the singular vectors of degree 1 constructed there
(the degree on $M(F)=\mathcal{U}(\g_{<0})\otimes F$ is induced by the degree on 
$\g_{<0}=\oplus_{j<0}\g_j$). In a more recent paper \cite{R} it was proved that these are all singular vectors of degree 1, and also some singular vectors of degree 2,3,4 and 5 have been constructed. In the subsequent paper \cite{CC} it was shown that the singular vectors of degree less than or equal to 3 constructed by Rudakov 
are all singular vectors of degree less than or equal to 3. Actually, the morphisms of degrees 2, 3 and 5 corresponding to singular vectors constructed in \cite{R} are composition of morphisms of degree 1 and 4, and the morphisms of degree 1 and 4 can be arranged in an infinite number of infinite complexes \cite{R}. However, in Figure 2 of \cite{R}
there are two notable gaps in the complexes.

The key discovery of the present paper is the existence of morphisms of degree 7 and 11, which fill these gaps (see Figure \ref{E510morphisms}). Moreover, we show that there are no further singular vectors (Theorem \ref{conclusions}), thereby proving the conjecture from \cite{KR}
on classification of degenerate finite Verma modules over $E(5,10)$.

The proof of Theorem \ref{conclusions} goes as follows. First, using a result
from \cite{R0} on $S_n$-modules for $n=5$, which is the even part of $E(5,10)$, we show that there are no singular vectors of degree greater than 14. Next we find that for degrees between 11 and 14 there is only one singular vector, it has degree 11 and defines a morphism from $M(\C^5)$ to $M({\C^5}^*)$, where $\C^5$
is the standard $\slc$-module and ${\C^5}^*$ its dual. After that, using the techniques of \cite{CC}, we show that in degrees between 6 and 10 the
only singular vector has degree 7 and it defines a morphism from $M(S^2\C^5)$
to $M(S^2{\C^5}^*)$. These are precisely the two morphisms, missing in Figure 2 of \cite{R}. Finally, we show that in degrees less than or equal to 6 there are no other singular vectors as compared to \cite{R}. The calculations involve solution of large
systems of linear equations, which are performed with the aid of computer.
Note also that the construction of morphisms is facilitated by the duality, constructed in \cite{CCK}, such that the morphism $M(F) \rightarrow M(F_1)$
induces the morphism $M(F_1)^*\rightarrow M(F)^*$ and for $E(5,10)$ has the property that $M(F)^*=M(F^*)$.

We have learned recently that Daniele Brilli obtained in \cite{B} the upper bound 12 on the degrees of singular vectors for finite Verma modules over $E(5,10)$, using
the techniques of representation theory of Lie pseudoalgebras.

\section{Preliminaries}\label{S1}
We let $\N=\{0,1,2,3,\dots\}$ be the set of non-negative integers and for $n\in\N$ we set $[n]=\{i\in\N ~|~ 1\leq i\leq n\}$.


We consider the simple, linearly compact Lie superalgebra of exceptional type $\g=E(5,10)$ whose even and odd parts are
as follows: $\g_{\bar{0}}$ consists of zero-divergence vector fields in five (even) indeterminates $x_1,\ldots,x_5$, i.e., 
\[\g_{\bar{0}}=S_5=\{X=\sum_{i=1}^5f_i\partial_i ~|~ f_i\in\C[[x_1,\dots,x_5]], \textrm{div}(X)=0\},\]
where $\partial_i=\partial_{x_i}$,
and $\g_{\bar{1}}=\Omega^2_{cl}$ consists of closed two-forms in the five indeterminates $x_1,\ldots,x_5$.
The bracket between a vector field and a two-form is given by the Lie derivative and for $f,g\in \C[[x_1,\dots,x_5]]$ we have
$$[fdx_i\wedge dx_j,g dx_k\wedge dx_l]=\varepsilon_{ijkl}fg\partial_{t_{ijkl}}$$ 
where, for $i,j,k,l\in [5]$, $\varepsilon_{ijkl}$ and $t_{ijkl}$ are defined as follows: if $|\{i,j,k,l\}|=4$ we let $t_{ijkl}\in [5]$ be such that $|\{i,j,k,l,t_{ijkl}\}|=5$ and $\varepsilon_{ijkl}$ be the sign of the permutation 
$(i,j,k,l,t_{ijkl})$. If $|\{i,j,k,l\}|<4$ then $\varepsilon_{ijkl}=0$. 

From now on we shall denote $dx_i\wedge dx_j$ simply by $d_{ij}$.

The Lie superalgebra $\g$ has a consistent, irreducible, transitive $\Z$-grading of depth 2 where,
for $k\in\N$,
\begin{align*}
\g_{2k-2}&=\langle f\partial_i ~|~i=1,\dots,5, f\in\C[[x_1,\dots, x_5]]_{k}\rangle\cap S_5\\
\g_{2k-1}&=\langle fd_{ij} ~|~ i,j=1,\dots,5, f\in\C[[x_1,\dots, x_5]]_{k}\rangle\cap\Omega^2_{cl}
\end{align*}
where 
by $\C[[x_1,\dots, x_5]]_{k}$ we denote the homogeneous component of $\C[[x_1,\dots, x_5]]$ of degree $k$.

Note that $\g_0\cong \mathfrak{sl}_5$, $\g_{-2}\cong (\C^5)^*$, $\g_{-1}\cong \inlinewedge^2\C^5$ as $\g_0$-modules (where $\C^5$ denotes the
standard $\mathfrak{sl}_5$-module).
We set $\g_{-}=\g_{-2}\oplus \g_{-1}$, $\g_{+}=\oplus_{j>0}\g_j$ and $\g_{\geq 0}=\g_0\oplus \g_+$.

We denote by $U$ (resp.\ $U_{-}$) the universal enveloping algebra of $\g$ (resp.\ $\g_-$). Note that $U_-$ is
a $\g_0$-module with respect to the adjoint action: for $x\in \g_0$ and $u\in U_-$,
$$x.u=[x,u]=xu-ux.$$
We also point out that the $\Z$-grading of $\g$ induces a $\Z$-grading on the enveloping algebra $U_-$.
It is customary, though, to invert the sign of the degrees hence getting a grading over $\N$. Note that the homogeneous component
$(U_-)_d$ of degree $d$
of $U_-$ under this grading is a $\g_0$-submodule. 

We fix the Borel subalgebra $\langle x_i\partial_j, h_{ij}=x_i\partial_i-x_j\partial_j ~|~ i<j\rangle$ of $\g_0$ and we consider the usual base of the corresponding root system given by $\{\alpha_{12},\ldots,\alpha_{45}\}$. We let $\Lambda$ be the weight lattice of $\frak{sl}_5$ and we express all weights of $\frak{sl}_5$ using their coordinates with respect to the fundamental weights $\varphi_{12},\varphi_{23},\varphi_{34},\varphi_{45}$, i.e., for $\lambda\in \Lambda$ we write $\lambda=(\lambda_{12},\ldots,\lambda_{45})$ for some $\lambda_{i\,i+1}\in \mathbb Z$ to mean $\lambda=\lambda_{12}\varphi_{12}+\cdots+\lambda_{45}\varphi_{45}$.

 If $\lambda\in \Lambda$ is a weight,  we use the following convention: for all $1\leq i<j\leq 5$ we let
\[
\lambda_{ij}=\sum_{k=i}^{j-1}\lambda_{k\,k+1}.
\]
If $V$ is a $\frak {sl}_5$-module and $v\in V$ is a weight vector we denote by $\lambda(v)$ the weight of $v$ and by $\lambda_{ij}(v)=(\lambda(v))_{ij}$.

If $\lambda=(a,b,c,d)\in \Lambda$ is a dominant weight, i.e. $a,b,c,d\geq 0$, let us denote by $F(\lambda)=F(a,b,c,d)$ the irreducible $\mathfrak{sl}_5$-module of highest 
weight $\lambda$. In this paper we always think of $F(a,b,c,d)$ as the irreducible submodule of \[\Sym^a(\C^5)\otimes \Sym^b(\displaywedge^2(\C^5))\otimes 
\Sym^c(\displaywedge^2(\C^5)^*)\otimes \Sym^d((\C^5)^*)\] generated by the highest weight vector $x_1^ax_{12}^b{x_{45}^*}^c{x_5^*}^d$,
where  $\{x_1,\dots, x_5\}$ denotes the standard basis of $\C^5$, $x_{ij}=x_i\wedge x_j$, and $x_i^*$ and $x_{ij}^*$ are the corresponding dual
basis elements. Besides, for a weight $\lambda=(a,b,c,d)$ we let  $\lambda^*=(d,c,b,a)$, so that $F(\lambda)^*\cong F(\lambda^*)$.
 
Notice that, as a $\g_0$-module, $\g_1\cong F(1,1,0,0)$ and that $x_5d_{45}$ is a lowest weight vector in $\g_1$. Moreover,
for $j\geq 1$, we have $\g_j=\g_1^j$.

\section{Generalized Verma modules and morphisms}\label{S3}
We recall the definition and some properties of (generalized) Verma modules over $E(5,10)$, most of which hold in the generality of arbitrary $\mathbb Z$--graded Lie superalgebras (for  some detailed proofs see \cite{CC}). 
 
Given a $\g_0$-module $V$, we extend it to
a $\g_{\geq 0}$-module by letting $\g_+$ act trivially, and define
$$M(V)=U\otimes_{U(\g_{\geq 0})}V.$$
Note that $M(V)$ has a $\g$-module structure by multiplication on the left, and is called the (generalized) Verma module associated to $V$. We also observe that $M(V)\cong U_{-}\otimes_{\C}V$ as $\g_0$-modules.

 If the $\g_0$-module $V$ is finite-dimensional and irreducible, then we call $M(V)$ a finite Verma module (it is finitely-generated as a $U_-$-module). We denote by $M(\lambda)$ the finite Verma module $M(F(\lambda))$. A finite Verma module is said to be non-degenerate if it is irreducible and degenerate otherwise.
\begin{definition}
We say that an element $w\in M(V)$ is homogeneous of degree $d$ if $w\in (U_-)_d\otimes V$.
\end{definition}
\begin{definition} A vector $w\in M(V)$ is called a {\em singular} vector if it satisfies the following conditions:
\begin{itemize}
\item[(i)] $x_i\partial_{i+1}w=0$ for every $i=1,\dots,4$;
\item[(ii)] $zw=0$ for every $z\in \g_1$;
\item[(iii)] 
$w$ does not lie in $V$.
\end{itemize}
\end{definition}
We observe that the homogeneous components of positive degree of a singular vector are singular vectors. 
The same holds for its weight components. From now on we will thus assume that a singular vector is a homogeneous weight
vector unless otherwise specified.
Notice that if  condition (i) is satisfied then condition (ii) holds if $x_5d_{45}w=0$ since $x_5d_{45}$ is a lowest weight vector in $\g_1$. 

We recall that a minimal Verma module $M(V)$ is degenerate if and only if it contains a singular vector \cite[Proposition 3.3]{CC}.

Degenerate Verma modules can be described in terms of morphisms.
A morphism $\varphi: M(V)\rightarrow M(W)$ can always be associated to an element $\Phi\in U_{-}\otimes \Hom(V,W)$ as follows: for $u\in U_-$ and $v\in V$ we let
$$\varphi(u\otimes v)=u\Phi(v)$$
where, if $\Phi=\sum_iu_i\otimes \theta_i$ with $u_i\in U_-$,
$\theta_i\in \Hom(V,W),$ we let $\Phi(v)=\sum_iu_i\otimes \theta_i(v)$.
We will say that $\varphi$ (or $\Phi$) is a morphism of degree $d$ if $u_i\in (U_-)_d$ for every $i$.

\medskip

The following proposition characterizes morphisms between Verma modules.

\begin{proposition}\cite{KR,R}\label{morphisms}
Let $\varphi: M(V)\rightarrow M(W)$ be the linear map associated with the element $\Phi\in U_{-}\otimes \Hom(V,W)$.
Then $\varphi$ is a morphism of $\g$-modules if and only if the following conditions
hold:
\begin{itemize}
\item[(a)] $\g_0.\Phi=0$;
\item[(b)] $X\varphi(v)=0$ for every $X\in \g_1$ and for every $v\in V$.
\end{itemize} 
\end{proposition}
We observe that, if $M(V)$ is a finite Verma module and condition (a) holds, it is enough to verify  condition (b) for an element $X$ generating
$\g_1$ as a $\g_0$-module and for $v$ a highest weight vector in $V$.

We recall that   a finite Verma module $M(\mu)$  contains a singular vector if and only if
 there exist a finite Verma module $M(\lambda)$ and a morphism $\varphi:M(\lambda)\rightarrow M(\mu)$ of positive degree
 \cite[Proposition 3.5]{CC}.

We recall the following duality on Verma modules which is established in \cite{CCK} in a much wider generality. 
\begin{theorem}\label{duality}
Let $\varphi:M(\lambda)\rightarrow M(\mu)$ be a morphism of $\g$-modules of degree $d$. Then there exists a dual morphism $\varphi^*:M(\mu^*)\rightarrow M(\lambda^*)$ of the same degree $d$. Equivalently, if $M(\lambda)$ contains a singular vector of degree $d$ and weight $\mu$, then $M(\mu^*)$ contains a singular vector of degree $d$ and weight $\lambda^*$.
\end{theorem}
\begin{remark}\label{dual}
Let $\varphi: M(V)\rightarrow M(W)$ be a linear map of degree $d$ associated to an element $\Phi\in U_-\otimes \Hom(V,W)$ that satisfies condition (a) of Proposition \ref{morphisms}. Then there exists a $\g_0$-morphism $\psi: (U_-)_d^*\rightarrow \Hom(V,W)$ such that
$\Phi= \sum_i u_i\otimes \psi(u_i^*)$ where $\{u_i, i\in I\}$ is any basis of $(U_-)_d$ and $\{u_i^*, i\in I\}$ is the corresponding dual basis.
\end{remark}

\begin{definition} Let $M(\mu)$ be a finite Verma module and let $\pi: M(\mu)\rightarrow U_-\otimes F(\mu)_{\mu}$ be the natural projection,
$F(\mu)_{\mu}$ being the weight space of $F(\mu)$ of weight $\mu$. Given a singular vector $w\in M(\mu)$, we call $\pi(w)$ the leading term of $w$.
\end{definition}

It is shown in \cite{CC} that the leading term of a singular vector is non-zero, and therefore a singular vector is uniquely determined by its leading term.

\medskip
The action of $E(5,10)$ on a module $M$ restricts to an action of its even part on $M$. It is therefore natural to take into account the structure of $M$ as an $S_5$-module also. In order to do this 
we consider the grading on $S_5$ given by $\deg x_i=2$ and $\deg (\de_i)=-2$ to be consistent with the embedding of $S_5$ in $E(5,10)$. The definition of a Verma module for $S_5$ is analogous to the one for $E(5,10)$. Rudakov classified all singular vectors for the infinite-dimensional Lie algebra $S_n$ in \cite{R0} and we recall here his results in the special case of $S_5$. 
\begin{theorem}\cite{R0}\label{classS5}
The following is a complete list (up to multiplication by a scalar) of singular vectors $w$ in Verma modules $M(\lambda)$ for $S_5$.
\begin{itemize}

\item[R1.] $\lambda=(1,0,0,0)$, $w=\de_1 \otimes x_1+ \de_2 \otimes x_2  +\de_3 \otimes x_3 + \de_4 \otimes x_4 +\de_5\otimes x_5$;
\item [R2.] $\lambda=(0,1,0,0)$, $w=\de_2 \otimes x_{12}+ \de_3 \otimes x_{13}+ \de_4 \otimes x_{14}+\de_5 \otimes x_{15}$;
\item [R3.] $\lambda=(0,0,1,0)$, $w=\de_3 \otimes x_{45}^*+ \de_4 \otimes x_{53}^*+ \de_5 \otimes x_{34}^*$;
\item [R4.] $\lambda=(0,0,0,1)$, $w=\de_4 \otimes x_5^*- \de_5 \otimes x_4^*$;
\item [R5.] $\lambda=(0,0,0,0)$, $w=\de_5 \otimes 1$;
\item[R6.] $\lambda=(1,0,0,0)$, $w=\de_5(\de_1 \otimes x_1+ \de_2 \otimes x_2  +\de_3 \otimes x_3 + \de_4 \otimes x_4 +\de_5\otimes x_5)$;
\end{itemize}  
\end{theorem}

\medskip
\noindent
Theorem \ref{classS5} provides the diagram of all non-zero morphisms between finite Verma modules for $S_5$ shown in Figure \ref{S5morphisms}.
\begin{figure}[h]
	\begin{center}
		$$
		\begin{tikzpicture}
		\draw[->, line width=3pt](1.8,0)--(-1.8,0);
		\draw[->, line width=3pt](3.17,3.6)--(2.07,0.2);
		\draw[->, line width=3pt](0.16,6.62)--(3.06,3.96);
		\draw[->, line width=3pt](-3.06,3.96)--(-0.16,6.62);
		\draw[->, line width=3pt](-2.07,0.2)--(-3.17,3.6);
		\draw[->, line width=3pt](-3,3.8)--(3,3.8);
		
		\draw[fill=black]{(2,0) circle(3pt)};
		\draw[fill=black]{(-2,0) circle(3pt)};
		\draw[fill=black]{(3.23,3.8) circle(3pt)};
		\draw[fill=black]{(-3.23,3.8) circle(3pt)};
		\draw[fill=black]{(0,6.75) circle(3pt)};
		
		\node at (2.3,-0.5) {$M(0,1,0,0)$};
		\node at (-2.3,-0.5) {$M(0,0,1,0)$};
		\node at (4.6,3.8) {$M(1,0,0,0)$};
		\node at (-0.2,7.2) {$M(0,0,0,0)$};
		\node at (-4.6,3.8) {$M(0,0,0,1)$};
		\node at (1.9,5.5){$\varphi_1$};
		\node at (-1.9,5.5){$\varphi_5$};
		\node at (0,4.1){$\varphi_6=\varphi_1 \circ \varphi_5$};
		\node at (-3,1.8){$\varphi_4$};
		\node at (3,1.8){$\varphi_2$};
		\node at (0,0.3){$\varphi_3$};
		
		\end{tikzpicture}$$
	\end{center}
	\caption{\label{S5morphisms} All non-zero morphisms between finite Verma modules for $S_5$. External morphisms have degree 2, and the internal one has degree 4. The morphisms $\varphi_1, \dots, \varphi_6$ correspond to the singular vectors in R1, \dots, R6 in Theorem \ref{classS5}. } 
	\end{figure}
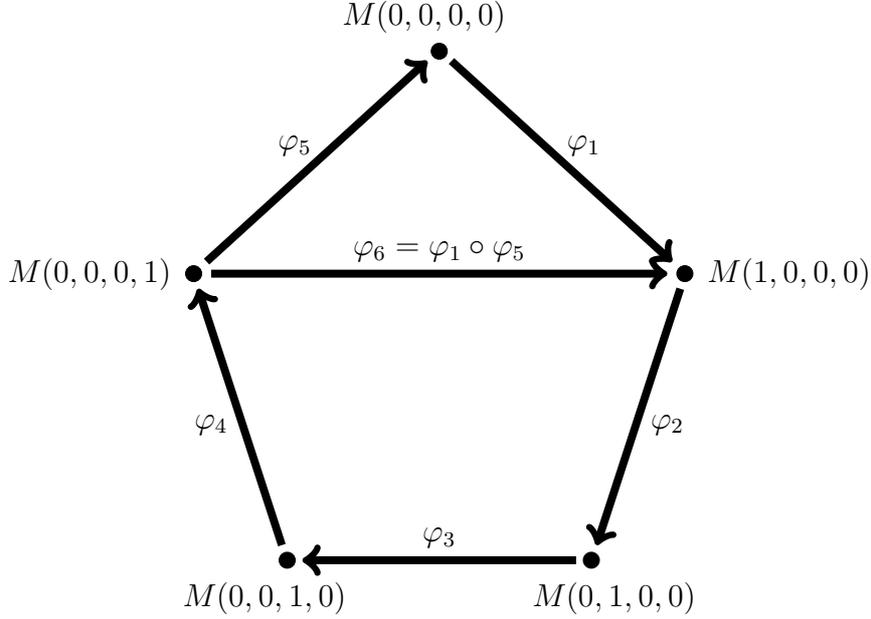
	
\section{A first bound}\label{first}
Let $\Omega=\{\{1,2\},\{1,3\},\{1,4\},\{1,5\},\{2,3\},\{2,4\},\{2,5\},\{3,4\}, \{3,5\},\{4,5\}\}$ and, if $p=\{i,j\}\in \Omega$ with $i<j$, then we let $d_p=d_{ij}=dx_i\wedge dx_j$. In order to avoid cumbersome notation, when no confusion may arise we will denote in this section the subset $\{i,j\}$ simply as $ij$.

Let $V$ be a finite dimensional $\g_0$-module. For all $k\geq 0$ we let
\[
M_k(V)=\F[\deb] \sum_{j\leq k} \sum_{p_1,\ldots,p_j \in \Omega} \F d_{p_1}\cdots d_{p_j}\otimes V.
\]

Note that $M_k(V)$ is not an $E(5,10)$-submodule of $M(V)$. Nevertheless the following result holds.
\begin{proposition}\label{s5mod}
	For all $k=0,1,\ldots,10$ the subspace $M_k(V)$ is an $S_5$-module.
\end{proposition}
\begin{proof}
	It is enough to show that for all $X\in S_5$, $1\leq j\leq k$, $p_1,\ldots,p_j\in \Omega$, and $v\in V$
	\begin{equation}\label{s5mod1}
	Xd_{p_1}\cdots d_{p_j}\otimes v \in M_j(V),
	\end{equation}
	since $M_j(V)\subseteq M_k(V)$. We also show that
	\begin{equation}
	\label{s5mod2} [X,d_{p_1}]d_{p_2}\cdots d_{p_j}\otimes v\in M_j(V)
	\end{equation}
	and we prove that \eqref{s5mod1} and \eqref{s5mod2} hold simultaneously by a double induction on $j$ and $\deg X$. 
	
	If $j=1$ then \eqref{s5mod2} is trivial and \eqref{s5mod1} follows from \eqref{s5mod2}.
	
	If $\deg X=-2$ then \eqref{s5mod1} and \eqref{s5mod2} are both trivial, so we assume that $j\geq 2$ and $\deg X\geq 0$.
	
	We have
	\[
	Xd_{p_1}\dots d_{p_j}\otimes v=[X,d_{p_1}]d_{p_2}\cdots d_{p_j}\otimes v + d_{p_1} X d_{p_2}\cdots d_{p_j} \otimes v.
	\]
	The latter summand clearly lies in $M_j(V)$ by induction on $j$ and so  \eqref{s5mod1} will follow from \eqref{s5mod2}. We have
	\[
	[X,d_{p_1}]d_{p_2}\cdots d_{p_j}\otimes v = -d_{p_2}[X,d_{p_1}]d_{p_3}\cdots d_{p_j}\otimes v +[[X,d_{p_1}],d_{p_2}]d_{p_3}\cdots d_{p_j}\otimes v.
	\]
	The former summand lies in  $M_j(V)$ by induction on $j$ and the latter by induction on $\deg X$: the result follows.
	\end{proof}
By Proposition \ref{s5mod} we have a filtration
\[
\{0\}=M_{-1}(V)\subseteq \F[\deb]\otimes V=M_0(V) \subseteq M_1(V) \subseteq \cdots \subseteq M_{10}(V)=M(V)
\]
of $S_5$-modules and we let
\[
N_k(V)=M_k(V)/M_{k-1}(V)
\]
for all $k=0,\ldots,10$.
\begin{proposition}\label{NkV}
	For all $k=0,\ldots,10$ and for any total order $\prec$ on $\Omega$ we have
	\[
	N_k(V)\cong \F[\deb]\otimes \bigoplus_{p_1\prec \cdots \prec p_k} \F d_{p_1}\cdots d_{p_k}\otimes V
	\]
as $\C$-vector spaces. 
\end{proposition}
\begin{proof}
	For all $p_1,\ldots,p_k \in \Omega$ and every permutation $\sigma$ of the indices $\{1,\ldots,k\}$ we have 
	\begin{equation}\label{permform}d_{p_1}d_{p_2}\cdots d_{p_k}-\varepsilon(\sigma) d_{p_{\sigma(1)}}\cdots d_{p_{\sigma(k)}}\in M_{k-1}(V)
	\end{equation}
	and so $N_k(V)$ is generated as $\F[\deb]$-module by the elements $d_{p_1}\cdots d_{p_k}\otimes v$ for all $p_1\prec \cdots \prec p_k$ and all $v\in V$. The result follows by Poincar\'e--Birkhoff--Witt theorem for $U(\g_{-})$.
\end{proof}	

Next we observe that the subspace
\[
F_k(V)=\bigoplus_{p_1\prec \cdots \prec p_k}\F d_{p_1}\cdots d_{p_k}\otimes V
\]
of $N_k(V)$ also has a special structure:
\begin{proposition}\label{NkVisVerma}
	The subspace $F_k(V)$ of $N_k(V)$ is an $\slc$-module annihilated by $(S_5)_{>0}$.
	The $S_5$-module $N_k(V)$ is the finite Verma module for $S_5$ induced by $F_k(V)$, i.e.
	\[
	N_k(V)=M(F_k(V)).
	\]
\end{proposition}
\begin{proof}
The subspace $F_k(V)$ of $N_k(V)$ is an $\slc$-module since $\g_{-1}$ is a $\g_0$-module, and by the definition of $N_k(V)$. The fact that $F_k(V)$ is annihilated by $(S_5)_{>0}$ follows easily by degree reasons.  
The second part follows from Proposition \ref{NkV} and the first part.
\end{proof}
This result together with Theorem \ref{classS5} allows us to determine a first bound on the degree of singular vectors for $E(5,10)$.
\begin{corollary}
	Let $M(V)$ be a Verma module for $E(5,10)$ and $w\in M(V)$ be a singular vector. Then $w$ has degree at most 14.
\end{corollary}
\begin{proof}
Let $k$ be minimal such that $w\in M_k(V)$. Then $w$ is a fortiori either a highest weight vector  in $F_k(V)$ or a singular vector in the $S_5$-Verma module $N_k(V)$, and as such it has  degree at most 4. It follows that $w$ has degree at most $k+4$, where $k\leq 10$.
\end{proof}
\section{Singular vectors of degree greater than 10}\label{>10}
The description of singular vectors for $S_5$ allows us to give a much more precise description of possible singular vectors for $E(5,10)$ of degree greater than 10.
We fix a total order $\prec$ on the set $\Omega=\{\{1,2\},\{1,3\},\ldots,\{4,5\}\}$. If $I=\{p_1,\ldots,p_j\}\subseteq \Omega$ with $p_1 \prec \cdots \prec p_j$ we let $d^\prec_I=d_{p_1}\cdots d_{p_j}$.  
We let 
\[
L^{\prec}_{h}(V)=\F[\deb] \bigoplus_{I:\, |I|=h}\F d^\prec_I \otimes V.
\]
By construction we have
\[
M_h(V)=L^\prec _{h}(V)\oplus M_{h-1}(V)
\]
for all $h=0,\ldots,10$ and in particular 
\[
M(V)=\bigoplus_{h=0}^{10} L^{\prec}_{h}(V).
\]

Every non-zero vector $w$ in $M(V)$ can be expressed uniquely in the following form:
\[
w=w_h^\prec+w_{h-1}^\prec+\cdots +w_0^\prec,
\]
for some $h=0,\dots, 10$,
with $w_h^\prec\neq 0$ and $w_j^\prec \in L^\prec_j(V)$ for all $j\leq h$. We say in this case that $w$ has {\em height} $h$ and we call $w_h^\prec$ the {\em highest term} of $w$. Note that the height of an element does not depend on the order $\prec$, while its highest term does.

If $M=(m_1,\ldots,m_5)\in \N^5$ we let $\de^M=\de_1^{m_1}\cdots \de_5^{m_5}$ and $|M|=m_1+\cdots+m_5$. Moreover we let $e_1=(1,0,0,0,0)$, $e_2=(0,1,0,0,0),\ldots,e_5=(0,0,0,0,1)$.

If $w$ is homogeneous of degree $d$, then the term $w_j^\prec$ has the following form
\begin{equation}\label{formofterms}
w_j^\prec=\sum_{I\subseteq \Omega:\, |I|=j}\sum_{M\in \mathbb N^5:\, |M|=\frac{d-j}{2}}\deb^M d^\prec_{I}\otimes v_{M,I},
\end{equation}
where $v_{M,I}\in V$.

Note that if $w$ is homogeneous of height $h$, then $w_j^\prec=0$ for all $j\not \equiv h \mod 2$. Observe that, by construction, if $w$ has height $h$, then
\[
w\equiv w_h^\prec \mod M_{h-1}(V),
\]
and, in particular, $w$ and $w_h^\prec$ lie in the same class in $N_h(V)$. Theorem \ref{classS5} provides us the following description of possible singular vectors for $E(5,10)$.
\begin{corollary}\label{d=h,h+4}
	Let $w\in M(V)$ be a singular vector of degree $d$ and height $h$, and let $w_h^\prec$ be its highest term. Let $w_j^\prec$ be as in \eqref{formofterms}. Then one of the following applies:
	\begin{enumerate}
		\item[(i)] $d=h$;
		\item[(ii)] $d=h+2$ and there exists $i\in [5]$ such that 
		\[\sum_{I:\, |I|=h}d^\prec_I v_{e_i,I}\neq 0
		\]
		is a highest weight vector for $\slc$ in $N_h(V)$ and 
		\[
		w_h^\prec =\sum_{j=i}^5\partial _j \sum_{I:\, |I|=h} d^\prec_I\otimes v_{e_j,I}
		\]
		with
		\[\lambda(w)=\begin{cases}
		(0,0,0,0) & \textrm{if }i=1;\\
			(1,0,0,0) & \textrm{if }i=2;\\
				(0,1,0,0) & \textrm{if }i=3;\\
					(0,0,1,0) & \textrm{if }i=4;\\
						(0,0,0,1) & \textrm{if }i=5;\\
		\end{cases}\hspace{5mm}
		\textrm{ and }\hspace{5mm} \lambda\Big(\sum_{I:\, |I|=h}d^\prec_I v_{e_i,I}\Big)= \begin{cases}
		(1,0,0,0) & \textrm{if }i=1;\\
		(0,1,0,0) & \textrm{if }i=2;\\
		(0,0,1,0) & \textrm{if }i=3;\\
		(0,0,0,1) & \textrm{if }i=4;\\
		(0,0,0,0) & \textrm{if }i=5;\\
		\end{cases}\]
		\item[(iii)] $d=h+4$,\[\sum_{I:\, |I|=h}d_I v_{e_1+e_5,I}\neq 0
		\]
		is a highest weight vector for $\slc$ in $N_h(V)$ and 
		\[
		w_h^\prec =\partial _5\sum_{j=1}^5\partial _j \sum_{I:\, |I|=h} d^\prec_I\otimes v_{e_j+e_5,I}
		\]
		with $\lambda(w)=(0,0,0,1)$ and $\lambda\big(\sum_{I:\, |I|=h}d_I v_{e_1+e_5,I}\big)=(1,0,0,0)$.
\end{enumerate}
\end{corollary}	
\begin{proof}
	This is a straightforward consequence of Theorem \ref{classS5}. 
	We know that $N_h(V)$ is a Verma module for $S_5$ and $w_h^\prec$ is annihilated by $(S_5)_{>0}$ and by $x_i \de_j$ for all $i<j$. In particular, if $d\neq h$, we have that  the class of $w_h^\prec$ in $N_h(V)$ is a genuine singular vector for $S_5$: the classification of singular vectors in Theorem \ref{classS5} then completes the proof. Note that if $d=h$, then the class of $w_h^\prec$ in $N_h(V)$ is actually a highest weight vector in $F_h(V)$, i.e. the $\slc$-module we are inducing from. 
\end{proof}

In this section we classify all possible singular vectors with degree strictly bigger than height, i.e. we treat the cases $d=h+2$ and $d=h+4$ in Corollary \ref{d=h,h+4}, and, in particular, we find all singular vectors of degree greater than 10. 
We fix the lexicographic order on $\Omega=\{\{1,2\},\{1,3\},\ldots,\{4,5\}\}$, i.e. we set
\[
\{1,2\}\prec \{1,3\}\prec\{1,4\}\prec \{1,5\} \prec \{2,3\}\prec\{2,4\}\prec \{2,5\}\prec \{3,4\}\prec\{3,5\}\prec\{4,5\}
\]
and we simply write $L_h(V)$ instead of $L^\prec_h(V)$, $w_h$ instead of $w^\prec _h$ and $d_I$ instead of $d_I^\prec$.

\begin{remark}\label{firstred}The following inclusions are immediate from the definition of the action of $\g_0$ and $\g_1$ on  $M(V)$:
	\begin{align}\label{g0action} \g_0.L_h(V)&\subseteq L_h(V)\oplus L_{h-2}(V)\\
	\label{x5d45action} \g_1.L_h(V) &\subseteq L_{h+1}(V)\oplus L_{h-1}(V)\oplus L_{h-3}(V).
	\end{align}
\end{remark}

Due to (\ref{g0action}), for $X\in\g_0$ and $w\in L_h(V)$, we adopt the following notation:
\begin{equation}
X w=X^0 w+X^{-2}w
\label{not1}
\end{equation}
with $X^0w \in L_h(V)$ and $X^{-2}w \in L_{h-2}(V)$.
Similarly, due to (\ref{x5d45action}), for $X\in \g_1$ and $w\in L_h(V)$ we write:
\begin{equation}
Xw=X^1w+X^{-1}w+X^{-3}w
\label{not2}
\end{equation}
with $X^1w \in L_{h+1}(V)$, $X^{-1}w \in L_{h-1}(V)$ and $X^{-3}w \in L_{h-3}(V)$.
The following simple observation will be crucial in the sequel.
\begin{remark}
	Let $w\in M(V)$ be a singular vector of height $h$. Then for all $X\in \g_1$ we have
	\begin{align}\label{1act} &X^1w_h=0\\ 
	\label{-1act}&X^{-1}w_h+X^1w_{h-2}=0.
	\end{align} 
Moreover, for all $i=1,2,3,4$ and $E_i=x_i \de_{i+1}\in\g_0$ we have 
\begin{align}\label{0act}&E_i^0 w_h=0\\
\label{-2act}&E_i^{-2}w_h+E_i^0w_{h-2}=0.
\end{align}
	 \end{remark}
 It will be convenient to rephrase \eqref{1act} in the following equivalent way:
 for all $X\in \g_1$ we have
 \begin{equation}\label{1actbis}
 Xw_h\equiv 0 \mod M_h(V).
 \end{equation}
\begin{proposition}\label{d=h+4prima}Let $w$ be a singular vector in $M(F)$ with height $h$ and degree $d$ with $d=h+4$. Then $d=14$ and $F=F(1,0,0,0)$ is the standard representation of $\g_0$.
\end{proposition}
\begin{proof}
By Corollary \ref{d=h,h+4} we have
\[
w_h=\de_5\sum_{i,I}\de_i d_I\otimes v_{i,I}.
\]
By applying \eqref{1actbis} with $X=x_kd_{kj}$ and all $k\neq j$,  we deduce that if $v_{i,I}\neq 0$ then $I$ must contain all pairs containing $i$ and all pairs containing $5$, and, in particular, $w$ has height at least 7 since $v_{1,I}\neq 0$ for some $I$.
If we apply \eqref{1actbis} with $X=x_1d_{23}+x_2d_{13}$, we obtain
\[
-\de_5 d_{23}\sum_I d_I v_{1,I}-\de_5 d_{13}\sum_{I} d_I v_{2,I}\equiv 0 \mod M_h(V).
\]
We deduce that, if $v_{1,I}\neq 0$ and $I$ does not contain $23$, then it necessarily contains $24$, since all terms in the second summand do, and one can similarly show that $I$ must contain $34$ using $X=x_1d_{23}-x_3 d_{12}$. Permuting the roles of $2,3,4$, this argument shows that $I$ must contain at least two of the three pairs 23, 24, 34 and hence $w$ has height at least 9.
A singular vector of height 9 and degree 13 produces a morphism $\varphi:M(0,0,0,1)\rightarrow M(\lambda)$ for some $\lambda$, by Corollary \ref{d=h,h+4}. The dual morphism $\varphi^*:M(\lambda^*)\rightarrow M(1,0,0,0)$ is also a morphism of degree 13 and so we necessarily have $\lambda=(1,0,0,0)$. Therefore, if $v_{1,I}\neq 0$ then the weight of $d_I$ must be $(0,0,0,0)$, but one can easily check that there are no $I$ with $|I|=9$ such that $\lambda(d_I)=(0,0,0,0)$ (see \cite[\S 6]{CC} for an easy way to compute the weight of the $d_I$'s).

If $w$ has height 10 and degree 14, then by Corollary \ref{d=h,h+4}, and an argument analogous to the previous one shows that $w\in M(1,0,0,0)$.

\end{proof}

Now we can rule out the only left case with $d=h+4$.
\begin{proposition}\label{d!=h+4}Let $w$ be a singular vector of degree $d$ and height $h$. Then $d< h+4$.
\end{proposition}
\begin{proof}
By Propositions \ref{d=h,h+4} and \ref{d=h+4prima} we can assume that $d=14$, $h=10$ and
\[
w_{10}=\de_5 (\alpha_1\de_1 d_{\Omega}\otimes x_1+\cdots +\alpha_5\de_5 d_{\Omega}\otimes x_5),
\]
for some $\alpha_1,\ldots,\alpha_5 \in \F$ with $\alpha_1\neq 0$,
and that $w_8$ has the following form:
\[
w_8=\sum_{I:|I|=8} \sum_{M: |M|=3} \sum_{k=1}^5\alpha_{M,I,k}\de^M d_I \otimes x_k, 
\]
for some $\alpha_{M,I,k}\in \F$.
If we expand
\[x_5 d_{45}(w_{10}+w_8)=\sum \beta_{M,I,k} \de^M d_I\otimes x_k,\]
by \eqref{-1act} we obtain the relation
\[
\beta_{(1,0,0,1,0),\Omega\setminus \{23\},4}=-\alpha_4-\alpha_{(1, 0, 0, 1, 1), \Omega\setminus\{23,45\},4}=0.
\]
Similarly, if we expand
\[x_4 d_{45}(w_{10}+w_8)=\sum \gamma_{M,I,k} \de^M d_I\otimes x_k,\]
by \eqref{-1act} we obtain te relation
\[
\gamma_{(1,0,0,0,1),\Omega\setminus\{23\},4}=\alpha_1-\alpha_4-\alpha_{(1, 0, 0, 1, 1), \Omega\setminus\{23,45\},4}=0,
\]
and hence $\alpha_1=0$, a contradiction.
\end{proof}
Next target is to deal with the case $d=h+2$ in Corollary \ref{d=h,h+4}.
\begin{proposition}\label{hgeq8}
Let $w$ be a singular vector in $M(V)$ of degree $d$ and height $h$, with $d=h+2$. Then $h\geq 8$. Moreover, if $h=8$ and 
\[w_8 =\sum_{j=i}^5\partial _j \sum_{I:\, |I|=8} d_I\otimes v_{j,I}\] 
as in Corollary \ref{d=h,h+4}, then $v_{j,I}\neq 0$ only if $\lambda(\de_jd_I)=(0,0,0,0)$.
\end{proposition}
\begin{proof}

By \eqref{1actbis} we have for all $k\neq l$
\[
x_kd_{kl}\, w_h\equiv -d_{kl}\sum_I d_I\otimes v_{k,I}\equiv 0 \mod M_h(V).
\]
This implies that, if $v_{k,I}\neq 0$, then $\{k,l\}\in I$ for all $l\neq k$. In particular, we immediately deduce that $h\geq 4$ and 
\[
w_h\equiv\sum_{j} \de_j d_{1j}\cdots \hat{d_{jj}}\cdots d_{5j}\sum_{I_j}d_{I_j}\otimes v_{j,I_j} \mod M_{h-1}(V),
\]
where $I_1$ runs through all subsets of $\{\{2,3\},\{2,4\},\{2,5\},\{3,4\},\{3,5\},\{4,5\}\}$ of cardinality $h-4$, and similarly for $I_2,\ldots,I_5$,  where the vectors $v_{j,I_j}$ have been reindexed.
Now let $k,l,m$ be distinct integers in $[5]$ and use again \eqref{1actbis} with the element $X=x_kd_{lm}+x_ld_{km}$. We obtain:
\[
d_{lm} d_{1k}\cdots \hat{d_{kk}}\cdots d_{5k}\sum_{I_k}d_{I_k}\otimes v_{k,I_k}+d_{km}d_{1l}\cdots \hat{d_{ll}}\cdots d_{5l}\sum_{I_l}d_{I_l}\otimes v_{l,I_l}\equiv 0 \mod M_h(V).
\]
Again, by \eqref{permform}, this implies that, if $v_{k,I_k}\neq 0$, then  $I_k$ contains $\{l,m\}$ (in which case the corresponding summand is zero), or it contains both pairs $\{l,r\}$ and $\{l,s\}$, where $\{k,l,m,r,s\}=[5]$. It follows that $I_k$ must contain at least two pairs containing $l$ (since if it does not contain one such pair it must contain the other two). This implies that $I_k$ contains at least four pairs, and this completes the proof that $h\geq 8$.

If $h=8$, by the previous argument, the two missing pairs in $I_k$ must contain the four elements distinct from $k$ exactly once, and so the weight of $\de_k d_{1k}\cdots \hat d_{kk}\cdots d_{5k}d_{I_k}$ is $(0,0,0,0)$.

\end{proof}

We can now tackle the case of singular vectors of height 8 and degree 10.
\begin{proposition}\label{h8d10}
There are no singular vectors $M(V)$ of height 8 and degree 10.
\end{proposition}
\begin{proof}
	Assume by contradiction that $w$ is a singular vector of height 8 and degree 10. For distinct $i,j,k,l\in [5]$, with $i<j$ and $k<l$ we let $d^\vee_{jk,lm}=d_{\Omega\setminus\{jk,lm\}}$. For example $d^\vee_{14,25}=d_{12}d_{13}d_{15}d_{23}d_{24}d_{34}d_{35}d_{45}$. By Proposition \ref{hgeq8} we have that $w_8$ can be expressed in the following way
	\[
	w_8=\sum_{i,j,k,l,m}\de_id^\vee_{jk,lm}\otimes v_{i,jk,lm}
	\]
	where the sum is over all distinct $i,j,k,l,m \in [5]$ such that $j<k,l$, and  $l<m$ (so we have exactly 15 summands).
	We also adopt the convention $v_{i,lm,jk}=v_{i,jk,lm}$ for notational convenience. By construction, we immediately have $(x_i\, d_{ik})^1w_8=0$ for all $i\neq k$. We will therefore consider elements in $\g_1$ of the form $x_i d_{jk}+x_j d_{ik}$ and $x_i d_{jk}- x_k d_{ij}$ (for all $i<j<k$), which will allow us to deduce that $w_8=0$. To perform this computation efficiently we need the following notation.
	
	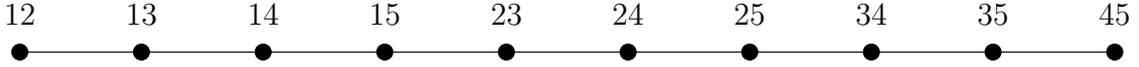
\begin{figure}[h]
		\begin{center}
			$$
			\begin{tikzpicture}
			\draw(-7.2,0)--(7.2,0);

			\draw[fill=black]{(-7.2,0) circle(3pt)};
			\draw[fill=black]{(-5.6,0) circle(3pt)};
			\draw[fill=black]{(-4,0) circle(3pt)};
			\draw[fill=black]{(-2.4,0) circle(3pt)};
			\draw[fill=black]{(-0.8,0) circle(3pt)};
			\draw[fill=black]{(0.8,0) circle(3pt)};
			\draw[fill=black]{(2.4,0) circle(3pt)};
			\draw[fill=black]{(4,0) circle(3pt)};
			\draw[fill=black]{(5.6,0) circle(3pt)};
			\draw[fill=black]{(7.2,0) circle(3pt)};

			\node at (-7.2,0.5) {$12$};
			\node at (-5.6,0.5) {$13$};
			\node at (-4,0.5) {$14$};
			\node at (-2.4,0.5) {$15$};
			\node at (-0.8,0.5) {$23$};
			\node at (0.8,0.5) {$24$};
			\node at (2.4,0.5) {$25$};
			\node at (4,0.5) {$34$};
			\node at (5.6,0.5) {$35$};
			\node at (7.2,0.5) {$45$};
			
		\end{tikzpicture}$$
		\end{center}
		\caption{\label{figlex}The lexicographic order}  
		\end{figure}

\medskip
	
	Let \[\eta_{ij}=\begin{cases}1& \textrm{if } i+j=5 \\0&\textrm{otherwise}.\end{cases}\] The reason for introducing this function is the following: let $d(ij,kl)$ be the distance between the pairs $ij$ and $kl$ in the lexicographic order (i.e. in the graph represented in Figure \ref{figlex}; then one can easily check that for all $i<j<k$ we have
	
	 \begin{equation}(-1)^{d(ik,jk)}=(-1)^{\eta_{ij}+1}\label{dist1}
	 \end{equation}
	 and
	 \begin{equation}\label{dist2}
	 (-1)^{d(ij,jk)}=(-1)^{\eta_{ij}+k+j+1}.
	 \end{equation}
	 
	Let $i,j,k,l,m$ be distinct such that $i<j<k$ and $l<m$. We have
	\[
	(x_id_{jk}+x_j d_{ik})^1w_8\equiv -d_{jk}d^\vee_{jk,lm}\otimes v_{i,jk,lm}-d_{ik}d^\vee_{ik,lm}\otimes v_{j,ik,lm} \mod M_8(V)
	\]
	By \eqref{permform} and \eqref{dist1} we have
	\begin{equation}\label{gb1}
	v_{i,jk,lm}=\begin{cases}
	(-1)^{\eta_{ij}+1}v_{j,ik,lm} &\textrm{if }i<l<j\\
	(-1)^{\eta_{ij}}v_{j,ik,lm}& \textrm{otherwise.}
	\end{cases}
	\end{equation}
	
	Similarly, applying $x_i d_{jk}-x_k d_{ij}$ we obtain
	\[
	(x_id_{jk}-x_k d_{ij})^1w_8\equiv -d_{jk}d^\vee_{jk,lm}\otimes v_{i,jk,lm}+d_{ij}d^\vee_{ij,lm}\otimes v_{k,ij,lm} \mod M_8(V)
\]
	
	and by \eqref{permform} and \eqref{dist2} we have
	\begin{equation}\label{gb2}
	v_{i,jk,lm}=\begin{cases}
	(-1)^{\eta_{ij}+k+j}v_{k,ij,lm} &\textrm{if }i<l<j\\
	(-1)^{\eta_{ij}+k+j+1}v_{k,ij,lm}& \textrm{otherwise.}
	\end{cases}
		\end{equation}
By repeated application of Eq.\ \eqref{gb1} we obtain
\begin{itemize}
\item $v_{1,23,45}=v_{2,13,45}=v_{4,15,23}$
\item $v_{1,24,35}=v_{2,14,35}=-v_{3,15,24}$
\item $v_{1,25,34}=v_{2,15,34}=- v_{3,14,25}$
\item $v_{2,13,45}=v_{4,13,25}$
\item $v_{2,14,35}=-v_{3,14,25}$
\item $v_{2,15,34}=-v_{3,15,24}$
\item $v_{3,12,45}=v_{4,12,35}$
\end{itemize}
and by repeated application of Eq. \eqref{gb2} we obtain
\begin{itemize}
\item $v_{1,23,45}=v_{3,12,45}=v_{5,14,23}$

\item $v_{1,24,35}=-v_{4,12,35}=v_{5,13,24}$

\item $v_{1,25,34}=v_{5,12,34}=-v_{4,13,25}$

\item $v_{2,13,45}=v_{5,13,24}$

\item $v_{2,14,35}=v_{5,14,23}$

\item $v_{2,15,34}= -v_{4,15,23}$

\item $v_{3,12,45}=v_{5,12,34}$.

\end{itemize}
All these equations together imply that all $v_{i,jk,lm}$ vanish.
\end{proof}	

We now consider the case of a singular vector $w$ of  height 9 and degree 11. In this case, as in the proof of Proposition \ref{hgeq8}, we can immediately deduce that $w_9$ must have the following form
\begin{equation}\label{w9}
w_9=\sum_{i,j,k}\de_i d^\vee_{jk}\otimes v_{i,jk},
\end{equation}
where the sum is over all distinct $i,j,k$ with $j<k$ (a total of 30 summands) and $d^\vee_{jk}=d_{\Omega\setminus\{{j,k}\}}$.
As in the case of height 8, we can now proceed by applying all elements in $\g_1$ of the form $x_i d_{jk}+x_j d_{ik}$ and $x_i d_{jk}- x_k d_{ij}$.
\begin{lemma}\label{lemh=9}
	If $w$ is a singular vector of height 9 and degree 11 with highest term $w_9$ as in \eqref{w9}, then
	\begin{itemize}
		\item $v_{1,23}=v_{2,13}=v_{3,12}$;
		\item $v_{1,24}=v_{2,14}=-v_{4,12}$;
		\item $v_{1,25}=v_{2,15}=v_{5,12}$;
		\item $v_{1,34}=v_{3,14}=v_{4,13}$;
		\item $v_{1,35}=v_{3,15}=-v_{5,13}$;
		\item $v_{1,45}=-v_{4,15}=-v_{5,14}$;
		\item $v_{2,34}=-v_{3,24}=-v_{4,23}$;
		\item $v_{2,35}=-v_{3,25}=v_{5,23}$;
		\item $v_{2,45}=v_{4,25}=v_{5,24}$;
		\item $v_{3,45}=v_{4,35}=v_{5,34}$.
	\end{itemize}
\end{lemma}
\begin{proof}
	All equalities are obtained using \eqref{permform} and \eqref{1actbis} applying elements $x_i d_{jk}+x_j d_{ik}$ and $x_i d_{jk}- x_k d_{ij}$. For example we have
	\[(x_2 d_{35}+x_3 d_{25}) w_9\equiv -d_{35} d_{35}^\vee\otimes v_{2,35}-d_{25} d_{25}^\vee\otimes v_{3,25}\equiv d_\Omega(-v_{2,35}-v_{3,25}) \mod M_8(V),\]
	hence $v_{2,35}=-v_{3,25}$. All other equalities can be obtained similarly.
	\end{proof}
Thanks to Lemma \ref{lemh=9} the highest term of the singular vector assumes the following form:
\begin{align}\label{w9reduced}
w_9& = \de_1(d_{23}^\vee\otimes u_1-d_{24}^\vee\otimes u_2+d_{25}^\vee\otimes u_3-d^\vee_{34}\otimes u_4+d^\vee_{35}\otimes u_5-d_{45}^\vee \otimes u_6)\\
\nonumber & \hspace{3mm} +\de_2(d_{13}^\vee\otimes u_1-d_{14}^\vee\otimes u_2+d_{15}^\vee\otimes u_3-d^\vee_{34}\otimes u_7+d_{35}^\vee \otimes u_8-d_{45}^\vee\otimes u_9)\\
\nonumber& \hspace{3mm} +\de_3(d_{12}^\vee\otimes u_1-d_{14}^\vee\otimes u_4+d_{15}^\vee\otimes u_5+d^\vee_{24}\otimes u_7-d_{25}^\vee \otimes u_8-d_{45}^\vee\otimes u_{10})\\
\nonumber& \hspace{3mm} +\de_4(d_{12}^\vee\otimes u_2-d_{13}^\vee\otimes u_4+d_{15}^\vee\otimes u_6+d^\vee_{23}\otimes u_7-d_{25}^\vee \otimes u_9-d_{35}^\vee\otimes u_{10})\\
\nonumber& \hspace{3mm} +\de_5(d_{12}^\vee\otimes u_3-d_{13}^\vee\otimes u_5+d_{14}^\vee\otimes u_6+d^\vee_{23}\otimes u_8-d_{24}^\vee \otimes u_9-d_{34}^\vee\otimes u_{10}).
\end{align}
for suitable elements $u_1,\ldots,u_{10} \in V$.
\begin{lemma}\label{g0h9}Let $E_i=x_i\de_{i+1}\in \g_0$ and $w$ be a singular vector of height 9 and degree 11 with $w_9$ as in \eqref{w9reduced} above. Then 
	\begin{itemize}
		\item $E_1$ annihilates $u_1,u_2,u_3,u_4,u_5,u_6,u_{10}$, $E_1.u_7= u_4$, $E_1. u_8= u_5$, $E_1.u_9=u_6$.
		\item $E_2$ annihilates $u_1,u_2,u_3,u_6,u_7,u_8,u_{9}$, $E_1.u_4= u_2$, $E_1.u_5=u_3$, $E_1 u_{10}= u_9$.
		\item $E_3$ annihilates $u_1,u_3,u_4,u_5,u_7,u_8,u_{10}$, $E_1.u_2= u_1$, $E_1.u_6= u_5$, $E_1 u_9= u_8$.
		\item $E_4$ annihilates $u_1,u_2,u_4,u_6,u_7,u_9,u_{10}$, $E_1.u_3=u_2$, $E_1.u_5= u_4$, $E_1 u_8= u_7$.
\end{itemize}
\end{lemma}
\begin{proof}
Recall the definition of $E_1^0$ from \eqref{not1}. By \eqref{0act} we have
\begin{align*}
0&=E_1^0 w_9\\
&=\de_2(d_{34}^\vee\otimes u_4-d_{35}^\vee\otimes u_5+d_{45}^\vee\otimes u_6)+\de_3(-d_{24}^\vee\otimes u_4+ d_{25}^\vee \otimes u_5)\\
&\hspace{3mm}+\de_4(- d_{23}^\vee\otimes u_4+\de_4 d_{25}^\vee\otimes u_6)+\de_5(-d_{23}^\vee \otimes u_5+d_{24}^\vee \otimes u_6)\\
&\hspace{3mm}+\de_1(d_{23}^\vee\otimes E_1.u_1-d_{24}^\vee\otimes E_1.u_2+d_{25}^\vee\otimes E_1.u_3-d^\vee_{34}\otimes E_1.u_4+d^\vee_{35}\otimes E_1.u_5-d_{45}^\vee \otimes E_1.u_6)\\
& \hspace{3mm} +\de_2(d_{13}^\vee\otimes E_1.u_1-d_{14}^\vee\otimes E_1.u_2+d_{15}^\vee\otimes E_1.u_3-d^\vee_{34}\otimes E_1.u_7+d_{35}^\vee \otimes E_1.u_8-d_{45}^\vee\otimes E_1.u_9)\\
& \hspace{3mm} +\de_3(d_{12}^\vee\otimes E_1.u_1-d_{14}^\vee\otimes E_1.u_4+d_{15}^\vee\otimes E_1.u_5+d^\vee_{24}\otimes E_1.u_7-d_{25}^\vee \otimes E_1.u_8-d_{45}^\vee\otimes E_1.u_{10})\\
& \hspace{3mm} +\de_4(d_{12}^\vee\otimes E_1.u_2-d_{13}^\vee\otimes E_1.u_4+d_{15}^\vee\otimes E_1.u_6+d^\vee_{23}\otimes E_1.u_7-d_{25}^\vee \otimes e_1.u_9-d_{35}^\vee\otimes E_1.u_{10})\\
& \hspace{3mm} +\de_5(d_{12}^\vee\otimes E_1.u_3-d_{13}^\vee\otimes E_1.u_5+d_{14}^\vee\otimes E_1.u_6+d^\vee_{23}\otimes E_1.u_8-d_{24}^\vee \otimes E_1.u_9-d_{34}^\vee\otimes E_1.u_{10}).
\end{align*}

The result for $E_1$ follows. The other statements are obtained similarly.
\end{proof}

Lemma \ref{g0h9} is depicted in Figure \ref{Weightsh=9}, where an arrow from $u_i$ to $u_k$ labelled $E_j$ means $E_j.u_i=u_k$ and the absence of an arrow labelled $E_j$ coming out from $u_i$ means $E_j.u_i=0$.

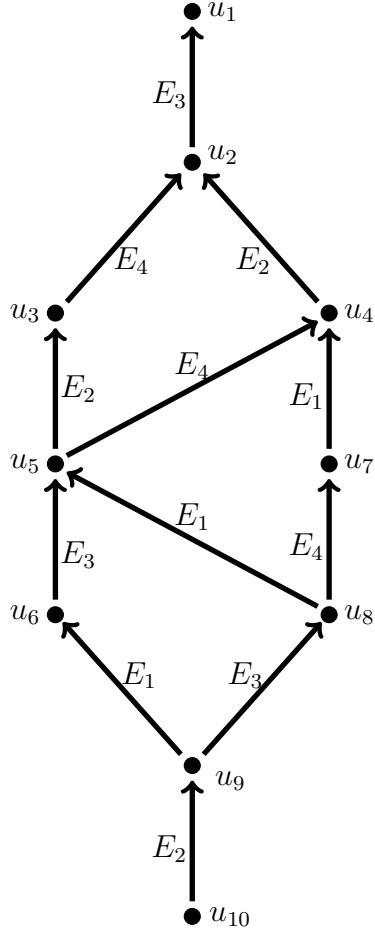
\begin{figure}[h]
	\begin{center}
		$$
		\begin{tikzpicture}
		\draw[->, line width=2pt](0,0.2)--(0,1.8);
		\draw[->, line width=2pt](0.15,2.15)--(1.7,3.9);
		\draw[->, line width=2pt](-0.15,2.15)--(-1.7,3.9);
		\draw[->, line width=2pt](-1.8,4.2)--(-1.8,5.8);
		\draw[->, line width=2pt](1.8,4.2)--(1.8,5.8);
		\draw[->, line width=2pt](-1.8,6.2)--(-1.8,7.8);
		\draw[->, line width=2pt](-1.65,8.15)--(-0.15,9.85);
		\draw[->, line width=2pt](0,10.2)--(0,11.8);
		\draw[->, line width=2pt](1.65,4.11)--(-1.65,5.89);
		\draw[->, line width=2pt](1.8,6.2)--(1.8,7.8);
		\draw[->, line width=2pt](1.65,8.15)--(0.15, 9.85);
		\draw[->, line width=2pt](-1.65,6.11)--(1.65,7.89);
		
		\draw[fill=black]{(0,0) circle(3pt)};
		\draw[fill=black]{(0,2) circle(3pt)};
		\draw[fill=black]{(-1.8,4) circle(3pt)};
		\draw[fill=black]{(1.8,4) circle(3pt)};
		\draw[fill=black]{(-1.8,6) circle(3pt)};
		\draw[fill=black]{(1.8,6) circle(3pt)};
		\draw[fill=black]{(-1.8,8) circle(3pt)};
		\draw[fill=black]{(1.8,8) circle(3pt)};
		\draw[fill=black]{(0,10) circle(3pt)};
		\draw[fill=black]{(0,12) circle(3pt)};
		
		\node at (0.5,0) {$u_{10}$};
		\node at (0.5,1.8) {$u_{9}$};
		\node at (2.2,4) {$u_{8}$};
		\node at (2.2,6) {$u_{7}$};
		\node at (2.2,8) {$u_{4}$};
		\node at (0.4,10.1) {$u_{2}$};
		\node at (0.4,12) {$u_{1}$};
		\node at (-2.2,8) {$u_{3}$};
		\node at (-2.2,6) {$u_{5}$};
		\node at (-2.2,4) {$u_{6}$};
		\node at (-0.3,0.9) {$E_2$};
		\node at (-0.7,3.2) {$E_{1}$};
		\node at (0.7,3.2) {$E_{3}$};
		\node at (-1.5,4.8) {$E_{3}$};
		\node at (0,5.3) {$E_{1}$};
		\node at (1.5,4.9) {$E_{4}$};
		\node at (0,7.3) {$E_{4}$};
		\node at (-1.5,7) {$E_{2}$};
		\node at (1.5,6.9) {$E_{1}$};
		\node at (0.8,8.7) {$E_{2}$};
		\node at (-0.8,8.7) {$E_{4}$};
		\node at (-0.3,10.9) {$E_3$};
	\end{tikzpicture}$$
	\end{center}
	\caption{\label{Weightsh=9} The action of the $E_i$'s on the elements $u_j$'s. } 
	\end{figure}
\begin{proposition}\label{h93cases}
Let $w\in M(V)$ be a singular vector of height 9 and degree 11 and let $w_9$ be as in \eqref{w9reduced}. Then one of the following applies:
\begin{enumerate}
	\item $u_1$ is a highest weight vector, $V=F(0,0,1,0)$ and $\lambda(w)=(0,0,0,0)$;
	\item $u_1=\cdots=u_6=0$, $u_7$ is a highest weight vector, $V=F(0,0,0,1)$ and $\lambda(w)=(1,0,0,0)$. 
	\item $u_1=u_2=\cdots=u_9=0$, $u_{10}$ is a highest weight vector, $V=F(0,0,0,0)$ and $\lambda(w)=(0,1,0,0)$. 
\end{enumerate}
\end{proposition}
\begin{proof}
If $u_1\neq 0$ then it is a highest weight vector in $V$ by Lemma \ref{g0h9}, and by Corollary \ref{d=h,h+4} we necessarily have $\lambda(w)=(0,0,0,0)$ and hence $\lambda(u_1)=(0,0,1,0)$.
If $u_1=0$ and $u_2\neq 0$, then $u_2$ is a highest weight vector by Lemma \ref{g0h9}, and by Corollary \ref{d=h,h+4} we necessarily have $\lambda(w)=(0,0,0,0)$ and hence $\lambda(u_2)=(0,1,-1,1)$, which is impossible since it is not a dominant weight.
Similarly, if $u_1,u_2=0$ and $u_3\neq 0$ then $u_3$ would be a highest weight vector of weight $(0,1,0,-1)$.

If $u_1=u_2=u_3=0$ and $u_4\neq 0$ then $u_4$ would be a highest weight vector of  weight $(1,-1,0,1)$.

If $u_1=u_2=u_3=u_4=0$ then then $u_5$ would be a highest weight vector of  weight $(1,-1,1,-1)$.

If $u_1=\cdots=u_5=0$ then $u_6$ would be a highest weight vector of  weight $(1,0,-1,0)$.

If $u_1=\cdots =u_6=0$ and $u_7\neq 0$ then by Corollary \ref{d=h,h+4} we have $\lambda(w)=(1,0,0,0)$ and so $\lambda(u_7)=(0,0,0,1)$.

If $u_1=\cdots =u_7=0$ and $u_8\neq 0$, then $\lambda(w)=(1,0,0,0)$ by Corollary \ref{d=h,h+4} and hence $\lambda(u_8)=(0,0,1,-1)$.

If $u_1=\cdots =u_8=0$ and $u_9\neq 0$, then $\lambda(w)=(1,0,0,0)$ by Corollary \ref{d=h,h+4} and hence $\lambda(u_9)=(0,1,-1,0)$.

Finally, if $u_1=\cdots =u_9=0$ then $u_{10}\neq 0$ is a highest weight vector, $\lambda(w)=(0,1,0,0)$ by Corollary \ref{d=h,h+4} and so $\lambda (u_{10})=(0,0,0,0)$.

\end{proof}
Now we show that cases (1) and (3) in Proposition \ref{h93cases} can not occur. Observe that by Theorem \ref{duality} it is enough to show that case (3) does not occur.
\begin{lemma}\label{h9case3}
There are no singular vectors as in Proposition \ref{h93cases} (3).
\end{lemma}
\begin{proof}
	In this case we have:
	\[
	w_9=(\de_3 d_{45}^\vee+\de_4 d_{35}^\vee+\de_5 d_{34}^\vee)\otimes u,
	\]
	where $u$ is a generator of the trivial $\g_0$-module. By construction $w_9$ satisfies \eqref{1act} for all $X\in \g_1$ and \eqref{0act} for all $i$. We will therefore take into account also \eqref{-1act} and \eqref{-2act} showing that there exists no $w_7$ which satisfies these equations.
	We start computing 
	\begin{align*}
	(x_5 d_{45})^{-1}w_9&\equiv (-\de_1 d_{23,34}^\vee-\de_2 d_{13,34}^\vee+\de_3 d_{13,24}^\vee+\de_4 d_{13,23}^\vee)\otimes u \mod M_7(V).
	\end{align*}
	
	We have 
	\begin{equation}\label{w7}
	w_7=\sum_{i\leq j}\de_i \de_j \sum_{p_1\prec p_2\prec p_3} d^\vee_{p_1,p_2,p_3} \otimes v_{i,j,\{p_1,p_2,p_3\}},
	\end{equation}
	and
	\[
	(x_5 d_{45})^1w_7\equiv  \sum_i (1+\delta_{i,5})\de_i \sum_{p_1\prec p_2 \prec \{4,5\}}d^\vee_{p_1,p_2}d_{45}\otimes v_{i,5,\{p_1,p_2,45\}} \mod M_7(V).
	\]
	We deduce in particular that $v_{4,5,\{13,23,45\}}=-u\neq 0$ by \eqref{-1act}.
	
	Next observe that
	$x_4 \de_5 w_9= 0$
	and 
	\[(x_4 \de_5)_0 \de_4 \de_5 d_{13,23,45}^\vee\otimes u=-\de_5^2 d_{13,23,45}^\vee\otimes u \mod M_6(V),
	\]
	and no other term of $w_7$ in \eqref{w7} can "produce" a summand $\de_5^2 d_{13,23,45}^\vee \otimes u$ by applying $x_4 \de_5$. This would imply $v_{4,5,\{13,23,45\}}=0$ by \eqref{-2act}, a contradiction. 
	\end{proof}
Case (2) in Proposition \ref{h93cases} leads us to the following surprising discovery.
\begin{theorem}\label{w[11]}
The following vector is a (unique up to multiplication by a scalar) singular vector in $M(0,0,0,1)$ of degree 11, height 9, and weight $(1,0,0,0)$:
\begin{align*}
w[11]&=d_{12}d_{13}d_{14}d_{15}\Big(
-\de_2 d_{23}d_{24}d_{25}d_{35}d_{45}\otimes x^*_5
-\de_2 d_{23}d_{24}d_{25}d_{34}d_{45}\otimes x^*_4
\\
& -\de_2 d_{23}d_{24}d_{25}d_{34}d_{35}\otimes x^*_3
+\de_3d_{23}d_{25}d_{34}d_{35}d_{45}\otimes x^*_5
+\de_3 d_{23}d_{24}d_{34}d_{35}d_{45}\otimes x^*_4\\
&+\de_3 d_{23}d_{24}d_{25}d_{34}d_{35}\otimes x^*_2
+\de_4 d_{24}d_{25}d_{34}d_{35}d_{45}\otimes x^*_5
-\de_4 d_{23}d_{24}d_{34}d_{35}d_{45}\otimes x^*_3\\
&+\de_4 d_{23}d_{24}d_{25}d_{34}d_{45}\otimes x^*_2
-\de_5 d_{24}d_{25}d_{34}d_{35}d_{45}\otimes x^*_4
-\de_5 d_{23}d_{25}d_{34}d_{35}d_{45}\otimes x^*_3\\
&+\de_5 d_{23}d_{24}d_{25}d_{35}d_{45}\otimes x^*_2
-\de_1 \de_2 d_{23} d_{24} d_{25} \otimes x^*_2
+\de_2^2 d_{23} d_{24} d_{25} \otimes x^*_1
+\de_1 \de_3 d_{23} d_{25} d_{34}\otimes x^*_2\\
&-\de_2 \de_3 d_{23} d_{25} d_{34}\otimes x^*_1
+\de_1 \de_4 d_{24} d_{25} d_{34}\otimes x^*_2
-\de_2 \de_4 d_{24} d_{25} d_{34}\otimes x^*_1
-\de_1 \de_3 d_{23} d_{24} d_{35} \otimes x^*_2\\
&+\de_2 \de_3 d_{23} d_{24} d_{35} \otimes x^*_1
+\de_1\de_5 d_{24} d_{25} d_{35} \otimes x^*_2
-\de_2 \de_5  d_{24} d_{25} d_{35} \otimes x^*_1
-\de_1 \de_3  d_{23} d_{34}d_{35} \otimes x^*_3\\
&+\de_3^2 d_{23} d_{34}d_{35} \otimes x^*_1
-\de_1 \de_4  d_{24} d_{34}d_{35} \otimes x^*_3
+\de_3 \de_4 d_{24} d_{34}d_{35} \otimes x^*_1
-\de_1 \de_5  d_{25} d_{34} d_{35} \otimes x^*_3\\
&+\de_3 \de_5 d_{25} d_{34}d_{35} \otimes x^*_1
-\de_1 \de_4  d_{23} d_{24} d_{45} \otimes x^*_2
+\de_2 \de_4 d_{23} d_{24} d_{45} \otimes x^*_1
-\de_1 \de_5  d_{23} d_{25} d_{45} \otimes x^*_2\\
&+\de_2 \de_5 d_{23} d_{25} d_{45} \otimes x^*_1
-\de_1 \de_3  d_{23} d_{34}d_{45} \otimes x^*_4
+\de_3 \de_4 d_{23} d_{34}d_{45} \otimes x^*_1
-\de_1 \de_4  d_{24} d_{34}d_{45} \otimes x^*_4\\
&+\de_4^2 d_{24} d_{34}d_{45} \otimes x^*_1
-\de_1 \de_5  d_{25} d_{34}d_{45} \otimes x^*_4
+\de_4 \de_5  d_{25} d_{34}d_{45} \otimes x^*_1
-\de_1 \de_3  d_{23} d_{35} d_{45} \otimes x^*_5\\
&+\de_3 \de_5 d_{23} d_{35} d_{45} \otimes x^*_1
-\de_1 \de_4  d_{24} d_{35} d_{45} \otimes x^*_5
+\de_4 \de_5  d_{24} d_{35} d_{45} \otimes x^*_1
-\de_1 \de_5  d_{25} d_{35} d_{45} \otimes x^*_5\\
&+\de_5^2 d_{25} d_{35} d_{45} \otimes x^*_1
-\de_1^2 \de_3 d_{23} \otimes x^*_2
+\de_1 \de_2 \de_3  d_{23} \otimes x^*_1
-\de_1^2 \de_4  d_{24} \otimes x^*_2\\
&+\de_1 \de_2 \de_4 d_{24} \otimes x^*_1
-\de_1^2 \de_5  d_{25} \otimes x^*_2
+\de_1 \de_2 \de_5 d_{25} \otimes x^*_1\Big).
\end{align*}
\end{theorem}
\begin{proof}
We prefer to omit the long but elementary computations that show that this is indeed a singular vector.  Its uniqueness follows from the fact that the term $w_9$ is determined up to  a scalar by Lemma \ref{g0h9} and Proposition \ref{h93cases}. So if $w'$ is another singular vector with $w_9=w'_9$ then $w-w'$ would be a singular vector of degree 11 with height at most 7 and this would contradict Proposition \ref{d!=h+4}.
\end{proof}

The last possible case with $d>h$ is ruled out by the following result.
\begin{proposition} There are no singular vectors of height 10 and degree 12.
\end{proposition}
\begin{proof} By hypothesis we are in one of the five cases in $(ii)$ of Corollary
\ref{d=h,h+4}.
Suppose there exists a singular vector $w$ of height 10 and degree 12  in $M(1,0,0,0)$ with weight $(0,0,0,0)$, i.e., by Theorem \ref{classS5},
$$w=\sum_{i=1}^5\de_id_{\Omega}\otimes x_i.$$
Then
there exists a morphism $\varphi$ of $E(5,10)$-modules,
$$\varphi: M(0,0,0,0)\rightarrow M(1,0,0,0).$$
By duality (see Theorem \ref{duality}), there exists a morphism
$$\varphi^*: M(0,0,0,1)\rightarrow M(0,0,0,0),$$
i.e., a singular vector $\bar{w}$ of degree 12,  of weight $(0,0,0,1)$ in $M(0,0,0,0)$. Then, by Theorem
\ref{classS5} and Proposition \ref{d!=h+4}, we have:
$\bar{w}_{10}=\de_5d_{\Omega}\otimes 1$ with $1$ the highest weight vector in $F(0,0,0,0)$. Let 
$$\bar{w}_8=\sum_{i,j,k,l}\alpha_{ij}\de_i\de_jd^\vee_{k5,\ell 5}\otimes 1+\sum_{i,j,k,l,t}\de_i\de_5
(\beta_{i,jk}d_{jk,\ell t}^\vee \otimes 1+\beta_{i,j\ell}d_{jl,kt}^\vee \otimes 1+\beta_{i,jt}d_{jt,k\ell}^\vee \otimes 1)$$
for some $\alpha_{ij}, \beta_{i,rs}\in\F$, where the first sum is over all
$\{i,j,k,l\}=[4]$ with $i<j$ and $k<l$, and the second sum is over all 
$\{i,j,k,l,t\}=[5]$ with $j<k<l<t$.
We apply condition \eqref{-1act} with $h=10$, using the following elements $X$ in $\g_1$:
\begin{itemize}
\item[i)] $X=x_5d_{45}$ hence getting $\beta_{2,45}=-1=\beta_{3,45}$;
\item[ii)] $X=x_5d_{13}+x_3d_{15}$ hence getting $\alpha_{23}=1=-\beta_{2,45}$;
\item[iii)] $X=x_5d_{12}-x_1d_{25}$ hence getting $\alpha_{13}=-1=\beta_{3,45}$;
\item[iv)] $X=x_1d_{25}+x_2d_{15}$ hence getting $\alpha_{13}=\alpha_{23}$.
\end{itemize}
These conditions lead to a contradiction, we therefore conclude that there is no singular vector of degree 12 and height 10 as in  R1 and R5, of Theorem \ref{classS5}.

Now assume that $w$ is a singular vector as in  R2, i.e., by Corollary
\ref{d=h,h+4},
$w=\sum_{i=2}^5\de_id_{\Omega}\otimes x_{1i}$, i.e., that there exists a morphism $\varphi$, of degree 12, of $E(5,10)$-modules:
$$\varphi: M(1,0,0,0) \rightarrow M(0,1,0,0).$$
By duality this means that there exists a morphism
$$\varphi^*: M(0,0,1,0) \rightarrow M(0,0,0,1)$$
of degree 12, i.e., a singular vector $\bar{w}$ of degree 12 and weight $(0,0,1,0)$  in $M(0,0,0,1)$. By Theorem \ref{classS5} and Proposition \ref{h8d10}, $\bar{w}$ is necessarily as in R4, with height 10, i.e.,
$$\bar{w}_{10}=\de_4d_{\Omega}\otimes x_5^*-\de_5d_{\Omega}\otimes x_4^*.$$
We have:
$$(x_5d_{45})^{-1}(\bar{w}_{10})=\de_4(d_{12}^\vee \otimes x_{3}^*+d_{13}^\vee\otimes  x_{2}^*+d_{23}^\vee \otimes x_{1}^*)
+(\de_3d_{12}^\vee+\de_2d_{13}^\vee +\de_1d_{23}^\vee)\otimes x_{4}^*.$$
By condition \eqref{-1act} with $h=10$ and $X=x_5d_{45}$ it follows that
in the expression of $\bar{w}_8$ the summand $\de_4\de_5d_{23,45}^\vee \otimes x_{1}^*$
must appear with coefficient equal to $1$. Now we have:
$$E_4(\de_4\de_5d_{23,45}^\vee \otimes x_{1}^*)= 
(E_4)^0(\de_4\de_5d_{23,45}^\vee\otimes  x_{1}^*)=-\de_5^2d_{23,45}^\vee \otimes x_{1}^*.$$
This contradicts condition \eqref{-2act}  for $h=10$. Indeed, one can see that no term in $E_4^{-2}\bar{w}_{10}+E_4^0\bar{w}_8$ can cancel  the summand 
$\de_5^2d_{23,45}^\vee \otimes x_{1}^*$.

Finally, let us assume that there exists a singular vector of degree 12 and height 10, as in R3, i.e., 
$$w_{10}=\de_3d_{\Omega}\otimes x_{45}^*+\de_4d_{\Omega}\otimes x_{53}^*+\de_5d_{\Omega}\otimes x_{34}^*.$$

Then we have:
\begin{align*}(x_5d_{45})^{-1}(w_{10})= 
&\de_3(d_{12}^\vee \otimes x_{34}^*+d_{13}^\vee \otimes x_{24}^*+d_{23}^\vee \otimes x_{14}^*)\\ &
+\de_4(-d_{13}^\vee \otimes x_{23}^*-d_{23}^\vee \otimes x_{13}^*
)-
(\de_3d_{12}^\vee +\de_2d_{13}^\vee+\de_1d_{23}^\vee)\otimes x_{34}^*.
\end{align*}
Therefore, similarly as above, by condition \eqref{-1act} with $h=10$ and $X=x_5d_{45}$, 
in the expression of ${w}_8$ the summand $\de_4\de_5d_{23,45}^\vee x_{13}^*$
must appear with coefficient $1$. Then we have:
$$E_4(\de_4\de_5d_{23,45}^\vee \otimes x_{13}^*)= 
(E_4)^0(\de_4\de_5d_{23,45}^\vee \otimes  x_{13}^*)=-\de_5^2d_{23,45}^\vee \otimes x_{13}^*.$$
This contradicts condition \eqref{-2act}  for $h=10$. Indeed, one can see that no term in $E_4^{-2}{w}_{10}+E_4^0{w}_8$ can cancel  the summand 
$\de_5^2d_{23,45}^\vee \otimes x_{13}^*$.
\end{proof}

\section{Properties of $\omega_I$}
In order to study morphisms between generalized Verma modules and to better understand their structure as $\g_0$-modules, a particular basis of $U_-$ has been introduced in \cite{CC}. The main goal of this section is to show that this basis is also extremely useful when considering the action of the whole $\g$ on a Verma module. We recall some technical notation needed to give an explicit definition of such a basis. We refer the reader to \cite[\S5]{CC} for further details.

We recall that $(U_-)_d$ denotes the homogeneous component of $U_-$ of degree $d$. 
We let 
\[\mathcal I_d=\{ I=(I_1,\ldots,I_d):\, I_l=(i_l,j_l) \textrm{ with  $ 1\leq i_l,j_l\leq 5$ for every $l=1,\ldots,d$}\}.
\]
If $I=(I_1,\ldots,I_d)\in \mathcal I_d$  we let $d_I=d_{I_1}\cdots d_{I_d}\in (U_-)_d$, with $d_{I_l}=d_{i_l j_l}$. Note that this notation is slightly different from the one adopted in Sections \ref{first} and \ref{>10}.

We let $\mathcal S_d$ be the set of subsets of $[d]$ of cardinality 2, so that $|\mathcal S_d|=\binom{d}{2}$.

Note that elements in $\mathcal I_d$ are ordered tuples of ordered pairs, while elements in $\mathcal S_d$ are unordered tuples of unordered pairs.

If $\{k,l\}\in \mathcal S_d$ and $I\in \mathcal I_d$ we let $t_{I_k,I_l}=t_{i_k,j_k, i_l,j_l}$ and $
\varepsilon_{I_k,I_l}=\varepsilon_{i_k,j_k, i_l,j_l}
$.

We also let
\[
D_{\{k,l\}}(I)=\frac{1}{2}  (-1)^{l+k}  \varepsilon_{I_k,I_l}\partial_{t_{I_k,I_l}}\in (U_-)_2.                                                                                           \]

\begin{definition}
	A subset $S$ of $\mathcal S_d$ is \emph{self-intersection free} if its elements are pairwise disjoint.\end{definition}
For example $S=\{\{1,3\},\{2,5\}, \{4,7\}\}$ is self-intersection free while $\{\{1,3\},\{2,5\}, \{3,7\}\}$ is not. We denote by $\textrm{SIF}_d$ the set of self-intersection free subsets of $\mathcal S_d$.
\begin{definition}
	Let $\{k,l\}, \{h,m\}\in \mathcal S_d$ be disjoint pairs. We say that $\{k,l\}$ and $\{h,m\}$ \emph{cross} if exactly one element in $\{k,l\}$ is between $h$ and $m$. If $S\in \textrm{SIF}_d$ we let the crossing number $c(S)$ of $S$ be the number of pairs of elements in $S$ that cross.
\end{definition}
\begin{definition}
	Let $S=\{S_1,\ldots,S_r\}\in \textrm{SIF}_d$. We let
	\[
	D_S(I)=\prod_{j=1}^rD_{S_j}(I)\in (U_-)_{2r}
	\]
	if $r\geq 2$ and $D_{\emptyset}(I)=1$ (note that the order of multiplication is irrelevant as the elements $D_{S_j}(I)$ commute among themselves).
\end{definition}
\begin{definition}
	For $I=(I_1,\ldots,I_d)\in \mathcal I_d$ and $S=\{S_1,\ldots,S_r\}\in \textrm{SIF}_d$ we let $C_S(I)\in \mathcal I_{d-2r}$ be obtained from $I$ by removing all $I_j$ such that $j\in S_k$ for some $k\in [r]$. 
\end{definition}

\begin{definition}\label{defomega}
	For all $I\in \mathcal I_d$ we let
	\[
	\omega_I=\sum_{S\in \textrm{SIF}_d}(-1)^{c(S)}D_S(I)\,d_{C_S(I)}\in (U_-)_d.
	\]
	
\end{definition}
If $I\in \mathcal I_d$ we let $x_I=x_{I_1}\wedge \cdots \wedge x_{I_d}\in \inlinewedge^d (\inlinewedge^2(\F^5))$.
The main properties of the elements $\omega_I$ have been obtained in {\cite[Proposition 5.6 and Theorem 5.8]{CC}} and can be summarized in the following result.
\begin{proposition}\label{isomega}Let $d=0,\ldots,10$. Then the map $\varphi:\inlinewedge^d (\inlinewedge^2(\F^5))\rightarrow U_-$, given by
	\[
	\varphi(x_{I_1}\wedge\cdots \wedge x_{I_d})=\omega_{I_1,\ldots, I_d}
	\]
	for all $(I_1,\ldots,I_d)\in \mathcal I_d$, is a (well-defined) injective morphism of $\g_0$-modules.
	
\end{proposition}
We will also need the following very useful notation.  Let  $I\in \mathcal I_d$ and $J\in \mathcal I_c$, with $c\leq d$.
If there exists $K\in \mathcal I_{d-c}$ such that $x_I=x_J\wedge x_K\neq 0$ we let $\omega_{I\setminus J}=\omega_K$, and we let $\omega_{I\setminus J}=0$ if such $K$ does not exist. Note that this notation is well-defined also thanks to Proposition \ref{isomega}. 

For example, in order to compute $\omega_{(12,24,35,54)\setminus (24,45)}$ we observe that $x_{12,24,35,54}=x_{24,45}\wedge x_{12,35}$, therefore  $\omega_{(12,24,35,54)\setminus (24,45)}=\omega_{12,35}$.

Instead of the explicit definition of the elements $\omega_I$ given in Definition \ref{defomega}, we will need some (equivalent) recursive properties that they satisfy.
\begin{lemma} Let $I=(I_1,\ldots,I_d)$. Then for all $k>1$ we have
	\[\sum_{S\in \textrm{SIF}_d:\,\{1,k\}\in S}(-1)^{c(S)}D_S(I)\,d_{C_S(I)}=-\frac{1}{2} \varepsilon_{I_1,I_k}\de_{t_{I_1,I_k}}\omega_{(I_2,\ldots,I_d)\setminus I_k}.
	\]
	
	Furthermore
	\begin{equation} \label{ricomega}\omega_I=d_{I_1}\omega_{(I_2,\ldots,I_d)}-\frac{1}{2}\sum_{k=2}^d \varepsilon_{I_1,I_k}\de_{t_{I_1,I_k}}\omega_{(I_2,\ldots,I_d)\setminus I_k}
	\end{equation}
\end{lemma}

\begin{proof}
	We prove the first statement for all $I\in \mathcal I_d$ by induction on $k$. If $k=2$ we have $D_{\{1,2\}}(I)=-\frac{1}{2} \varepsilon_{I_1,I_2}\de_{t_{I_1,I_2}}$ and so, letting $J=(J_1,\ldots,J_{d-2})=(I_3,\ldots,I_{d})$ we have
	
	\begin{align*}\sum_{S\in \textrm{SIF}_d:\,\{1,2\}\in S}(-1)^{c(S)}D_S(I)\,d_{C_S(I)}&=-\frac{1}{2} \varepsilon_{I_1,I_2}\de_{t_{I_1,I_2}} \sum_{S\in \textrm{SIF}_{d-2}}(-1)^{c(S)}D_S(J)d_{C_S(J)}\\
	&= -\frac{1}{2} \varepsilon_{I_1,I_2}\de_{t_{I_1,I_2}} \omega_{I_3,\ldots,I_d}\\
	&=-\frac{1}{2} \varepsilon_{I_1,I_2}\de_{t_{I_1,I_2}} \omega_{(I_2,\ldots,I_d)\setminus I_2}.
	\end{align*}
	If $k>2$ we let $J=(I_1,\ldots,I_{k-2},I_k,I_{k-1},I_{k+1},\ldots,I_d)$ be obtained from $I$ by swapping $I_k$ and $I_{k-1}$. We also observe that swapping $k$ with $k-1$ provides a bijection $S\mapsto S'$ between elements in $\textrm{SIF}_d$ containing $\{1,k\}$ and elements in $\textrm{SIF}_d$ containing $\{1,k-1\}$; we also observe that by this bijection we have $d_{C_S(I)}=d_{C_{S'}(J)}$ and $(-1)^{c(S)}D_S(I)=-(-1)^{c(S')}D_{S'}(J)$: indeed if there exists $l$ such that $\{k,l\}\in S$ then $(-1)^{c(S)}=-(-1)^{c(S')}$ and $D_S(I)=D_{S'}(J)$, and if such element $l$ does not exist then $(-1)^{c(S)}=(-1)^{c(S')}$ and $D_S(I)=-D_{S'}(J)$. Therefore, using the inductive hypothesis, we have
	\begin{align*}
	\sum_{S\in \textrm{SIF}_d:\,\{1,k\}\in S}(-1)^{c(S)}D_S(I)\,d_{C_S(I)}&= -\sum_{S'\in \textrm{SIF}_d:\,\{1,k-1\}\in S'}(-1)^{c(S')}D_{S'}(J)\,d_{C_S(J)}\\
	&=\frac{1}{2} \varepsilon_{J_1,J_{k-1}} \de_{t_{J_1,J_{k-1}}}\omega_{(J_2,\ldots,J_d)\setminus J_{k-1}}\\
	&=-\frac{1}{2} \varepsilon_{I_1,I_{k}}\de_{t_{I_1,I_{k}}}\omega_{(I_2,\ldots,I_d)\setminus I_{k}}.
	\end{align*}
	
	Equation \eqref{ricomega} now follows from the first part observing that the first summand in the right-hand side of \eqref{ricomega} is just
	
	\[
	\sum_{S\in \textrm{SIF}_d:\,\{1,k\}\notin S \,\forall k}(-1)^{c(S)}D_S(I)\,d_{C_S(I)}.
	\]
\end{proof}
The following result is probably the easiest way to handle and compute the elements $\omega_I$ in a recursive way.
\begin{proposition}\label{omegarec}Let $I=(I_1,\ldots,I_d)$. Then
	\[\omega_I=\frac{1}{d}\sum_{j=1}^d d_{I_j} \omega_{I\setminus I_j}.\]
\end{proposition}
\begin{proof}
	By \eqref{ricomega} and Proposition \ref{isomega} we have
	\begin{align*}
	\omega_I&=\frac{1}{d}\sum_{j=1}^d (-1)^{j+1} \omega_{I_j,I_1,\ldots,\hat I_j, \ldots,I_d}\\
	&=\frac{1}{d}\sum_{j=1}^d(-1)^{j+1} \big(d_{I_j}\omega_{ I_1,\ldots,\hat I_j,\ldots,I_d} -\frac{1}{2}\sum_{k\neq j} \varepsilon_{I_j,I_k} \de _{t_{I_j,I_k}}\omega_{(I_1,\ldots,\hat I_j,\ldots,I_d)\setminus I_k} \big)\\
	&=\frac{1}{d} \sum_{j=1}^d d_{I_j}\omega_{I\setminus I_j}-\frac{1}{2}\sum_{k\neq j}\varepsilon_{I_j,I_k} \de _{t_{I_j,I_k}}\omega_{I\setminus (I_j,I_k)}\\
	&=\frac{1}{d} \sum_{j=1}^d d_{I_j}\omega_{I\setminus I_j},
	\end{align*}
	since, clearly, $\omega_{I\setminus(I_j,I_k)}=-\omega_{I\setminus(I_k,I_j)}$ for all $k\neq j$.
	
\end{proof}
The following is an immediate consequence which is not needed in the sequel but sheds more light on the symmetric nature of the elements $\omega_I$'s.
\begin{corollary}
	We have
	\[\omega_{I_1,\ldots,I_d}=\frac{1}{d!}\sum_{\sigma \in S_d}\varepsilon_{\sigma} d_{I_{\sigma(1)}}\cdots d_{I_{\sigma(d)}}. \]
\end{corollary}
\begin{proof}
	We proceed by induction on $d$, the result being trivial for $d=1$. 
	
	We have
	\begin{align*}
	\frac{1}{d!}\sum_{\sigma \in S_d}\varepsilon_{\sigma} d_{I_{\sigma(1)}}\cdots d_{I_{\sigma(d)}}&= \frac{1}{d!}\sum_{j=1}^d\sum_{\sigma \in S([n]\setminus j)}(-1)^{j-1}\varepsilon_\sigma d_{I_j}d_{I_{\sigma(1)}}\cdots \hat d_{I_{\sigma(j)}} \cdots d_{I_{\sigma(d)}} \\
	&=\frac{1}{d}\sum_{j=1}^d (-1)^{j-1}d_{I_j}\sum _{\sigma \in S([n]\setminus j)}\varepsilon_\sigma d_{I_{\sigma(1)}}\cdots \hat d_{I_{\sigma(j)}} \cdots d_{I_{\sigma(d)}}\\
	&= \frac{1}{d}\sum_{j=1}^d (-1)^{j-1}d_{I_j} \omega_{I_1,\ldots,\hat I_j,\ldots,I_d}\\
	&=\frac{1}{d}\sum_{j=1}^d d_{I_j} \omega_{I\setminus I_j}.
	\end{align*}
\end{proof}

We now reformulate  \eqref{ricomega} in a way  which is more suitable for our next arguments.
\begin{lemma}\label{adgl}
Let $I\in \mathcal I_d$ and let $\{i,j,r,s,t\}= \{1,2,3,4,5\}$ be such that $\varepsilon_{ij,rs}=\varepsilon_{ij,st}=\varepsilon_{ij,tr}=1$. Then
	\[
	d_{ij}\omega_I=\omega_{ij,I}+\frac{1}{2}\de_r \omega_{I\setminus st}+\frac{1}{2}\de_s \omega_{I\setminus  tr}+\frac{1}{2}\de_t \omega_{I\setminus rs }.
	\]
\end{lemma}
Next target is to study the commutator between an element in $\g_1$ of the form $x_p d_{pq}$ and a generic element $\omega_I$. In order to simplify the reading of the arguments we prefer to show the proof explicitly in the special case $x_5d_{45}$.
\begin{lemma}\label{piec}
	We have
	\[
	\sum_{k\neq j}\big[[x_5d_{45},d_{I_k}],d_{I_j}\big]\omega_{I\setminus (I_k,I_j)}=\sum_j \big[[x_5d_{45},d_{I_j}],\omega_{I\setminus I_j}\big] -3 \de_4 \omega_{I\setminus(12,23,31)}+\frac{3}{2}\sum_{(\alpha,\beta,\gamma)\in S_3}\de_\alpha \omega_{I\setminus (\alpha\beta,\beta\gamma, \gamma 4)}.
	\]

\end{lemma}
\begin{proof} We first notice that in the left-hand side 
we have nonzero contributions  only for those $k$ such that $I_k=12,23,31$ (up to order). We compute the contribution of $I_k=12$, the others will be similar. We have
\[
\sum_j [x_5\de_3,d_{I_j}]\omega_{I\setminus(12,I_j)}=d_{25} \omega_{I\setminus(12,23)}+d_{51}\omega_{I\setminus(12,31)}+d_{54}\omega_{I\setminus(12,34)}
\]
and by Lemma \ref{adgl} and Proposition \ref{isomega} we have
\begin{align*}
\sum_j [&x_5\de_3,d_{I_j}]\omega_{I\setminus(12,I_j)}=\omega_{25,I\setminus(12,23)}+\omega_{51,I\setminus(12,31)}+\omega_{54,I\setminus(12,34)}\\
&+\frac{1}{2}\big(\de_1\omega_{I\setminus(12,23,34)}+\de_3 \omega_{I\setminus(12,23,41)}+\de_4\omega_{I\setminus(12,23,13)}+ \de_2 \omega_{I\setminus(12,31,34)}+\de_3\omega_{I\setminus(12,31,42)}+\de_4\omega_{I\setminus(12,31,23)}\\
&+ \de_1\omega_{I\setminus(12,34,32)}+\de_3\omega_{I\setminus(12,34,21)}+\de_2\omega_{I\setminus(12,34,13)}\big)\\
&= [x_5\de_3,\omega_{I\setminus 12}]-\de_4 \omega_{I\setminus(12,23,31)}+\de_1 \omega_{I\setminus(12,23,34)}+\de_2 \omega_{I\setminus(21,13,34)}+\frac{1}{2}\de_3 \omega_{I\setminus(32,21,14)}+\frac{1}{2}\de_3\omega_{I\setminus(31,12,24)}
\end{align*}
The contributions of $I_k=23,31$ are similarly computed and the result follows.
\end{proof}
\begin{theorem}
	For all $I\in \mathcal I_d$ we have
	
	\begin{align*}[x_5 d_{45},\omega_I]&=\sum_{j=1}^d \Big(\frac{1}{2} \big[[x_5d_{45},d_{I_j}],\omega_{I\setminus I_j}\big]+\omega_{I\setminus I_j} [x_5d_{45},d_{I_j}]\Big)\\
	&\hspace{5mm} +\frac{1}{2} \de_4 \omega_{I\setminus (12,23,31)}-\frac{1}{4}\sum_{(\alpha,\beta,\gamma) \in S_3} \de_{\alpha} \omega_{I\setminus (\alpha \beta, \beta \gamma, \gamma 4)}.
	\end{align*}
\end{theorem}
\begin{proof}
	We proceed by induction on $d$, the case $d=1$ being easy. 
	Note that by induction hypothesis we can assume that
	\begin{align*}
	[x_5d_{45},\omega_{I\setminus I_j}]&=\sum_{k\neq j} \Big(\frac{1}{2} \big[[x_5d_{45},d_{I_k}],\omega_{I\setminus (I_j,I_k)}\big]+\omega_{I\setminus (I_j,I_k)} [x_5d_{45},d_{I_k}]\Big)\\
	&\hspace{5mm} +\frac{1}{2} \de_4 \omega_{I\setminus (I_j,12,23,31)}-\frac{1}{4}\sum_{(\alpha,\beta,\gamma) \in S_3} \de_{\alpha} \omega_{I\setminus (I_j,\alpha \beta, \beta \gamma, \gamma 4)} .
	\end{align*}
	Using Proposition \ref{omegarec} and the inductive hypothesis we have:
	\begin{align*}
	[x_5 d_{45},\omega_I]&=\frac{1}{d}\sum_{j=1}^d \big([x_5d_{45},d_{I_j}]\omega_{I\setminus I_j}-d_{I_j}[x_5d_{45},\omega_{I\setminus I_j}]\big)\\
	&=\frac{1}{d}\sum_{j=1}^d \Big(\big[[x_5d_{45},d_{I_j}],\omega_{I\setminus I_j}\big]+\omega_{I\setminus I_j}[x_5d_{45},d_{I_j}]\\
	&\hspace{20mm} -d_{I_j}\sum_{k\neq j} \big(\frac{1}{2} \big[[x_5d_{45},d_{I_k}],\omega_{I\setminus (I_j,I_k)}\big]+\omega_{I\setminus (I_j,I_k)} [x_5d_{45},d_{I_k}]\big)\\
	&\hspace{20mm} +\frac{1}{2} \de_4 \omega_{I\setminus (I_j,12,23,31)}-\frac{1}{4}\sum_{(\alpha,\beta,\gamma) \in S_3} \de_{\alpha} \omega_{I\setminus (I_j,\alpha \beta, \beta \gamma, \gamma 4)} \Big)\\
	&=\frac{1}{d}\sum_{j=1}^d \big(\omega_{I\setminus I_j}[x_5d_{45},d_{I_j}]-d_{I_j} \sum_{k\neq j} \omega_{I\setminus (I_j,I_k)} [x_5d_{45},d_{I_k}]\big)\\
	& \hspace{5mm}+\frac{1}{d}\sum_{j=1}^d \big(\big[[x_5d_{45},d_{I_j}],\omega_{I\setminus I_j}\big]-d_{I_j}\sum_{k\neq j} \frac{1}{2} \big[[x_5d_{45},d_{I_k}],\omega_{I\setminus (I_j,I_k)}\big]\big)\\
	& \hspace{5mm}+ \frac{1}{d}\sum_{j=1}^dd_{I_j}\big(-\frac{1}{2} \de_4 \omega_{I\setminus (I_j,12,23,31)}+\frac{1}{4}\sum_{(\alpha,\beta,\gamma) \in S_3} \de_{\alpha} \omega_{I\setminus (I_j,\alpha \beta, \beta \gamma, \gamma 4)} \big)
	\end{align*}
	 
We split this formula into three parts (according to the last three lines above):
the first part is
\begin{align*}\frac{1}{d}&\sum_{j=1}^d \big(\omega_{I\setminus I_j}[x_5d_{45},d_{I_j}]-d_{I_j} \sum_{k\neq j} \omega_{I\setminus (I_j,I_k)} [x_5d_{45},d_{I_k}]\big)\\
&= \frac{1}{d} \big( \sum_{j=1}^d \omega_{I\setminus I_j}[x_5d_{45},d_{I_j}]+\sum_{k=1}^d \sum_{j\neq k} d_{I_j}\omega_{I\setminus(I_k,I_j)}[x_5d_{45},d_{I_k}]\big)\\
&=\frac{1}{d} \big( \sum_{j=1}^d \omega_{I\setminus I_j}[x_5d_{45},d_{I_j}]+\sum_{k=1}^d (d-1)\omega_{I\setminus I_k}[x_5d_{45},d_{I_k}]\big)\\
&= \sum_{j=1}^d \omega_{I\setminus I_j}[x_5d_{45},d_{I_j}].
\end{align*}

The third part is
\begin{align*}\frac{1}{d}&\sum_{j=1}^dd_{I_j}\big(-\frac{1}{2} \de_4 \omega_{I\setminus (I_j,12,23,31)}+\frac{1}{4}\sum_{(\alpha,\beta,\gamma) \in S_3} \de_{\alpha} \omega_{I\setminus (I_j,\alpha \beta, \beta \gamma, \gamma 4)} \big)\\
&=\frac{1}{2d}\de_4 \sum_{j=1}^d d_{I_j} \omega_{I\setminus (12,23,31,I_j)}-\frac{1}{4d}\sum_{(\alpha,\beta,\gamma) \in S_3} \de_{\alpha} \sum_{j=1}^d d_{I_j}\omega_{I\setminus (\alpha \beta, \beta \gamma, \gamma 4,I_j)} \\
&=\frac{d-3}{2d}\de_4 \omega_{I\setminus(12,23,31)}-\frac{d-3}{4d} \sum_{(\alpha,\beta,\gamma) \in S_3} \de_{\alpha}\omega_{I\setminus (\alpha \beta, \beta \gamma, \gamma 4)}.
\end{align*}
In order to compute the second part we notice, using Lemma \ref{piec}, that the following holds:
\begin{align*}
-d_{I_j}\sum_{k,j}& \big[[x_5d_{45},d_{I_k}],\omega_{I\setminus (I_j,I_k)}\big]=d_{I_j}\sum_{k, j} \big[[x_5d_{45},d_{I_k}],\omega_{I\setminus (I_k,I_j)}\big]\\
&=\sum_{j,k}\Big(\big[[x_5d_{45},d_{I_k}],d_{I_j}\omega_{I\setminus(I_k,I_j)}\big]-\big[[x_5d_{45},d_{I_k}],d_{I_j}\big]\omega_{I\setminus(I_k,I_j)}\Big)\\
&=\sum_k\big[[x_5d_{45},d_{I_k}],\sum_jd_{I_j}\omega_{I\setminus (I_k,I_j)}\big]-\sum_{j,k}\big[[x_5d_{45},d_{I_k}],d_{I_j}\big]\omega_{I\setminus(I_k,I_j)}\\
&=(d-1)\sum_k \big[[x_5d_{45},d_{I_k}],\omega_{I\setminus I_k}\big]-\sum_j \big[[x_5d_{45},d_{I_j}],\omega_{I\setminus I_j}\big] +3 \de_4 \omega_{I\setminus(12,23,31)}\\
&-\frac{3}{2}\sum_{(\alpha,\beta,\gamma)\in S_3}\de_\alpha \omega_{I\setminus (\alpha\beta,\beta\gamma, \gamma 4)}\\
&=(d-2)\sum_k \big[[x_5d_{45},d_{I_k}],\omega_{I\setminus I_k}\big]+3 \de_4 \omega_{I\setminus(12,23,31)}-\frac{3}{2}\sum_{(\alpha,\beta,\gamma)\in S_3}\de_\alpha \omega_{I\setminus (\alpha\beta,\beta\gamma, \gamma 4)}.
\end{align*}
Therefore the whole second part is
\begin{align*}\frac{1}{d}&\sum_{j=1}^d \big(\big[[x_5d_{45},d_{I_j}],\omega_{I\setminus I_j}\big]-d_{I_j}\sum_{k\neq j} \frac{1}{2} \big[[x_5d_{45},d_{I_k}],\omega_{I\setminus (I_j,I_k)}\big]\big)\\
&=\frac{1}{d}\Big(\sum_j \big[[x_5d_{45},d_{I_j}],\omega_{I\setminus I_j}\big]+\frac{1}{2} (d-2)\sum_k \big[[x_5d_{45},d_{I_k}],\omega_{I\setminus I_k}\big]+\frac{3}{2} \de_4 \omega_{I\setminus(12,23,31)}\\
&-\frac{3}{4}\sum_{(\alpha,\beta,\gamma)\in S_3}\de_\alpha \omega_{I\setminus (\alpha\beta,\beta\gamma, \gamma 4)}\Big)\\
&=\frac{1}{2} \sum_j \big[[x_5d_{45},d_{I_j}],\omega_{I\setminus I_j}+\frac{3}{2d} \de_4 \omega_{I\setminus(12,23,31)}-\frac{3}{4d}\sum_{(\alpha,\beta,\gamma)\in S_3}\de_\alpha \omega_{I\setminus (\alpha\beta,\beta\gamma, \gamma 4)}.
\end{align*}
The sum of the three parts gives the result.
	\end{proof}
One can analogously prove the following result. 
\begin{theorem}\label{xpdpqomega}
	Let $\{p,q,a,b,c\}=\{1,2,3,4,5\}$ and $I\in \mathcal I_d$. Then
	
	\begin{align*}[x_p d_{pq},\omega_I]&=\sum_{j=1}^d \Big(\frac{1}{2} \big[[x_pd_{pq},d_{I_j}],\omega_{I\setminus I_j}\big]+\omega_{I\setminus I_j} [x_pd_{pq},d_{I_j}]\Big)\\
	&\hspace{5mm} -\frac{1}{2} \de_q \omega_{I\setminus (ab,bc,ca)}+\frac{1}{4}\sum_{(\alpha,\beta,\gamma) \in S(a,b,c)} \de_{\alpha} \omega_{I\setminus (\alpha \beta, \beta \gamma, \gamma q)},
	\end{align*}
where $S(a,b,c)$ denotes the set of permutations of $\{a,b,c\}$.
\end{theorem}
\section{The fundamental equations}
We are now going to use Theorem \ref{xpdpqomega} to study possible morphisms
of finite Verma modules $\varphi:M(V)\rightarrow M(W)$. Let $\sim$ be the equivalence relation on $\mathcal I_d$ such that $I\sim I'$ if and only if $\omega_I=\pm \omega_{I'}$, i.e. if  $I'$ can be obtained from $I$ by permuting the pairs in $I$ and the elements in each pair. By Proposition \ref{morphisms}, Remark \ref{dual} and Proposition \ref{isomega} we know that if a morphism has degree $d$ then it can be expressed in the following way
\begin{equation}\label{varfi}
\varphi(v)=\sum_{l\leq d/2} \sum_{I\in \mathcal I_{d-2l}/\sim} \sum_{1\leq r_1\leq \cdots \leq r_l\leq 5} \de_{r_1}\cdots \de_{r_l} \omega_I \otimes \theta^{r_1,\ldots,r_l}_I(v)
\end{equation}
where the $\theta^{r_1,\ldots,r_l}_I:V\rightarrow W$ are such that the map
\[\Sym^l(\F^5)\otimes \displaywedge^{d-2l} (\displaywedge^2((\F^5)^*)\rightarrow \Hom(V,W)
\]
given by \begin{equation}\label{isotheta}x_{r_1}\cdots x_{r_l}\otimes  x_{I_1}^*\wedge \cdots \wedge x_{I_{d-2l}}^*\mapsto \theta^{r_1,\ldots,r_l}_{I_1,\ldots,I_{d-2l}}\end{equation}
is a (well-defined) morphism of $\g_0$-modules. This fact permits us to easily compute the action of $\g_0$ on the morphisms $\theta^{r_1,\ldots,r_l}_I$'s. For example we have 
\[x_1 \de_2.\theta^{2,3}_{12,13,14,23}=\theta^{1,3}_{12,13,14,23}
-\theta^{2,3}_{12,23,14,23}-\theta^{2,3}_{12,13,24,23}=\theta^{1,3}_{12,13,14,23}+\theta^{2,3}_{12,13,23,24}.\]
A technical lemma is in order.
\begin{lemma}\label{jkjk} For all distinct $\alpha, \beta, \gamma, p \in [5]$ we have:
\[
\sum_{J\in \mathcal I_d/\sim} [x_p\de_\gamma,\omega_J]\otimes \theta_{\alpha\beta, J}=-\sum_{J\in \mathcal I_d/\sim} \omega_J \otimes (x_p \de_\gamma.\theta_{\alpha \beta, J}).
\]
and
\[
\sum_{K\in \mathcal I_d/\sim }(-\de_{\gamma} \omega_K \otimes \theta^p_{\alpha \beta, K}+\sum_{t=1}^5 \de_t [x_p \de_\gamma, \omega_K]\otimes \theta^t_{\alpha\beta, K})=-\sum_{t=1}^5 \sum_{K\in \mathcal I_d/\sim} \de_t \omega_K\otimes (x_p \de_\gamma.\theta^t_{\alpha \beta,K}).
\]
\end{lemma}
\begin{proof}
The first equation follows from the following observation. By Proposition \ref{isomega} and \eqref{isotheta} we have that if
$
[x_p \de_\gamma,\omega_J]=\sum_{J'} a_{J,J'} \omega_{J'}
$
then
$
x_p \de_\gamma.\theta_{J'}=-\sum_J a_{J,J'}\theta_J
$
 and hence also
\[
x_p \de_\gamma.\theta_{\alpha \beta,J'}=-\sum_J a_{J,J'}\theta_{\alpha \beta,J},
\]
since $\alpha$ and $\beta$ are distinct from $p$.
We can conclude that

\[\sum_J [x_p \de_\gamma,\omega_J]\otimes \theta_{\alpha \beta,J}=\sum_{J'} \omega_{J'}\otimes \sum_J a_{J,J'}\theta_{\alpha \beta,J}=-\sum_{J'}\omega_{J'}\otimes (x_p\de_\gamma.\theta_{\alpha \beta, J'}).
\]	
In order to prove the second equation we proceed in a similar way. If $[x_p\de_\gamma, \omega_K]=\sum_{K'} a_{K,K'} \omega_{K'}$
we have $x_p \de_\gamma.\theta_{K'}=-\sum_K a_{K,K'} \theta_K$ and also $x_p \de_\gamma.\theta_{\alpha \beta,K'}=-\sum_K a_{K,K'} \theta_{\alpha \beta,K}$. Therefore, if $t\neq \gamma$ we have $x_p \de_{\gamma}.\theta^t_{\alpha \beta,K'}=-\sum_K a_{K,K'} \theta^t_{\alpha \beta,K}$ and for $t=\gamma$ we have
\[
(x_p \de_\gamma.\theta^{\gamma}_{\alpha \beta,K'})=\theta^p_{\alpha \beta,K'}-\sum_K a_{K,K'} \theta^{\gamma}_{\alpha \beta,K}.
\]
So we can compute
\begin{align*}\sum_{K}&(-\de_{\gamma} \omega_K \otimes \theta^p_{\alpha \beta, K}+\sum_t \de_t [x_p \de_\gamma, \omega_K]\otimes \theta^t_{\alpha\beta, K})\\
&=\sum_{K}(-\de_{\gamma} \omega_K \otimes \theta^p_{\alpha \beta, K}+\sum_t \de_t \sum_{K'} a_{K,K'}\omega_K' \otimes \theta^t_{\alpha\beta, K})\\
&=\sum _{K'}\de_{\gamma}\omega_{K'}\otimes (\theta^p_{\alpha \beta,K'}+\sum_K a_{K,K'} \theta^\gamma_{\alpha \beta,K})+\sum_{t\neq \gamma} \sum_{K'} \de_t \omega_{K'}\otimes \sum _K a_{K,K'}\theta^t_{\alpha \beta, K}\\
&= -\sum_t \sum_{K'} \de_t \omega_K' \otimes (x_p \de_\gamma.\theta^\gamma_{\alpha \beta,K'}).
\end{align*}
\end{proof}
We let $C(a,b,c)=\{(a,b,c), (b,c,a), (c,a,b)\}$ the set of cyclic permutations of $(a,b,c)$. From now on when we write $\sum_{\alpha \beta \gamma}$ we always mean the sum over $(\alpha,\beta,\gamma)\in C(a,b,c)$.
\begin{lemma}\label{xpdpqaomth1}
	Let $\varphi:M(V)\rightarrow M(W)$ be a morphism of Verma modules as in \eqref{varfi} and $(p,q,a,b,c)$ be any permutation of $[5]$. Then
\begin{align*}
x_pd_{pq} &\sum_{I\in \mathcal I_d/\sim}\omega_I\otimes \theta_I(v)=\sum_{J\in \mathcal I_{d-1}/\sim} \omega_J\otimes \frac{1}{2}\varepsilon_{pqabc} \sum_{\alpha\beta\gamma} \big(-(x_p\de_\gamma.\theta_{\alpha\beta,J})(v)+2x_p\de_\gamma.(\theta_{\alpha\beta,J}(v))\big)\\
\nonumber &\hspace{5mm}+\sum_{K\in \mathcal I_{d-3}/\sim}\de_q\omega_K\otimes -\frac{1}{2}\theta_{ab,bc,ca,K}(v)+\sum_{\alpha\beta\gamma}\de_\alpha \omega_K \otimes \frac{1}{4}\big(\theta_{\alpha\beta,\beta\gamma,\gamma q,K}(v)+\theta_{\alpha\gamma,\gamma \beta,\beta q,K}(v)\big).
\end{align*}

\end{lemma}
\begin{proof}Theorem \ref{xpdpqomega} can be reformulated in the following more convenient way
\begin{align}\label{xpdpqomegabis}
[x_pd_{pq},\omega_I]= & \frac{1}{2}\varepsilon_{pqabc}\sum_{(\alpha,\beta,\gamma)\in C(a,b,c)}\big([x_p\de_\gamma,\omega_{I\setminus \alpha\beta} ]+2\omega_{I\setminus \alpha \beta}\, x_p \de_\gamma\big)\\ & \nonumber -\frac{1}{2} \de_q \omega_{I\setminus(ab,bc,ca)}+\frac{1}{4} \sum_{(\alpha,\beta,\gamma)\in C(a,b,c)}\de_\alpha (\omega_{I\setminus(\alpha \beta, \beta \gamma, \gamma q)}+\omega_{I\setminus(\alpha \gamma,\gamma \beta , \beta q)}).
\end{align}

We can therefore compute
\begin{align*}
x_pd_{pq} &\sum_{I\in \mathcal I_d/\sim}\omega_I\otimes \theta_I(v)=
\frac{1}{2}\varepsilon_{pqabc} \sum_{I\in \mathcal I_d/\sim} \,\,\sum_{(\alpha,\beta,\gamma)\in C(a,b,c)}\Big([x_p\de_\gamma,\omega_{I\setminus \alpha \beta}]\otimes \theta_I(v)+2\omega_{I\setminus \alpha\beta}\otimes x_p \de_\gamma.(\theta_I(v)) \Big)\\
&+\sum_{I\in \mathcal I_d/\sim}\Big(-\frac{1}{2} \de_q \omega_{I\setminus(ab,bc,ca)}\otimes \theta_I(v)+\frac{1}{4}\sum_{(\alpha,\beta,\gamma)\in C(a,b,c)} \de_\alpha (\omega_{I\setminus(\alpha \beta, \beta \gamma, \gamma q)}+\omega_{I\setminus(\alpha \gamma,  \gamma \beta, \beta q)})\otimes \theta_I(v)\Big)\\
&= \frac{1}{2}\varepsilon_{pqabc} \sum_{J\in \mathcal I_{d-1}/\sim} \,\,\sum_{(\alpha,\beta,\gamma)\in C(a,b,c)}\Big([x_p\de_\gamma,\omega_J]\otimes \theta_{\alpha \beta,J}(v)+2\omega_{J}\otimes x_p \de_\gamma.(\theta_{\alpha \beta,J}(v)) \Big)\\
&+\sum_{K\in \mathcal I_{d-3}} \Big(-\frac{1}{2}\de_q \omega_{K}\otimes \theta_{ab,bc,ca,K}(v)+\frac{1}{4}\sum_{(\alpha,\beta,\gamma)\in C(a,b,c)} \de_\alpha \omega_K\otimes (\theta_{\alpha \beta, \beta \gamma, \gamma q, K}(v)+\theta_{\alpha \gamma, \gamma\beta , \beta q, K}(v))\Big)\\
&= \frac{1}{2}\varepsilon_{pqabc} \sum_{J\in \mathcal I_{d-1}/\sim} \,\,\sum_{(\alpha,\beta,\gamma)\in C(a,b,c)}\Big(-\omega_J\otimes (x_p\de_\gamma.\theta_{\alpha \beta,J})(v)+2\omega_{J}\otimes x_p \de_\gamma.(\theta_{\alpha \beta,J}(v)) \Big)\\
&+\sum_{K\in \mathcal I_{d-3}} \Big(-\frac{1}{2}\de_q \omega_{K}\otimes \theta_{ab,bc,ca,K}(v)+\frac{1}{4}\sum_{(\alpha,\beta,\gamma)\in C(a,b,c)} \de_\alpha \omega_K\otimes (\theta_{\alpha \beta, \beta \gamma, \gamma q, K}(v)+\theta_{\alpha \gamma, \gamma\beta , \beta q,K}(v))\Big),
\end{align*}
where we have used Lemma \ref{jkjk}.
\end{proof}
\begin{lemma}\label{xpdpqaomth2}Let $\varphi:M(V)\rightarrow M(W)$ be a morphism of Verma modules as in \eqref{varfi}. Then
\begin{align*}
x_pd_{pq} &\sum_{t=1}^5\sum_{I\in \mathcal I_{d-2}/\sim}\de_t\omega_I\otimes \theta^t_I(v)=\sum_{J\in \mathcal I_{d-1}/\sim}\omega_J\otimes -\theta^p_{J\setminus pq}(v)\\
\nonumber&+ \sum_{t=1}^5 \sum_{K\in \mathcal I_{d-3}/\sim} \de_t\omega_K\otimes \frac{1}{2}\varepsilon_{pqabc} \sum_{\alpha\beta\gamma} \big(-(x_p\de_\gamma.\theta^t_{\alpha\beta,K})(v)+2x_p\de_\gamma.(\theta^t_{\alpha\beta,K}(v))\big)\\
\nonumber & \hspace{5mm}+\sum_t\sum_{L\in \mathcal I_{d-5}/\sim}\de_t\de_q\omega_L\otimes -\frac{1}{2}\theta^t_{ab,bc,ca,L}(v)+\sum_{\alpha\beta\gamma}\de_t\de_\alpha \omega_L \otimes \frac{1}{4}\big(\theta^t_{\alpha\beta,\beta\gamma,\gamma q,L}(v)+\theta^t_{\alpha\gamma,\gamma \beta,\beta q,L}(v)\big)
\end{align*}
\end{lemma}
\begin{proof}
	By Lemma \ref{adgl}, Lemma \ref{jkjk} and \eqref{xpdpqomegabis} we have
	\begin{align*}
	x_pd_{pq} &\sum_{t=1}^5\sum_{I\in \mathcal I_{d-2}/\sim}\de_t\omega_I\otimes \theta^t_I(v)=\\
	\nonumber&=\sum_{I\in \mathcal I_{d-2}/\sim}(-\omega_{pq,I}-\frac{1}{2}\varepsilon_{pqabc} \sum_{\alpha\beta\gamma} \de_\gamma \omega_{I\setminus \alpha \beta})\otimes \theta^p_I(v)\\
	\nonumber&\hspace{5mm}+\sum_{t=1}^5 \sum_{I\in \mathcal I_{d-2}/\sim}  \frac{1}{2}\varepsilon_{pqabc} \de_t\sum_{\alpha\beta\gamma}\big([x_p \de_\gamma,\omega_{I\setminus \alpha \beta}]\otimes \theta_I^t(v)+2 \omega_{I\setminus \alpha \beta} \otimes x_p\de_\gamma.(\theta^t_{I}(v))\big)\\
	\nonumber & \hspace{5mm}+\sum_t\sum_{I\in \mathcal I_{d-2}/\sim}\Big(-\frac{1}{2}\de_t\de_q\omega_{I\setminus (ab,bc,ca)}\otimes \theta^t_{I}(v)+\frac{1}{4}\sum_{\alpha\beta\gamma}\de_t\de_\alpha \big(\omega_{I\setminus (\alpha \beta, \beta \gamma, \gamma q)}+\omega_{I\setminus (\alpha \gamma, \gamma \beta, \beta q)}) \otimes \theta_I^t(v)\Big)\\
	\nonumber&=\sum_{J\in \mathcal I_{d-1}/\sim}\omega_J\otimes -\theta^p_{J\setminus pq}(v) \\
	&\hspace{5mm} +\frac{1}{2}\varepsilon_{pqabc}\sum_{\alpha \beta \gamma}\sum_{K\in \mathcal I_{d-3}/\sim}\Big(-\de_\gamma \omega_K\otimes \theta^p_{\alpha \beta, K}+\sum_{t=1}^5 \de_t[x_p\de_\gamma, \omega_K]\otimes \theta^t_{\alpha \beta,K}\\
	&\hspace{10mm}+ \sum_{t=1}^5 2 \de_t \omega_K \otimes x_p\de_\gamma.(\theta^t_{\alpha \beta, K}(v))\Big)\\
	& \hspace{5mm}+\sum_t\sum_{L\in \mathcal I_{d-5}/\sim}\de_t\de_q\omega_L\otimes -\frac{1}{2}\theta^t_{ab,bc,ca,L}(v)+\sum_{\alpha\beta\gamma}\de_t\de_\alpha \omega_L \otimes \frac{1}{4}\big(\theta^t_{\alpha\beta,\beta\gamma,\gamma q,L}(v)+\theta^t_{\alpha\gamma,\gamma \beta,\beta q,L}(v)\big)\\
	&= \sum_{J\in \mathcal I_{d-1}/\sim}\omega_J\otimes -\theta^p_{J\setminus pq}(v)\\
	&\hspace{5mm} +\frac{1}{2}\varepsilon_{pqabc}\sum_{t=1}^5\sum_{K\in \mathcal I_{d-3}/\sim} \de_t \omega_K\otimes \sum_{\alpha \beta \gamma}-(x_p\de_\gamma.\theta^t_{\alpha \beta,K})(v)+2x_p \de_\gamma.(\theta^t_{\alpha \beta, K}(v))\\
	&\hspace{5mm}+\sum_t\sum_{L\in \mathcal I_{d-5}/\sim}\de_t\de_q\omega_L\otimes -\frac{1}{2}\theta^t_{ab,bc,ca,L}(v)+\sum_{\alpha\beta\gamma}\de_t\de_\alpha \omega_L \otimes \frac{1}{4}\big(\theta^t_{\alpha\beta,\beta\gamma,\gamma q,L}(v)+\theta^t_{\alpha\gamma,\gamma \beta,\beta q,L}(v)\big).
	\end{align*}
\end{proof}
The following result is fundamental in our study.
\begin{corollary}\label{fund1}
	If $\varphi:M(V)\rightarrow M(W)$ is a morphism of Verma modules as in \eqref{varfi}, then for all $p,q,a,b,c$ such that $\{p,q,a,b,c\}=[5]$ we have
	\begin{equation}\label{sing1}-\theta^p_{J\setminus pq}(v)+\frac{1}{2}\varepsilon_{pqabc} \sum_{\alpha\beta\gamma} \big(-(x_p\de_\gamma.\theta_{\alpha\beta,J})(v)+2x_p\de_\gamma.(\theta_{\alpha\beta,J}(v))\big)=0
	\end{equation}
for all $J\in \mathcal I_{d-1}$ and
\begin{equation} \label{sing2}
\frac{1}{4}\big(\theta_{ab,bc,cq,K}(v)+\theta_{ac,cb,bq,K}(v)\big)-\theta^{a,p}_{K\setminus pq}(v)+\frac{1}{2}\varepsilon_{pqabc} \sum_{\alpha\beta\gamma} \big(-(x_p\de_\gamma.\theta^a_{\alpha\beta,K})(v)+2x_p\de_\gamma.(\theta^a_{\alpha\beta,K}(v))\big)=0,
\end{equation}
\begin{equation} \label{sing3}
-2\theta^{p,p}_{K\setminus pq}(v)+\frac{1}{2}\varepsilon_{pqabc}\sum_{\alpha\beta\gamma} \big(-(x_p\de_\gamma.\theta^p_{\alpha\beta,K})(v)+2x_p\de_\gamma.(\theta^p_{\alpha\beta,K}(v))\big)=0,
\end{equation}
\begin{equation} \label{sing4}
-2\theta^{p,q}_{K\setminus pq}(v)-\theta_{ab,bc,ca,K}(v)+\varepsilon_{pqabc}\big(-(x_p\de_\gamma.\theta^q_{\alpha\beta,K})(v)+2x_p\de_\gamma.(\theta^q_{\alpha\beta,K}(v))\big)=0,
\end{equation}
for all $K\in \mathcal I_{d-3}$.

\end{corollary}
\begin{proof}
	Recall that $\varphi(v)=\sum_{l\leq d/2} \sum_{I\in \mathcal I_{d-2l}/\sim} \sum_{1\leq r_1\leq \cdots \leq r_l\leq 5} \de_{r_1}\cdots \de_{r_l} \omega_I \otimes \theta^{r_1,\ldots,r_l}_I(v)$.
	By Proposition \ref{morphisms} we have $x_pd_{pq} \varphi(v)=0$ and if we expand
	\[
	x_p d_{pq}\varphi(v)=\sum_{l\leq (d-1)/2} \sum_{1\leq r_1\leq\cdots\leq r_l\leq 5} \sum_{I\in \mathcal I_{d-1-2l}/\sim}\de_{r_1}\cdots \de_{r_l} \omega_I\otimes v_{(r_1,\ldots,r_l),I} 
	\]
	we have that all vectors $v_{(r_1,\ldots,r_l),I}$ must be 0.
	By Lemmas \ref{xpdpqaomth1} and \ref{xpdpqaomth2} we have
	for all $J\in \mathcal I_{d-1}$
	\[
	v_{(\,),J}=-\theta^p_{J\setminus pq}(v)+\frac{1}{2}\varepsilon_{pqabc} \sum_{\alpha\beta\gamma} \big(-(x_p\de_\gamma.\theta_{\alpha\beta,J})(v)+2x_p\de_\gamma.(\theta_{\alpha\beta,J}(v))\big)
	\]
	hence \eqref{sing1} follows.
	Moreover, for all $K\in \mathcal I_{d-3}$, by Lemmas \ref{xpdpqaomth1} and \ref{xpdpqaomth2} we have
	\[
	v_{(a),K}=\frac{1}{4}\big(\theta_{ab,bc,cq,K}(v)+\theta_{ac,cb,bq,K}(v)\big)-\theta^{a,p}_{K\setminus pq}(v)+\frac{1}{2}\varepsilon_{pqabc} \sum_{\alpha\beta\gamma} \big(-(x_p\de_\gamma.\theta^a_{\alpha\beta,K})(v)+2x_p\de_\gamma.(\theta^a_{\alpha\beta,K}(v))\big).
	\]
	Note that in this case we have an additional term $-\theta^{a,p}_{K\setminus pq}(v)$ which is produced by
	\[x_p d_{pq}\,\de_a \de_p \omega_{K\setminus pq} \otimes \theta_{K\setminus pq}(v).\] Equation \eqref{sing2} follows. Equations \eqref{sing3} and \eqref{sing4} are obtained similarly by considering $v_{(p),K}$ and $v_{(q),K}$.
\end{proof}

\medskip

Corollary \ref{fund1} can be slightly simplified if $v$ is a highest weight vector in $V$.
For  all $n,m$ we let
\[
\chi_{n>m}=\begin{cases}
1& \textrm{if }n>m;\\
0& \textrm{otherwise.}
\end{cases}\]
\begin{corollary}\label{singo}
	Let $\varphi:M(V)\rightarrow M(W)$ be a morphism of Verma modules as in \eqref{varfi} and let $s\in V$ be a highest weight vector. Then for all $\{p,q,a,b,c\}$ we have:
	\begin{equation}\label{sing1b}-2\theta^p_{J\setminus pq}(s)+\varepsilon_{pqabc} \sum_{\alpha\beta\gamma} \big((-1)^{\chi_{p>\gamma}}(x_p\de_\gamma.\theta_{\alpha\beta,J})(s)+2\chi_{p>\gamma}x_p\de_\gamma.(\theta_{\alpha\beta,J}(s))\big)=0
	\end{equation}
	for all $J\in \mathcal I_{d-1}$ and
	\begin{equation} \label{sing2b}
	-4\theta^{a,p}_{K\setminus pq}(s)+\big(\theta_{ab,bc,cq,K}(s)+\theta_{ac,cb,bq,K}(s)\big)+2\varepsilon_{pqabc} \sum_{\alpha\beta\gamma} \big((-1)^{\chi_{p>\gamma}}(x_p\de_\gamma.\theta^a_{\alpha\beta,K})(s)+2\chi_{p>\gamma}x_p\de_\gamma.(\theta^a_{\alpha\beta,K}(s))\big)=0,
	\end{equation}
	\begin{equation} \label{sing3b}
	-4\theta^{p,p}_{K\setminus pq}(s)+\varepsilon_{pqabc}\sum_{\alpha\beta\gamma} \big((-1)^{\chi_{p>\gamma}}(x_p\de_\gamma.\theta^p_{\alpha\beta,K})(s)+2\chi_{p>\gamma}x_p\de_\gamma.(\theta^p_{\alpha\beta,K}(s))\big)=0,
	\end{equation}
	\begin{equation} \label{sing4b}
	-2\theta^{p,q}_{K\setminus pq}-\theta_{ab,bc,ca,K}(s)+\varepsilon_{pqabc}\big((-1)^{\chi_{p>\gamma}}(x_p\de_\gamma.\theta^q_{\alpha\beta,K})(s)+2\chi_{p>\gamma}x_p\de_\gamma.(\theta^q_{\alpha\beta,K}(s))\big)=0
	\end{equation}
	for all $K\in \mathcal I_{d-3}$.
\end{corollary}
\begin{proof}
We prove Equation \eqref{sing1b}, the others are similar. By \eqref{sing1}, if $p>\gamma$ we clearly have \[(-1)^{\chi_{p>\gamma}}(x_p\de_\gamma.\theta_{\alpha\beta,J})(s)+2\chi_{p>\gamma}x_p\de_\gamma.(\theta_{\alpha\beta,J}(s))=-(x_p\de_\gamma.\theta_{\alpha\beta,J})(s)+2x_p\de_\gamma.(\theta_{\alpha\beta,J}(s)).
\]
If $p<\gamma$ we have
\[
-(x_p\de_\gamma.\theta_{\alpha\beta,J})(s)+2x_p\de_\gamma.(\theta_{\alpha\beta,J}(s))=(x_p \de_\gamma.\theta_{\alpha\beta,J})(s)
\]
since $s$ is a highest weight vector, and the result follows.
\end{proof}
Now observe that all (non trivial) summands in any equation appearing in Corollary \ref{singo} have the same weight, and we call it the weight of the equation. Next result shows that if $\varphi:M(\lambda)\rightarrow M(\mu)$ is a morphism between finite Verma modules, then every equation of weight $\mu$ in Corollary \ref{singo} can be further simplified. 
	\begin{corollary}\label{singol}Let $\varphi:M(\lambda)\rightarrow M(\mu)$ be a morphism of finite Verma modules.
		If $J\in \mathcal I_{d-1}$ and $a,b,c,p,q$ are such that Equation \eqref{sing1b} has weight $\mu$, then
	\begin{equation}\label{sing1c}-2\theta^p_{J\setminus pq}(s)+\varepsilon_{pqabc} \sum_{\alpha\beta\gamma} (-1)^{\chi_{p>\gamma}}(x_p\de_\gamma.\theta_{\alpha\beta,J})(s)=0.
	\end{equation}
	If $K$ and $a,b,c,p,q$ are such that Equation \eqref{sing2b} has weight $\mu$, then
	\begin{equation} \label{sing2c}
	-4\theta^{a,p}_{K\setminus pq}(s)+\theta_{ab,bc,cq,K}(s)+\theta_{ac,cb,bq,K}(s)+2\varepsilon_{pqabc} \sum_{\alpha\beta\gamma} (-1)^{\chi_{p>\gamma}}(x_p\de_\gamma.\theta^a_{\alpha\beta,K})(s)=0.
	\end{equation}
	
	If $K$ and $a,b,c,p,q$ are such that Equation \eqref{sing3b} has weight $\mu$ then
	\begin{equation} \label{sing3c}
	-4\theta^{p,p}_{K\setminus pq}(s)+\varepsilon_{pqabc}\sum_{\alpha\beta\gamma} (-1)^{\chi_{p>\gamma}}(x_p\de_\gamma.\theta^p_{\alpha\beta,K})(s)=0.
	\end{equation}
	If $K$ and $a,b,c,p,q$ are such that Equation \eqref{sing4b} has weight $\mu$ then
	\begin{equation} \label{sing4c}
	-2\theta^{p,q}_{K\setminus pq}(s)-\theta_{ab,bc,ca,K}(s)+\varepsilon_{pqabc}(-1)^{\chi_{p>\gamma}}(x_p\de_\gamma.\theta^q_{\alpha\beta,K})(s)=0.
	\end{equation}
\end{corollary}

\medskip

Note that all Equations appearing in Corollary \ref{singol} do not depend on the weights $\lambda$ and $\mu$: this observation will be the keypoint of our final classification.
\section{Singular vectors of degree between 5 and 10}\label{min10}
If $w\in M(\mu)$ is a singular vector of degree $d$ at most 10 we know that it also has height $d$ by the results in Section 4.
In particular we can express it as
\begin{equation}\label{lare}
w=\sum_{l\leq d/2} \sum_{I\in \mathcal I_{d-2l}/\sim} \sum_{1\leq r_1\leq \cdots \leq r_l\leq 5} \de_{r_1}\cdots \de_{r_l} \omega_I \otimes \theta^{r_1,\ldots,r_l}_I(s),
\end{equation}
where $s$ is a highest weight vector in $F(\lambda)$.
\begin{lemma}\label{I0}
If $w\in M(\mu)$ is a singular vector of degree and height $d$ as in \eqref{lare}, then there exists $I_0\in \mathcal I_d$ such that $\theta_{I_0}(s)\neq 0$ is a highest weight vector in $F(\mu)$.
\end{lemma}
\begin{proof}
Since $w$ has height $d$, we know that there exists $I\in \mathcal I_d$ such that $\theta_I(s)\neq 0$. Among all such $I$'s choose $I_0$ such that $\theta_{I_0}(s)$ has maximal weight. Applying $x_i\de_{i+1}$ to $w$ we obtain a term $\omega_{I_0}\otimes x_i \de_{i+1}.\theta_{I_0}(s)$, which cannot be simplified by any other term. Therefore $x_i \de_{i+1}.\theta_{I_0}(s)=0$.
\end{proof}
If we fix any possible $I_0$ as in Lemma \ref{I0}, we can consider all equations in Corollary \ref{singol} with weight $\mu$ and we observe that these equations do not depend on $\mu$.
For example, if we choose $I_0=(12,24,34,45)$, we can consider Equation \eqref{sing1c} with  $(a,b,c,p,q)=(1,2,4,5,3)$ and $J=(25,34,45)$, getting
\[
2\theta_{14,24,25,34}(s)+2\theta_{12,24,34,45}(s)=0,
\]
and also with $(a,b,c,p,q)=(2,4,5,3,1)$ and $J=(14,23,34)$, getting
\[
2\theta_{14,24,25,34}(s)=0,\]
and deducing $\theta_{I_0}(s)=0$.
\begin{theorem}\label{5o7}
	Let 
	\[
	w=\sum_{l\leq d/2} \sum_{I\in \mathcal I_{d-2l}/\sim} \sum_{1\leq r_1\leq \cdots \leq r_l\leq 5} \de_{r_1}\cdots \de_{r_l} \omega_I \otimes \theta^{r_1,\ldots,r_l}_I(s)
	\]
	 be a singular vector of degree $d$, with $5\leq d\leq 10$ and let ${I_0}\in \mathcal I_d$ be such that $\theta_{I_0}(s)\neq 0$ is a highest weight vector. Then $d=7$ and $I_0\sim (12,13,14,15,25,35,45)$ or $d=5$ and $I_0\sim (12,13,14,15,45)$ or $I_0\sim(12,15,25,35,45)$.
\end{theorem}
\begin{proof}
The proof is based on Corollary \ref{singol}. The set of Equations \eqref{sing1c}--\eqref{sing4c} of weight $\lambda(\theta_{I_0}(s))$ provides a system of homogeneous linear equations among all $\theta_I(s)$, $\theta_J^r(s)$ and $\theta^{r_1,r_2}_K(s)$ (with $I\in \mathcal I_d$, $J\in \mathcal I_{d-2}$ and $K \in \mathcal I_{d-4}$) such that $\theta_I$, $\theta_J^r$ and $\theta^{r_1,r_2}_K$ have the same weight of $\theta_{I_0}$, and which do not depend on (the weight of) $s$. This system can be solved with the help of a computer in all possible cases and one can check that it  implies $\theta_{I_0}(s)=0$ in all cases, but in the three exceptions stated above.

We add a few words to explain what happens in the most complicated case, i.e., $d=10$ and $I_0=(12,13,14,15,23,24,25,34,35,45)$. In this case 86 variables are involved: $\theta_{I_0}(s)$, 15 vectors of the form $\theta^r_J(s)$, and 70 vectors of the form $\theta^{r_1,r_2}_K(s)$ with $r_1\neq r_2$. Equations \eqref{sing1c} and \eqref{sing3c} do not provide any condition.
In Equation \eqref{sing2c} we can choose $(a,b,c,p,q)$ to be any permutation in $S_5$ and $K=(ac,ap,aq,bp,bq,cp,pq)$, getting 120 linear equations among our 86 variables, and in Equation \eqref{sing4c} we can choose $(a,b,c,p,q)$ to be any permutation in $S_5$ (with $a<b<c$ to avoid repeated equations) and $K=(ap,aq,bp,bq,cp,cq,pq)$ getting 20 more equations. This system of 140 equations implies that all 86 variables involved vanish.
\end{proof}
Now we study the exceptions given by Theorem \ref{5o7}. The case of degree 7 leads to another new singular vector.
\begin{theorem}\label{w[7]}
The following vector $w[7]\in M(0,0,0,2)$ of weight $(2,0,0,0)$ is the unique (up to a scalar factor) singular vector of degree 7 in a finite Verma module: 
\begin{align*}w[7]&=d_{12}d_{13} d_{14} d_{15} \Big( d_{23} d_{24} d_{25} \otimes ({x_{2}^*})^2
- d_{23} d_{25} d_{34} \otimes x_2^*x_3^*
-  d_{24} d_{25} d_{34} \otimes x_2^*x_4^*
+  d_{23} d_{24} d_{35} \otimes x_2^*x_3^*\\
&-  d_{24} d_{25} d_{35} \otimes x_2^* x_5^*
+  d_{23} d_{34} d_{35} \otimes (x_3^*)^2
+  d_{24} d_{34} d_{35} \otimes x_3^* x_4^*
+  d_{25} d_{34} d_{35} \otimes x_3^* x_5^*
+  d_{23} d_{24} d_{45} \otimes x_2^*x_4^*\\
&+  d_{23} d_{25} d_{45} \otimes x_2^* x_5^*
+  d_{23} d_{34} d_{45} \otimes x_3^* x_4^*
+  d_{24} d_{34} d_{45} \otimes (x_4^*)^2
+  d_{25} d_{34} d_{45} \otimes x_4^* x_5^*
+  d_{23} d_{35} d_{45} \otimes x_3^* x_5^*\\
&+  d_{24} d_{35} d_{45} \otimes x_4^* x_5^*
+  d_{25} d_{35} d_{45} \otimes (x_5^*)^2
+\de_1  d_{23} \otimes x_2^* x_3^*
+\de_1  d_{24} \otimes x_2^* x_4^*
+\de_1  d_{25} \otimes x_2^* x_5^*\\
&-\de_2  d_{23} \otimes x_1^* x_3^*
-\de_2   d_{24} \otimes x_1^* x_4^*
-\de_2  d_{25} \otimes x_1^* x_5^*
+\de_3  d_{23} \otimes x_1^* x_2^*
-\de_3 d_{34} \otimes x_1^* x_4^*
-\de_3   d_{35} \otimes x_1^* x_5^*\\
&+\de_4   d_{24}\otimes x_1^* x_2^*
+\de_4   d_{34}\otimes x_1^* x_3^*
-\de_4  d_{45} \otimes x_1^* x_5^*
+\de_5   d_{25}\otimes x_1^* x_2^*
+\de_5   d_{35}\otimes x_1^* x_3^*
+\de_5  d_{45}\otimes x_1^* x_4^*\Big).
\end{align*}
\end{theorem}

\begin{proof}
Let $I_0=(12,13,14,15,25,35,45)$. In this case Equations \eqref{sing1c}--\eqref{sing4c} of weight $\lambda(\theta_{I_0}(s))$ provide the following homogeneous linear relations:
\begin{align}\label{condsette}
\theta_{13, 14, 15, 25, 35, 45}(s)&=
-2\theta_{14, 15, 25, 35}^2(s)=
2\theta_{12, 13, 15, 25, 45}^2(s)=
-2\theta_{13, 14, 15, 25, 35}^3(s)=
2\theta_{12, 13, 15, 35, 45}^3(s)\\
\nonumber &=
-2\theta_{13, 14, 15, 25, 45}^4(s)=
2\theta_{12, 14, 15, 35, 45}^4(s)=
-4\theta_{12, 15, 25}^{2,2}(s)=
-4\theta_{13, 15, 25}^{2,3}(s)\\
\nonumber &=
-4\theta_{ 12, 15, 35}^{2,3}(s)=
-4\theta_{ 13, 15, 35}^{3,3}(s)=
-4	\theta_{ 14, 15, 25}^{2,4}(s)=
-4	\theta_{ 12, 15, 45}^{2,4}(s)\\
\nonumber &=
-4	\theta_{ 14, 15, 35}^{3,4}(s)=
-4	\theta_{ 13, 15, 45}^{3,4}(s)=
-4	\theta_{ 14, 15, 45}^{4,4}(s).
\end{align}
We use Equation \eqref{sing1b} to determine the weight $\mu=(\mu_{12},\mu_{23},\mu_{34}, \mu_{45})$.
Taking $(a,b,c,p,q)=(1,2,3,4,5)$ and $J=(13,14,15,25,35,45)$ in \eqref{sing1b} we obtain $\mu_{34}=0$, using \eqref{condsette}.

Taking $(a,b,c,p,q)=(4,5,1,3,2)$ and $J=(12,13,14,15,25,35,)$ in \eqref{sing1b} we obtain $\mu_{13}=0$.

Taking $(a,b,c,p,q)=(1,2,4,5,3)$ and $J=(13,14,15,25,35,45)$ in \eqref{sing1b} we have
\begin{align*}
0&=-2\theta^5_{(13,14,15,25,35,45)\setminus(53)}-(x_5\de_4.\theta_{12,13,14,15,25,35,45})(s)+2x_5\de_4(\theta_{12,13,14,15,25,35,45}(s))\\&\hspace{5mm}-(x_5\de_1.\theta_{24,13,14,15,25,35,45})(s)\\
&=2\theta^5_{13,14,15,25,45}(s)+\theta_{12,13,14,15,25,35,45}(s)+\theta_{12,13,14,15,25,34,45}(s)+2x_5\de_4.(\theta_{12,13,14,15,25,35,45}(s))\\&\hspace{5mm}+\theta_{24,13,14,15,21,35,45}(s)\\
&=2\theta^5_{13,14,15,25,45}(s)+2\theta_{12,13,14,15,25,35,45}(s)+\theta_{12,13,14,15,25,34,45}(s)+2x_5\de_4.(\theta_{12,13,14,15,25,35,45}(s)).
\end{align*}
Finally we can apply $x_4\de_5$ to this equation and use \eqref{condsette} to conclude
\begin{align*}
0&=2\theta^4_{13,14,15,25,45}(s)-2\theta_{12,13,14,15,25,35,45}(s)-\theta_{12,13,14,15,25,35,45}(s)+2\mu_{45}\theta_{12,13,14,15,25,35,45}(s)\\&=2(\mu_{45}-2)\theta_{12,13,14,15,25,35,45}(s).
\end{align*}
This shows that the only possible singular vector of degree 7 sits in $M(0,0,0,2)$ and has weight $(0,0,0,2)+\lambda(\omega_{12,13,14,15,25,35,45})=(2,0,0,0)$. The uniqueness of such singular vector follows from Lemma \ref{I0}, since there are no other $I\in \mathcal I_7$ such that $\lambda (I)=\lambda (I_0)$.
The fact that the displayed vector is indeed a singular vector can be checked with a long and technical calculation.

Note that \eqref{condsette} is consistent with the vector $w[7]$ since one can check that
\begin{align*}\label{sing7}
4&d_{12}d_{13}d_{14}d_{15}d_{25}d_{35}d_{45}=\\
&4\omega_{13, 14, 15, 25, 35, 45}
-2\de_2\omega_{12,14, 15, 25, 35}+
2\de_2\omega_{12, 13, 15, 25, 45}
-2\de_3\omega_{13, 14, 15, 25, 35}+
2\de_3\omega_{12, 13, 15, 35, 45}\\
\nonumber &
-2\de_4\omega_{13, 14, 15, 25, 45}+
2\de_4\omega_{12, 14, 15, 35, 45}
-\de_2^2\omega_{12, 15, 25}
-\de_2\de_3\omega_{13, 15, 25}-\de_2\de_3\omega_{ 12, 15, 35}-\de_3^2\omega_{ 13, 15, 35}\\
\nonumber &
-	\de_2\de_4\omega_{ 14, 15, 25}
-	\de_2\de_4\omega_{ 12, 15, 45}
-	\de_3\de_4\omega_{ 14, 15, 35}
-	\de_3\de_4\omega_{ 13, 15, 45}-	\de_4^2\omega_{ 14, 15, 45}.
\end{align*}
\end{proof}
The two possible cases in degree 5 given by Theorem \ref{5o7} are dual to each other. They lead to singular vectors which were already known to Rudakov in \cite{R}.
\begin{theorem}
	Let $w$ be a singular vector of degree 5 in $M(\mu)$ of weight $\lambda$. Then one of the following occurs.
\begin{enumerate}
	\item $\mu=(0,0,1,0)$, $\lambda=(3,0,0,0)$ and 
	\[
	w=w[5_{CD}]=d_{12}d_{13}d_{14}d_{15}(d_{45}\otimes x_{45}^*+d_{35}\otimes x_{35}^* +d_{25}\otimes x_{25}^*+d_{24}\otimes x_{24}^*+d_{23}\otimes x_{23}^*);
	\]
	\item $\mu=(0,0,0,3)$, $\lambda=(0,1,0,0)$ and $w=w[5_{EA}]=d_{12}w[4_E]$,
where $w[4_E]$ is explicitly described in Section \ref{w4E}.
\end{enumerate}
\end{theorem}
\begin{proof}
	By Theorem \ref{5o7} we can assume that $\theta_{I_0}(s)$ is a highest weight vector with $I_0=(12,13,14,15,45)$ or $I_0=(12,15,25,35,45)$ and by Theorem \ref{duality} it is enough to show that the case $I_0=(12,13,14,15,45)$ leads to  conditions (1).

In this case Equations \eqref{sing1c}--\eqref{sing4c} easily provide the following relations:
\begin{equation}\label{condcinque}
\theta_{12,13,14,15,45}(s)=-2\theta^2_{12,14,15}(s)=-2\theta^3_{13,14,15}(s).
\end{equation}
We use Equation \eqref{sing1b} three times to show that  necessarily $\mu=(0,0,1,0)$.
We first use Equation \eqref{sing1b} with $(a,b,c,p,q)=(4,5,1,3,2)$ and $J=(12,13,14,15)$. All terms but one vanish and we obtain
\[
x_3\de_1.(\theta_{45,12,13,14,15}(s))=0,
\]
and so $\mu_{1,2}=\mu_{2,3}=0$.
Using Equation (9) with $(a,b,c,p,q)=(1,2,3,4,5)$ and $J=(13,14,15,45)$, we obtain
\begin{align*}
0&=-2\theta^4_{(13,14,15,45)\setminus(45)}-(x_4\de_3.\theta_{12,13,14,15,45})(s)+2x_4\de_3.(\theta_{12,13,14,15,45}(s))
-(x_4\de_2.\theta_{31,13,14,15,45})(s)\\
&\hspace{5mm}-2x_4\de_2.(\theta_{31,13,14,15,45}(s))
-(x_4\de_1.\theta_{23,13,14,15,45})(s)-2x_4\de_1.(\theta_{23,13,14,15,45}(s))\\
&=2\theta^4_{13,14,15}(s)+\theta_{12,13,14,15,35}(s)+2x_4\de_3.(\theta_{12,13,14,15,45}(s)),
\end{align*}
where we have used the fact that $\theta_{23,13,14,15,45}(s)=0$ since it has weight greater than $\theta_{12,13,14,15,45}(s)$.
Applying $x_3\de_4$ to the previous equation, we obtain
\[
-2\theta^3_{13,14,15}(s)-\theta_{12,13,14,15,45}(s)+2\mu_{34}\theta_{12,13,14,15,45}(s)=0.
\] 
By \eqref{condcinque} we can conclude that $\mu_{34}=1$.

Similarly, by Equation (9) with $(a,b,c,p,q)=(1,2,3,5,4)$ and $J=(13,14,15,45)$, we obtain
\begin{align*}
0&=-2\theta^5_{(13,14,15,45)\setminus(54)}+(x_5\de_3.\theta_{12,13,14,15,45})(s)-2x_5\de_3.(\theta_{12,13,14,15,45}(s))
\end{align*}
and applying $x_3\de_5$ and then using \eqref{condcinque} we obtain
\[
0=-2\theta^3_{13,14,15}(s)+\theta_{12,13,14,15,45}(s)-2\mu_{35}\theta_{12,13,14,15,45}(s)=2(1-\mu_{35})\theta_{12,13,14,15,45}(s).
\]
So $\mu_{35}=1$ and $\mu_{45}=\mu_{35}-\mu_{34}=0$.
The weight of $w$ is $\lambda=\mu+\lambda(\omega_{12,13,14,15,45})=(0,0,1,0)+(3,0,-1,0)=(3,0,0,0)$. The uniqueness follows by Lemma \ref{I0} and the verification that $d_{12}d_{13}d_{14}d_{15}(d_{45}\otimes x_{45}^*+d_{35}\otimes x_{35}^* +d_{25}\otimes x_{25}^*+d_{24}\otimes x_{24}^*+d_{23}\otimes x_{23}^*)$ is actually such a singular vector is left to the reader.
This shows, by duality, that there exists a (unique) singular vector in $M(0,0,0,3)$ of weight $(0,1,0,0)$, and one can check that it is given by $d_{12}w[4_E]$.
\end{proof}
\section{Degree 4}\label{d4}

The last case to be considered concerns singular vectors of degree and height 4.
\begin{proposition}\label{gr4}
	Let $w$ be a singular vector of degree 4 as in (\ref{varfi}) and let $I_0\in \mathcal I_4$ be such that $\theta_{I_0}(s)\neq 0$ is a highest weight vector in $M(\mu)$. Then $I_0\sim(12,13,14,15)$ or $I_0\sim(15,25,35,45)$.
\end{proposition}
\begin{proof}
	
We make use of the duality in Theorem \ref{duality} to consider nearly a half of the cases. Indeed let $w$ be a singular vector in $M(\mu)$ of weight $\lambda$ such that $\theta_{I_0}(s)$ is a highest weight vector (of weight $\mu$). Also consider the dual singular vector $w^*$ in $M(\lambda^*)$ of weight $\mu^*$ and assume that $w^*=\varphi^*(s^*)$, where $s^*$ is a highest weight vector in $M(\lambda^*$),  can be expressed as in \eqref{varfi} with $\theta^*$'s instead of $\theta$'s, and with $\theta^*_{J_0}(s')$ of weight $\lambda'$. Then $\lambda(\theta_{I_0})=\mu-\lambda$ and $\lambda(\theta^*_{J_0})=\lambda^*-\mu^*=-\lambda(\theta_{I_0})^*$.  In particular if there are no singular vectors such that $\theta_{I_0}(s)\neq 0$ for all $I_0$ such that $\lambda(\theta_{I_0})=\lambda$ then there are no singular vectors such that $\theta_{I_0}(s)\neq 0$ for all $I_0$ such that $\lambda(\theta_{I_0})=-\lambda^*$. 

As in Theorem \ref{5o7}, we can use Equations \eqref{sing1c}--\eqref{sing4c}, but in this case we have some more exceptions which must be considered apart. 

More precisely we have that $\theta_{I_0}(s)=0$ in all but the following cases (and their duals):
\begin{enumerate}
	\item $(12,13,14,15)$;
	\item $(12,14,15,23)$, $(12,13,15,24)$, $(12,13,14,25)$;
	\item $(13,23,34,35)$;
	\item $(14,23,34,35)$, $(13,24,34,35)$, $(13,23,34,45)$;
	\item $(14,24,34,35)$, $(13,24,34,45)$, $(14,23,34,45)$;
	\item $(14,24,34,45)$;
	\item $(15,24,34,45)$, $(14,34,25,45)$, $(14,24,35,45)$;
	\item $(15,25,34,45)$, $(14,25,35,45)$, $(15,24,35,45)$.
\end{enumerate}
These exceptions have been grouped according to their weight.
We now analyze  all these cases.
\begin{enumerate}
\item This is the case which is not excluded by the statement.
\item Equations \eqref{sing1c}--\eqref{sing4c} provide $\theta_{12,14,15,23}=\theta_{12,13,14,25}=-\theta_{12,13,15,24}$.
In this case Equation \eqref{sing1b} with $a=2$, $b=3$, $c=1$, $p=5$, $q=4$ and $J=(12,14,15)$ gives:
$$x_5\de_2.(\theta_{31,12,14,15}(s))+x_5\de_1.(\theta_{23,12,14,15}(s))=0.$$
If we apply the vector field $x_1\de_5$, we get the following equation:
$$x_1\de_2.(\theta_{31,12,14,15}(s))+\mu_{15}(\theta_{23,12,14,15}(s))=0,$$
where we used that if $\theta_{23,12,14,15}(s)$ has the highest weight, then
$\theta_{35,12,14,15}(s)=\theta_{31,52,14,15}(s)=\theta_{31,12,54,15}(s)=0$.
Hence we get
$$-\theta_{32,12,14,15}(s)-\theta_{31,12,24,15}(s)-\theta_{31,12,14,15}(s)
+\mu_{15}(\theta_{23,12,14,15}(s))=0,$$ i.e.,
$(-\mu_{15}-3)(\theta_{12,14,15,23}(s))=0$,
a contradiction.
\item 
In this case Equation \eqref{sing1b} with $a=1$, $b=3$, $c=2$, $p=5$, $q=4$ and $J=(23,34,35)$ gives:
$$-x_5\de_2.(\theta_{13,23,34,35}(s))+x_5\de_3.(\theta_{12,23,34,35}(s))=0.$$
If we apply the vector field $x_2\de_5$, we get the following equation:
$$-\mu_{25}(\theta_{13,23,34,35}(s))+x_2\de_3.(\theta_{12,23,34,35}(s))=0,$$
where we used that if $\theta_{13,23,34,35}(s)$ has the highest weight then
$\theta_{15,23,34,35}(s)=0$.
Hence we get
$$(-\mu_{25}-1)(\theta_{13,23,34,35}(s))=0,$$
a contradiction.
\item In this case Equations \eqref{sing1c}--\eqref{sing4c} provide $\theta_{14,23,34,35}=\theta_{13,24,34,35}=\theta_{13,23,34,45}$.
Equation \eqref{sing1b} with $a=1$, $b=2$, $c=3$, $p=5$, $q=4$ and $J=(24,34,35)$ gives:
\begin{align*} -(x_5\de_1. \theta_{23,24,34,35})(s)+(x_5\de_2.\theta_{31,24,34,35})(s)
+
(x_5\de_3.\theta_{12,24,34,35})(s)\\
+2
x_5\de_2.(\theta_{31,24,34,35}(s))+
2x_5\de_3.(\theta_{12,24,34,35}(s))=0,
\end{align*}
where we used that, if $\theta_{13,24,34,35}(s)$ has  the highest weight, then
$\theta_{23,24,34,35}(s)=0$. This is equivalent to the following:
$$2\theta_{13,23,24,34}(s)-2x_5\de_2.(\theta_{13,24,34,35}(s))+
2x_5\de_3.(\theta_{12,24,34,35}(s))=0.$$
If we apply the vector field $x_2\de_5$ and use the equality
$\theta_{13,24,34,35}(s)=\theta_{13,23,34,45}(s)$, we get the following equation:
$$-\mu_{25}(\theta_{13,24,34,35}(s))+x_2\de_3.(\theta_{12,24,34,35}(s))=0,$$
where we used that, if $\theta_{13,24,34,35}(s)$ has the highest weight, then
$\theta_{15,24,34,35}(s)=0=\theta_{12,34,35,45}(s)$.
Hence we get
$$(-\mu_{25}-1)(\theta_{13,24,34,35}(s))=0,$$
a contradiction.
\item  In this case Equations \eqref{sing1c}--\eqref{sing4c} provide $\theta_{14,24,34,35}=\theta_{14,23,34,45}=\theta_{13,24,34,45}$.
Equation \eqref{sing1} with $a=1$, $b=4$, $c=3$, $p=5$, $q=2$ and $J=(24,34,35)$ gives:
$$-x_5\de_4.(\theta_{13,24,34,35}(s))+x_5\de_3.(\theta_{14,24,34,35}(s))=0.$$
If we apply the vector field $x_3\de_5$, we get the following equation:
$$-x_3\de_4.(\theta_{13,24,34,35}(s))+\mu_{35}(\theta_{14,24,34,35}(s))=0,$$
where we used that if $\theta_{14,24,34,35}(s)$ has the highest weight, then
$\theta_{15,24,34,35}(s)=0=\theta_{13,24,35,45}(s)$.
Hence, using the hypothesis $\theta_{14,24,34,35}(s)=\theta_{13,24,34,45}(s)$, we get
$$(2+\mu_{35})(\theta_{14,24,34,35}(s))=0,$$
a contradiction.
\item In this case Equation \eqref{sing1b} with $a=1$, $b=4$, $c=3$, $p=5$, $q=2$ and $J=(24,34,45)$ gives:
$$-x_5\de_4.(\theta_{13,24,34,45}(s))+x_5\de_3.(\theta_{14,24,34,45}(s))=0.$$
If we apply the vector field $x_3\de_5$, we get the following equation:
$$-x_3\de_4.(\theta_{13,24,34,45}(s))+\mu_{35}(\theta_{14,24,34,45}(s))=0,$$
where we used that if $\theta_{14,24,34,45}(s)$ has the highest weight, then
$\theta_{15,24,34,45}(s)=0$.
Hence we get
$$(1+\mu_{35})(\theta_{14,24,34,45}(s))=0,$$
a contradiction.
\item In this case Equations \eqref{sing1c}--\eqref{sing4c} provide $\theta_{15,24,34,45}=\theta_{14,25,34,45}=\theta_{14,24,35,45}$.
Equation \eqref{sing1b} with $a=1$, $b=2$, $c=4$, $p=5$, $q=3$ and $J=(15,34,45)$ gives:
\begin{align*}
-(x_5\de_1.\theta_{24,15,34,45})(s)&-(x_5\de_2.\theta_{41,15,34,45})(s)-(x_5\de_4.\theta_{12,15,34,45})(s)+2x_5\de_1.(\theta_{24,15,34,45}(s))\\&+2x_5\de_2.(\theta_{41,15,34,45}(s))+2x_5\de_4.(\theta_{12,15,34,45}(s))=0,
\end{align*}
which is equivalent to the following equation:
\begin{align*}
-2\theta_{14,15,24,34}(s)&+2\theta_{12,14,34,45}(s)
-2x_5\de_1.(\theta_{15,24,34,45}(s))-2x_5\de_2.(\theta_{14,15,34,45}(s))+\\
&2x_5\de_4.(\theta_{12,15,34,45}(s))=0.
\end{align*}
If we apply the vector field $x_1\de_5$ we get the following equation:
\begin{align*}
-2\theta_{54,15,24,34}(s)&-2\theta_{52,14,34,45}(s)-
2\mu_{15}(\theta_{15,24,34,45}(s))-2x_1\de_2.(\theta_{14,15,34,45}(s))\\
&+2x_1\de_4.(\theta_{12,15,34,45}(s))=0,\end{align*}
where we used that if $\theta_{15,24,34,45}(s)$ has highest weight then
$\theta_{15,25,34,45}(s)=0$.
Hence we get
$$(-6-2\mu_{15})(\theta_{15,24,34,45}(s))=0,$$
a contradiction.
\item We have by Equations \eqref{sing1c}--\eqref{sing4c} $\theta_{15,25,34,45}=\theta_{15,24,35,45}=\theta_{14,25,35,45}$ and $\theta^1_{15,45}=\theta^2_{25,45}=\theta^3_{35,45}=0$.
 In this case Equation \eqref{sing1b} with $a=1$, $b=2$, $c=4$, $p=5$, $q=3$ and $J=(25,35,45)$ gives:
 \begin{align*}
 -2\theta^5_{J\setminus \{53\}}(s)-(x_5\de_1.\theta_{24,25,35,45})(s)-
 (x_5\de_2.\theta_{41,25,35,45})(s)-(x_5\de_4.\theta_{12,25,35,45})(s)\\
 -2x_5\de_2.(\theta_{14,25,35,45}(s))+2x_5\de_4.(\theta_{12,25,35,45}(s))=0
 \end{align*}
 where we used that if $\theta_{14,25,35,45}(s)$ has highest weight then
$\theta_{24,25,35,45}(s)=0$.
Hence we have:
\begin{align*}
 -2\theta^5_{J\setminus \{53\}}(s)+2\theta_{12,24,35,45}(s)+\theta_{24,25,31,45}(s)
+\theta_{41,25,32,45}(s)+2\theta_{14,24,25,35}(s)\\
 +\theta_{12,25,34,45}(s)-2x_5\de_2.(\theta_{14,25,35,45}(s))+
 2x_5\de_4.(\theta_{12,25,35,45}(s))=0.
 \end{align*}
 If we apply the vector field $x_2\de_5$ we obtain the following equation:
 \begin{align*}
 -2\theta_{15,24,35,45}(s)-
\theta_{41,25,35,45}(s)-2\theta_{14,54,25,35}(s)\\
 -\theta_{15,25,34,45}(s)-2\mu_{25}(\theta_{14,25,35,45}(s))-
 2\theta_{14,25,35,45}(s)=0
 \end{align*}
 where we used $\theta^2_{25,45}(s)=0$ and that if $\theta_{14,25,35,45}(s)$ has highest weight then
$\theta_{15,25,35,45}(s)=0$. Now, using the hypotheses
$\theta_{15,25,34,45}=\theta_{15,24,35,45}=\theta_{14,25,35,45}$, we get:
$$-2(\mu_{25}+1)\theta_{14,25,35,45},$$
a contradiction.

\end{enumerate}
\end{proof}
The following result completes the study of singular vectors of degree 4
\begin{theorem}\label{degree4}
	Let $w$ be a singular vector in $M(\mu)$ of weight $\lambda$ and degree 4. Then one of the following occurs:
	\begin{enumerate}
		\item $\mu=(n,0,0,0)$, $\lambda=(n+3,0,0,0)$ and $w=d_{12}d_{13}d_{14}d_{15} \otimes x_1^{n}$ for some $n\in \mathbb N$.
		\item $\mu=(0,0,0,n+3)$, $\lambda=(0,0,0,n)$ and $w=w[4_E]$ (see Section \ref{w4E}) for some $n\in \mathbb N$.
	\end{enumerate}
\end{theorem}
\begin{proof}
By Proposition \ref{gr4} we know that we can assume that $w$ is as in \eqref{varfi} with $\theta_{12,13,14,15}(s)$ a highest weight vector. By \eqref{sing1b} with $(a,b,c,p,q)=(1,2,3,5,4)$ and $J=(12,14,15)$ we immediately get $x_5\de_2.(\theta_{12,13,14,15}(s))=0$ (recalling that $\theta_{12,23,14,15}(s)=0$ for weight reasons) and therefore $\mu=(n,0,0,0)$ for some $n$ and $\lambda=\lambda(\omega_{12,13,14,15})+(n,0,0,0)=(n+3,0,0,0)$. The uniqueness follows by Lemma \ref{I0}. It is a trivial check that the vector $d_{12}d_{13}d_{14}d_{15} \otimes x_1^{n}$ is such a singular vector. By duality we have that the other possible case in Proposition \ref{gr4} leads to a unique singular vector in $M(0,0,0,n+3)$ of weight $(0,0,0,n)$. The fact that this vector is actually $w[4_E]$ displayed in Section 10 is a huge verification that can be made with a computer.
\end{proof}
\section{Conclusions}
As a result of discussions in the previous Section we are now in a position to state a complete classification of singular vectors in finite Verma modules for $E(5,10)$.
\begin{theorem}\label{conclusions}
	The following is a complete classification of singular vectors and corresponding morphisms for finite Verma modules.

In degree 1 we have
\begin{itemize}
	\item $w[1_A]=d_{12}\otimes x_1^{m}x_{12}^{n}\in M(m,n,0,0)$ for all $m,n\geq 0$, giving a morphism 
	\[\varphi[1_A]:M(m,n+1,0,0)\rightarrow M(m,n,0,0);\]
	\item $w[1_B]=d_{15}\otimes x_1^m (x_5^*)^{n+1}+d_{14}\otimes x_1^m x_4^*(x_5^*)^{n}+d_{13}\otimes x_1^m x_3^*(x_5^*)^{n}+d_{12}\otimes x_1^m x_2^*(x_5^*)^{n}$ for all $m,n\geq 0$, giving a morphism
	\[\varphi[1_B]:M(m+1,0,0,n)\rightarrow M(m,0,0,n+1);\]
	\item $w[1_C]=\sum_{i<j} d_{ij}\otimes x_{ij}^*(x_{45}^*)^{m}(x_5*)^n$ for all $m,n \geq 0$, giving a morphism
	\[\varphi[1_C]:M(0,0,m,n)\rightarrow M(0,0,m+1,n).
	\] 
\end{itemize}
In degree 2 we have
\begin{itemize}
	\item $w[2_{BA}]=\sum_{j>1}d_{12}d_{1j}\otimes x_1^m x_j^* \in M(m,0,0,1)$ for all $m\geq 0$, giving a morphism
	\[
	\varphi[2_{BA}]=\varphi[1_B]\circ \varphi[1_A]:M(m+1,1,0,0)\rightarrow M(m,0,0,1);
	\]
	\item $w[2_{CB}]=\sum_{j>1} \sum_{h<k}d_{1j}d_{hk}\otimes x_{hk}^* x_j^* (x_5^*)^n \in M(0,0,1,n+1)$ for all $n\geq 0$, giving a morphism
	\[
	\varphi[2_{CB}]=\varphi[1_C]\circ \varphi[1_B]:M(1,0,0,n)\rightarrow M(0,0,1,n+1);
	\]
	\item $w[2_{CA}]=\sum_{i<j}d_{12}d_{ij}\otimes x_{ij}^* \in M(0,0,1,0)$, giving a morphism
	\[
	\varphi[2_{CA}]=\varphi[1_C]\circ \varphi[1_A]:M(0,1,0,0)\rightarrow M(0,0,1,0).
	\]
\end{itemize}
In degree 3 we have
\begin{itemize}
	\item $w[3_{CBA}]=\sum_{j>1}\sum_{k<l}d_{12}d_{1j}d_{kl}\otimes x_j^* x_{kl}^* \in M(0,0,1,1)$, giving a morphism
	\[
	\varphi[3_{CBA}])=\varphi[1_C]\circ \varphi[1_B]\circ \varphi[1_A]:M(1,1,0,0)\rightarrow M(0,0,1,1).
	\]
\end{itemize}
In degree 4 we have
\begin{itemize}
	\item $w[4_D]=d_{12}d_{13}d_{14}d_{15}\otimes x_1^m \in M(m,0,0,0)$ for all $m\geq 0$, giving a morphism
	\[
	\varphi[4_D]:M(m+3,0,0,0)\rightarrow M(m,0,0,0);
	\]
	\item $w[4_E]\in M(0,0,0,n+3)$ shown in \S \ref{w4E}, giving a morphism
	\[
	\varphi[4_E]:M(0,0,0,n)\rightarrow M(0,0,0,n+3).
	\]
\end{itemize}
In degree 5 we have
\begin{itemize}
	\item $w[5_{CD}]=d_{12}d_{13}d_{14}d_{15} \sum_{2<i<j} d_{ij}\otimes x_{ij}^*$, giving a morphism
	\[
	\varphi[5_{CD}]=\varphi[1_C]\circ \varphi:[4_D]: M(3,0,0,0)\rightarrow M(0,0,1,0);
	\]
	\item $w[5_{EA}]=d_{12}w[4_E]\in M(0,0,0,3)$, giving a morphism
	\[
	\varphi[5_{EA}]=\varphi[4_E]\circ \varphi[1_A]:M(0,1,0,0)\rightarrow M(0,0,0,3).
	\]
\end{itemize}
In degree 7 we have
\begin{itemize}
\item $w[7]\in M(0,0,0,2)$ given in Theorem \ref{w[7]}, giving a morphism
\[
\varphi[7]:M(2,0,0,0)\rightarrow M(0,0,0,2).
\]
\end{itemize}
In degree 11 we have
\begin{itemize}
	\item $w[11]\in M(0,0,0,1)$ given in Theorem \ref{w[11]}, giving a morphism
	\[
	\varphi[11]:M(1,0,0,0)\rightarrow M(0,0,0,1).
	\]
\end{itemize}
\end{theorem}
\begin{proof}
In \cite{KR} singular vectors of degree 1 were constructed, and in \cite{R} it was shown that there are no other ones. In \cite{R} singular vectors of degree 2,3,4 and 5 were constructed, and it was shown in \cite{CC} that there are no other ones of degree 2 and 3. In the paper we show that there are no other singular vectors of degree 4 and 5. More precisely, singular vectors of degree 4 are classified in Section \ref{d4}, singular vectors with degree equal to height greater than 5 are classified in Section \ref{min10}, and singular vectors of degree greater than height are classified in Section \ref{>10}.
\end{proof}
This theorem gives an affirmative answer to the conjecture posed in \cite{KR}.
\begin{corollary} All degenerate finite Verma modules over $E(5,10)$ are of the form $M(m,n,0,0)$, $M(0,0,m,n)$ or $M(m,0,0,n)$, where $m,n\in\N$.
\end{corollary}

Figure \ref{E510morphisms} represents all morphisms between  finite degenerate Verma modules, which are not compositions of other morphisms. 

Since a singular vector of weight $\mu$ in a finite Verma module with highest weight $\lambda$ corresponds to a non-zero morphism $M(\mu)\rightarrow M(\lambda)$, we can construct infinite sequences of morphisms as in Figure \ref{E510morphisms}. All of these sequences are complexes (i.e. all compositions of consecutive morphisms vanish), except for the one through the origin; the latter becomes a complex if we replace the sequence
$$M(0,1,0,0)\rightarrow M(0,0,0,0)\rightarrow M(1,0,0,0)$$
of morphisms of degree 1 by their composition of degree 2. This claim, when morphisms of degree 7 and 11 are not involved, follows from \cite{R}; if they are involved, it follows since there are no morphisms of degree 8 and 12.

Figure \ref{E510morphismsbis} represents all morphisms between  finite degenerate Verma modules of degree 2 and 3; the corresponding bilateral infinite sequences shown in this picture are complexes, since any possible composition of two of these morphisms does not appear in Theorem \ref{conclusions}.

\begin{figure}[h]
	\begin{center}
		\scalebox{1}{	$$
			\begin{tikzpicture}
			\draw[fill=white]{(0,0) circle(3pt)}; 
			\draw[fill=white]{(1,0) circle(3pt)}; 
			\draw[fill=none]{(2,0) circle(3pt)};
			\draw[fill=none]{(3,0) circle(3pt)};
			\draw[fill=none]{(4,0) circle(3pt)};
			\draw[fill=none]{(5,0) circle(3pt)};
			\draw[fill=none]{(6,0) circle(3pt)};
			\draw[fill=none]{(6,1) circle(3pt)};
			\draw[fill=none]{(6,2) circle(3pt)};
			\draw[fill=none]{(6,3) circle(3pt)};
			\draw[fill=none]{(6,4) circle(3pt)};
			\draw[fill=none]{(6,5) circle(3pt)};
			\draw[fill=none]{(-6,-6) circle(3pt)};
			\draw[fill=none]{(-1,-6) circle(3pt)};
			\draw[fill=none]{(-2,-6) circle(3pt)};
			\draw[fill=none]{(-3,-6) circle(3pt)};
			\draw[fill=none]{(-4,-6) circle(3pt)};
			\draw[fill=none]{(-5,-6) circle(3pt)};
			\draw[fill=none]{(0,1) circle(3pt)};
			\draw[fill=none]{(1,1) circle(3pt)};
			\draw[fill=none]{(2,1) circle(3pt)};
			\draw[fill=none]{(3,1) circle(3pt)};
			\draw[fill=none]{(4,1) circle(3pt)};
			\draw[fill=none]{(5,1) circle(3pt)};
			\draw[fill=none]{(0,2) circle(3pt)};
			\draw[fill=none]{(1,2) circle(3pt)};
			\draw[fill=none]{(2,2) circle(3pt)};
			\draw[fill=none]{(3,2) circle(3pt)};
			\draw[fill=none]{(4,2) circle(3pt)};
			\draw[fill=none]{(5,2) circle(3pt)};
			\draw[fill=none]{(0,3) circle(3pt)};
			\draw[fill=none]{(1,3) circle(3pt)};
			\draw[fill=none]{(2,3) circle(3pt)};
			\draw[fill=none]{(3,3) circle(3pt)};
			\draw[fill=none]{(4,3) circle(3pt)};
			\draw[fill=none]{(5,3) circle(3pt)};
			\draw[fill=none]{(0,4) circle(3pt)};
			\draw[fill=none]{(1,4) circle(3pt)};
			\draw[fill=none]{(2,4) circle(3pt)};
			\draw[fill=none]{(3,4) circle(3pt)};
			\draw[fill=none]{(4,4) circle(3pt)};
			\draw[fill=none]{(5,4) circle(3pt)};
			\draw[fill=none]{(0,5) circle(3pt)};
			\draw[fill=none]{(1,5) circle(3pt)};
			\draw[fill=none]{(2,5) circle(3pt)};
			\draw[fill=none]{(3,5) circle(3pt)};
			\draw[fill=none]{(4,5) circle(3pt)};
			\draw[fill=none]{(5,5) circle(3pt)};
			
			\draw[fill=none]{(1,-7) circle(3pt)};
			\draw[fill=none]{(2,-7) circle(3pt)};
			\draw[fill=none]{(3,-7) circle(3pt)};
			\draw[fill=none]{(4,-7) circle(3pt)};
			\draw[fill=none]{(5,-7) circle(3pt)};
			\draw[fill=none]{(6,-7) circle(3pt)};
			\draw[fill=none]{(1,-2) circle(3pt)};
			\draw[fill=none]{(2,-2) circle(3pt)};
			\draw[fill=none]{(3,-2) circle(3pt)};
			\draw[fill=none]{(4,-2) circle(3pt)};
			\draw[fill=none]{(5,-2) circle(3pt)};
			\draw[fill=none]{(6,-2) circle(3pt)};
			\draw[fill=none]{(1,-3) circle(3pt)};
			\draw[fill=none]{(2,-3) circle(3pt)};
			\draw[fill=none]{(3,-3) circle(3pt)};
			\draw[fill=none]{(4,-3) circle(3pt)};
			\draw[fill=none]{(5,-3) circle(3pt)};
			\draw[fill=none]{(6,-3) circle(3pt)};
			\draw[fill=none]{(1,-4) circle(3pt)};
			\draw[fill=none]{(2,-4) circle(3pt)};
			\draw[fill=none]{(3,-4) circle(3pt)};
			\draw[fill=none]{(4,-4) circle(3pt)};
			\draw[fill=none]{(5,-4) circle(3pt)};
			\draw[fill=none]{(6,-4) circle(3pt)};
			\draw[fill=none]{(1,-5) circle(3pt)};
			\draw[fill=none]{(2,-5) circle(3pt)};
			\draw[fill=none]{(3,-5) circle(3pt)};
			\draw[fill=none]{(4,-5) circle(3pt)};
			\draw[fill=none]{(5,-5) circle(3pt)};
			\draw[fill=none]{(6,-5) circle(3pt)};
			\draw[fill=none]{(1,-6) circle(3pt)};
			\draw[fill=none]{(2,-6) circle(3pt)};
			\draw[fill=none]{(3,-6) circle(3pt)};
			\draw[fill=none]{(4,-6) circle(3pt)};
			\draw[fill=none]{(5,-6) circle(3pt)};
			\draw[fill=none]{(6,-6) circle(3pt)};
			
			\draw[fill=none]{(-1,-7) circle(3pt)};
			\draw[fill=none]{(-2,-7) circle(3pt)};
			\draw[fill=none]{(-3,-7) circle(3pt)};
			\draw[fill=none]{(-4,-7) circle(3pt)};
			\draw[fill=none]{(-5,-7) circle(3pt)};
			\draw[fill=none]{(-6,-7) circle(3pt)};
			\draw[fill=none]{(-6,-1) circle(3pt)};
			\draw[fill=none]{(-1,-1) circle(3pt)};
			\draw[fill=none]{(-2,-1) circle(3pt)};
			\draw[fill=none]{(-3,-1) circle(3pt)};
			\draw[fill=none]{(-4,-1) circle(3pt)};
			\draw[fill=none]{(-5,-1) circle(3pt)};
			\draw[fill=none]{(-6,-2) circle(3pt)};
			\draw[fill=none]{(-1,-2) circle(3pt)};
			\draw[fill=none]{(-2,-2) circle(3pt)};
			\draw[fill=none]{(-3,-2) circle(3pt)};
			\draw[fill=none]{(-4,-2) circle(3pt)};
			\draw[fill=none]{(-5,-2) circle(3pt)};
			\draw[fill=none]{(-6,-3) circle(3pt)};
			\draw[fill=none]{(-1,-3) circle(3pt)};
			\draw[fill=none]{(-2,-3) circle(3pt)};
			\draw[fill=none]{(-3,-3) circle(3pt)};
			\draw[fill=none]{(-4,-3) circle(3pt)};
			\draw[fill=none]{(-5,-3) circle(3pt)};
			\draw[fill=none]{(-6,-4) circle(3pt)};
			\draw[fill=none]{(-1,-4) circle(3pt)};
			\draw[fill=none]{(-2,-4) circle(3pt)};
			\draw[fill=none]{(-3,-4) circle(3pt)};
			\draw[fill=none]{(-4,-4) circle(3pt)};
			\draw[fill=none]{(-5,-4) circle(3pt)};
			\draw[fill=none]{(-6,-5) circle(3pt)};
			\draw[fill=none]{(-1,-5) circle(3pt)};
			\draw[fill=none]{(-2,-5) circle(3pt)};
			\draw[fill=none]{(-3,-5) circle(3pt)};
			\draw[fill=none]{(-4,-5) circle(3pt)};
			\draw[fill=none]{(-5,-5) circle(3pt)};
			
			\draw[gray,thick](-6.5,-1)--(-6.1,-1);
			\draw[gray,thick](0.1,0)--(0.9,0);
			\draw[gray,thick](1.1,0)--(1.9,0);
			\draw[gray,thick](2.1,0)--(2.9,0);
			\draw[gray,thick](3.1,0)--(3.9,0);
			\draw[gray,thick](4.1,0)--(4.9,0);
			\draw[gray,thick](5.1,0)--(5.9,0);
			\draw[gray,thick](6.1,0)--(6.5,0);
			
			\draw[gray,thick](1.1,-2)--(1.9,-2);
			\draw[gray,thick](2.1,-2)--(2.9,-2);
			\draw[gray,thick](3.1,-2)--(3.9,-2);
			\draw[gray,thick](4.1,-2)--(4.9,-2);
			\draw[gray,thick](5.1,-2)--(5.9,-2);
			\draw[gray,thick](6.1,-2)--(6.5,-2);
			
			\draw[gray,thick](1,-2.1)--(1,-2.9);
			\draw[gray,thick](1,-3.1)--(1,-3.9);
			\draw[gray,thick](1,-4.1)--(1,-4.9);
			\draw[gray,thick](1,-5.1)--(1,-5.9);
			\draw[gray,thick](1,-6.1)--(1,-6.9);
			\draw[gray,thick](1,-7.1)--(1,-7.5);
			
			\draw[gray,thick](-1,-7.1)--(-1,-7.5);
			\draw[gray,thick](-1,-1.1)--(-1,-1.9);
			\draw[gray,thick](-1,-2.1)--(-1,-2.9);
			\draw[gray,thick](-1,-3.1)--(-1,-3.9);
			\draw[gray,thick](-1,-4.1)--(-1,-4.9);
			\draw[gray,thick](-1,-5.1)--(-1,-5.9);
			\draw[gray,thick](-1,-6.1)--(-1,-6.9);
			
			\draw[->, line width=1pt](0,0.9)--(0,0.1);
			\draw[->, line width=1pt](1,0.9)--(1,0.1);
			\draw[->, line width=1pt](2,0.9)--(2,0.1);
			\draw[->, line width=1pt](3,0.9)--(3,0.1);
			\draw[->, line width=1pt](4,0.9)--(4,0.1);
			\draw[->, line width=1pt](5,0.9)--(5,0.1);
			\draw[->, line width=1pt](6,0.9)--(6,0.1);
			\draw[->, line width=1pt](0,1.9)--(0,1.1);
			\draw[->, line width=1pt](1,1.9)--(1,1.1);
			\draw[->, line width=1pt](2,1.9)--(2,1.1);
			\draw[->, line width=1pt](3,1.9)--(3,1.1);
			\draw[->, line width=1pt](4,1.9)--(4,1.1);
			\draw[->, line width=1pt](5,1.9)--(5,1.1);
			\draw[->, line width=1pt](6,1.9)--(6,1.1);
			\draw[->, line width=1pt](0,2.9)--(0,2.1);
			\draw[->, line width=1pt](1,2.9)--(1,2.1);
			\draw[->, line width=1pt](2,2.9)--(2,2.1);
			\draw[->, line width=1pt](3,2.9)--(3,2.1);
			\draw[->, line width=1pt](4,2.9)--(4,2.1);
			\draw[->, line width=1pt](5,2.9)--(5,2.1);
			\draw[->, line width=1pt](6,2.9)--(6,2.1);
			\draw[->, line width=1pt](0,3.9)--(0,3.1);
			\draw[->, line width=1pt](1,3.9)--(1,3.1);
			\draw[->, line width=1pt](2,3.9)--(2,3.1);
			\draw[->, line width=1pt](3,3.9)--(3,3.1);
			\draw[->, line width=1pt](4,3.9)--(4,3.1);
			\draw[->, line width=1pt](5,3.9)--(5,3.1);
			\draw[->, line width=1pt](6,3.9)--(6,3.1);
			\draw[->, line width=1pt](0,4.9)--(0,4.1);
			\draw[->, line width=1pt](1,4.9)--(1,4.1);
			\draw[->, line width=1pt](2,4.9)--(2,4.1);
			\draw[->, line width=1pt](3,4.9)--(3,4.1);
			\draw[->, line width=1pt](4,4.9)--(4,4.1);
			\draw[->, line width=1pt](5,4.9)--(5,4.1);
			\draw[->, line width=1pt](6,4.9)--(6,4.1);
			\draw[->, line width=1pt](0,5.5)--(0,5.1);
			\draw[->, line width=1pt](1,5.5)--(1,5.1);
			\draw[->, line width=1pt](2,5.5)--(2,5.1);
			\draw[->, line width=1pt](3,5.5)--(3,5.1);
			\draw[->, line width=1pt](4,5.5)--(4,5.1);
			\draw[->, line width=1pt](5,5.5)--(5,5.1);
			\draw[->, line width=1pt](6,5.5)--(6,5.1);

			\draw[->, line width=1pt](-5.1,-7)--(-5.9,-7);
			\draw[->, line width=1pt](-1.1,-7)--(-1.9,-7);
			\draw[->, line width=1pt](-2.1,-7)--(-2.9,-7);
			\draw[->, line width=1pt](-3.1,-7)--(-3.9,-7);
			\draw[->, line width=1pt](-4.1,-7)--(-4.9,-7);
			\draw[line width=1pt](-6.1,-7)--(-6.5,-7);
			\draw[line width=1pt](-6.1,-2)--(-6.5,-2);
			\draw[line width=1pt](-6.1,-3)--(-6.5,-3);
			\draw[line width=1pt](-6.1,-4)--(-6.5,-4);
			\draw[line width=1pt](-6.1,-5)--(-6.5,-5);
			\draw[line width=1pt](-6.1,-6)--(-6.5,-6);
			\draw[->, line width=1pt](-5.1,-1)--(-5.9,-1);
			\draw[->, line width=1pt](-1.1,-1)--(-1.9,-1);
			\draw[->, line width=1pt](-2.1,-1)--(-2.9,-1);
			\draw[->, line width=1pt](-3.1,-1)--(-3.9,-1);
			\draw[->, line width=1pt](-4.1,-1)--(-4.9,-1);
			\draw[->, line width=1pt](-5.1,-2)--(-5.9,-2);
			\draw[->, line width=1pt](-1.1,-2)--(-1.9,-2);
			\draw[->, line width=1pt](-2.1,-2)--(-2.9,-2);
			\draw[->, line width=1pt](-3.1,-2)--(-3.9,-2);
			\draw[->, line width=1pt](-4.1,-2)--(-4.9,-2);
			\draw[->, line width=1pt](-5.1,-3)--(-5.9,-3);
			\draw[->, line width=1pt](-1.1,-3)--(-1.9,-3);
			\draw[->, line width=1pt](-2.1,-3)--(-2.9,-3);
			\draw[->, line width=1pt](-3.1,-3)--(-3.9,-3);
			\draw[->, line width=1pt](-4.1,-3)--(-4.9,-3);
			\draw[->, line width=1pt](-5.1,-4)--(-5.9,-4);
			\draw[->, line width=1pt](-1.1,-4)--(-1.9,-4);
			\draw[->, line width=1pt](-2.1,-4)--(-2.9,-4);
			\draw[->, line width=1pt](-3.1,-4)--(-3.9,-4);
			\draw[->, line width=1pt](-4.1,-4)--(-4.9,-4);
			\draw[->, line width=1pt](-5.1,-5)--(-5.9,-5);
			\draw[->, line width=1pt](-1.1,-5)--(-1.9,-5);
			\draw[->, line width=1pt](-2.1,-5)--(-2.9,-5);
			\draw[->, line width=1pt](-3.1,-5)--(-3.9,-5);
			\draw[->, line width=1pt](-4.1,-5)--(-4.9,-5);
			\draw[->, line width=1pt](-5.1,-6)--(-5.9,-6);
			\draw[->, line width=1pt](-1.1,-6)--(-1.9,-6);
			\draw[->, line width=1pt](-2.1,-6)--(-2.9,-6);
			\draw[->, line width=1pt](-3.1,-6)--(-3.9,-6);
			\draw[->, line width=1pt](-4.1,-6)--(-4.9,-6);

			\draw[->, line width=1pt](1.9,-2.1)--(1.1,-2.9);
			\draw[->, line width=1pt](2.9,-2.1)--(2.1,-2.9);
			\draw[->, line width=1pt](3.9,-2.1)--(3.1,-2.9);
			\draw[->, line width=1pt](4.9,-2.1)--(4.1,-2.9);
			\draw[->, line width=1pt](5.9,-2.1)--(5.1,-2.9);
			\draw[->, line width=1pt](6.5,-2.5)--(6.1,-2.9);
			\draw[->, line width=1pt](1.9,-3.1)--(1.1,-3.9);
			\draw[->, line width=1pt](2.9,-3.1)--(2.1,-3.9);
			\draw[->, line width=1pt](3.9,-3.1)--(3.1,-3.9);
			\draw[->, line width=1pt](4.9,-3.1)--(4.1,-3.9);
			\draw[->, line width=1pt](5.9,-3.1)--(5.1,-3.9);
			\draw[->, line width=1pt](6.5,-3.5)--(6.1,-3.9);
			\draw[->, line width=1pt](1.9,-4.1)--(1.1,-4.9);
			\draw[->, line width=1pt](2.9,-4.1)--(2.1,-4.9);
			\draw[->, line width=1pt](3.9,-4.1)--(3.1,-4.9);
			\draw[->, line width=1pt](4.9,-4.1)--(4.1,-4.9);
			\draw[->, line width=1pt](5.9,-4.1)--(5.1,-4.9);
			\draw[->, line width=1pt](6.5,-4.5)--(6.1,-4.9);
			\draw[->, line width=1pt](1.9,-5.1)--(1.1,-5.9);
			\draw[->, line width=1pt](2.9,-5.1)--(2.1,-5.9);
			\draw[->, line width=1pt](3.9,-5.1)--(3.1,-5.9);
			\draw[->, line width=1pt](4.9,-5.1)--(4.1,-5.9);
			\draw[->, line width=1pt](5.9,-5.1)--(5.1,-5.9);
			\draw[->, line width=1pt](6.5,-5.5)--(6.1,-5.9);
			\draw[->, line width=1pt](1.9,-6.1)--(1.1,-6.9);
			\draw[->, line width=1pt](2.9,-6.1)--(2.1,-6.9);
			\draw[->, line width=1pt](3.9,-6.1)--(3.1,-6.9);
			\draw[->, line width=1pt](4.9,-6.1)--(4.1,-6.9);
			\draw[->, line width=1pt](5.9,-6.1)--(5.1,-6.9);
			\draw[->, line width=1pt](6.5,-6.5)--(6.1,-6.9);
			\draw[line width=1pt](1.9,-7.1)--(1.6,-7.4);
			\draw[line width=1pt](2.9,-7.1)--(2.6,-7.4);
			\draw[line width=1pt](3.9,-7.1)--(3.6,-7.4);
			\draw[line width=1pt](4.9,-7.1)--(4.6,-7.4);
			\draw[line width=1pt](5.9,-7.1)--(5.6,-7.4);

			\draw[->, line width=1.5pt](0.9,-0.1)--(-0.9,-1.9);
			\draw[->, line width=1.5pt](1.9,-0.1)--(-0.9,-2.9);
			\draw[->, line width=1.5pt](2.9,-0.1)--(1.1,-1.9);
			\draw[->, line width=1.5pt](3.9,-0.1)--(2.1,-1.9);
			\draw[->, line width=1.5pt](4.9,-0.1)--(3.1,-1.9);
			\draw[->, line width=1.5pt](5.9,-0.1)--(4.1,-1.9);
			\draw[->, line width=1.5pt](6.5,-0.5)--(5.1,-1.9);
			\draw[->, line width=1.5pt](6.5,-1.5)--(6.1,-1.9);
			
			\draw[->, line width=1.5pt](0.9,-2.1)--(-0.9,-3.9);
			\draw[->, line width=1.5pt](0.9,-3.1)--(-0.9,-4.9);
			\draw[->, line width=1.5pt](0.9,-4.1)--(-0.9,-5.9);
			\draw[->, line width=1.5pt](0.9,-5.1)--(-0.9,-6.9);
			\draw[ line width=1.5pt](0.9,-6.1)--(-0.3,-7.3);
			\draw[line width=1.5pt](0.9,-7.1)--(0.5,-7.5);

			
			\node at (-7,-1) {$m$};
			\node at (7,0) {$m$};
			\node at (7,-2) {$m$};
			\node at (0,5.8) {$n$};
			\node at (-1,-7.8) {$n$};
			\node at (1,-7.8) {$n$};
			\node at (-0.2,-0.8) {$11$};
			\node at (0.3,-1.4) {$7$};
			\node at (0,-2.7) {$4$};
			\node at (0,-3.7) {$4$};
			\node at (0,-4.7) {$4$};
			\node at (0,-5.7) {$4$};
			\node at (1.8,-0.9) {$4$};
			\node at (2.8,-0.9) {$4$};
			\node at (3.8,-0.9) {$4$};
			\node at (4.8,-0.9) {$4$};
			\node at (6.5,6){$M(m,n,0,0)$};
			\node at (-6.5,-7.5){$M(0,0,m,n)$};
			\node at (7,-7.8){$M(m,0,0,n)$};
			\node[rotate=45] at (-0.5,-0.5){$=$};
			\end{tikzpicture}$$}
	\end{center}

	\caption{\label{E510morphisms} All nonzero morphisms between finite generalized Verma modules for $E(5,10)$ which are not compositions of other morphisms. Morphisms of degree $>1$ are labelled by their degree.} 
\end{figure}

\newpage

\newpage
\section{The singular vector in $M(0,0,0,n+3)$ of degree 4 and weight $(0,0,0,n)$}\label{w4E}
The following is the singular vector (which has been found and checked with the aid of a computer) in $M(0,0,0,n+3)$ of degree 4 and weight $(0,0,0,n)$. Here we denote $x_i^*$ by $f_i$ and we omit all tensor product symbols. 
\begin{align*}
w[4_E]&=d_{12}d_{13}d_{14}d_{15}f_1^3f_5^{n}
+  d_{12}d_{14}d_{15}d_{23}f_1^2 f_2f_5^{n}
+  d_{13}d_{14}d_{15}d_{23}f_1^2 f_3  f_5^{n}
- d_{12}d_{13}d_{15}d_{24}f_1^2 f_2f_5^{n}\\&
+  d_{13}d_{14}d_{15}d_{24}f_1^2 f_4f_5^{n}
+  d_{12}d_{15}d_{23}d_{24}f_1 f_2^2 f_5^{n}
+  d_{13}d_{15}d_{23}d_{24}f_1 f_2 f_3 f_5^{n}
+  d_{14}d_{15}d_{23}d_{24}f_1 f_2 f_4 f_5^{n}\\&
+  d_{12}d_{13}d_{14}d_{25}f_1^2 f_2f_5^{n}
+  d_{13}d_{14}d_{15}d_{25}f_1^2 f_5^{n+1}
- d_{12}d_{14}d_{23}d_{25}f_1 f_2^2 f_5^{n}
- d_{13}d_{14}d_{23}d_{25}f_1 f_2 f_3 f_5^{n}
\end{align*}
\begin{align*}
&+  d_{14}d_{15}d_{23}d_{25}f_1 f_2 f_5^{n+1}
+  d_{12}d_{13}d_{24}d_{25}f_1 f_2^2 f_5^{n}
- d_{13}d_{14}d_{24}d_{25}f_1 f_2 f_4 f_5^{n}
- d_{13}d_{15}d_{24}d_{25}f_1 f_2 f_5^{n+1}\\&
+  d_{12}d_{23}d_{24}d_{25}f_2^3  f_5^{n}
+  d_{13}d_{23}d_{24}d_{25}f_2^2 f_3  f_5^{n}
+  d_{14}d_{23}d_{24}d_{25}f_2^2 f_4  f_5^{n}
+  d_{15}d_{23}d_{24}d_{25}f_2^2  f_5^{n+1}\\&
- d_{12}d_{13}d_{15}d_{34}f_1^2 f_3  f_5^{n}
- d_{12}d_{14}d_{15}d_{34}f_1^2 f_4f_5^{n}
+  d_{12}d_{15}d_{23}d_{34}f_1 f_2 f_3 f_5^{n}
+  d_{13}d_{15}d_{23}d_{34}f_1 f_3^2 f_5^{n}\\&
+  d_{14}d_{15}d_{23}d_{34}f_1 f_3 f_4 f_5^{n}
+  d_{12}d_{15}d_{24}d_{34}f_1 f_2 f_4 f_5^{n}
+  d_{13}d_{15}d_{24}d_{34}f_1 f_3 f_4 f_5^{n}
+  d_{14}d_{15}d_{24}d_{34}f_1 f_4^2  f_5^{n}\\&
- d_{12}d_{13}d_{25}d_{34}f_1 f_2 f_3 f_5^{n}
- d_{12}d_{14}d_{25}d_{34}f_1 f_2 f_4 f_5^{n}
+  d_{13}d_{15}d_{25}d_{34}f_1 f_3  f_5^{n+1}
+  d_{14}d_{15}d_{25}d_{34}f_1 f_4  f_5^{n+1}\\&
- d_{12}d_{23}d_{25}d_{34}f_2^2 f_3  f_5^{n}
- d_{13}d_{23}d_{25}d_{34}f_2 f_3^2  f_5^{n}
- d_{14}d_{23}d_{25}d_{34}f_2 f_3 f_4  f_5^{n}
- d_{15}d_{23}d_{25}d_{34}f_2 f_3  f_5^{n+1}\\&
- d_{12}d_{24}d_{25}d_{34}f_2^2 f_4  f_5^{n}
- d_{13}d_{24}d_{25}d_{34}f_2 f_3 f_4  f_5^{n}
- d_{14}d_{24}d_{25}d_{34}f_2 f_4^2  f_5^{n}
- d_{15}d_{24}d_{25}d_{34}f_2 f_4  f_5^{n+1}\\&
+  d_{12}d_{13}d_{14}d_{35}f_1^2 f_3  f_5^{n}
- d_{12}d_{14}d_{15}d_{35}f_1^2 f_5^{n+1}
- d_{12}d_{14}d_{23}d_{35}f_1 f_2 f_3 f_5^{n}
- d_{13}d_{14}d_{23}d_{35}f_1 f_3^2 f_5^{n}\\&
+  d_{14}d_{15}d_{23}d_{35}f_1 f_3  f_5^{n+1}
+  d_{12}d_{13}d_{24}d_{35}f_1 f_2 f_3 f_5^{n}
- d_{13}d_{14}d_{24}d_{35}f_1 f_3 f_4 f_5^{n}
+  d_{12}d_{15}d_{24}d_{35}f_1 f_2 f_5^{n+1}\\&
+  d_{14}d_{15}d_{24}d_{35}f_1 f_4  f_5^{n+1}
+  d_{12}d_{23}d_{24}d_{35}f_2^2 f_3  f_5^{n}
+  d_{13}d_{23}d_{24}d_{35}f_2 f_3^2  f_5^{n}
+  d_{14}d_{23}d_{24}d_{35}f_2 f_3 f_4  f_5^{n}\\&
+  d_{15}d_{23}d_{24}d_{35}f_2 f_3  f_5^{n+1}
- d_{12}d_{14}d_{25}d_{35}f_1 f_2 f_5^{n+1}
- d_{13}d_{14}d_{25}d_{35}f_1 f_3  f_5^{n+1}
+  d_{14}d_{15}d_{25}d_{35}f_1  f_5^{n+2}\\&
- d_{12}d_{24}d_{25}d_{35}f_2^2  f_5^{n+1}
- d_{13}d_{24}d_{25}d_{35}f_2 f_3  f_5^{n+1}
- d_{14}d_{24}d_{25}d_{35}f_2 f_4  f_5^{n+1}
- d_{15}d_{24}d_{25}d_{35}f_2  f_5^{n+2}\\&
+  d_{12}d_{13}d_{34}d_{35}f_1 f_3^2 f_5^{n}
+  d_{12}d_{14}d_{34}d_{35}f_1 f_3 f_4 f_5^{n}
+  d_{12}d_{15}d_{34}d_{35}f_1 f_3  f_5^{n+1}
+  d_{12}d_{23}d_{34}d_{35}f_2 f_3^2  f_5^{n}\\&
+  d_{13}d_{23}d_{34}d_{35}f_3^3 f_5^{n}
+  d_{14}d_{23}d_{34}d_{35}f_3^2 f_4 f_5^{n}
+  d_{15}d_{23}d_{34}d_{35}f_3^2 f_5^{n+1}
+  d_{12}d_{24}d_{34}d_{35}f_2 f_3 f_4  f_5^{n}\\&
+  d_{13}d_{24}d_{34}d_{35}f_3^2 f_4 f_5^{n}
+  d_{14}d_{24}d_{34}d_{35}f_3 f_4^2 f_5^{n}
+  d_{15}d_{24}d_{34}d_{35}f_3 f_4 f_5^{n+1}
+  d_{12}d_{25}d_{34}d_{35}f_2 f_3  f_5^{n+1}\\&
+  d_{13}d_{25}d_{34}d_{35}f_3^2 f_5^{n+1}
+  d_{14}d_{25}d_{34}d_{35}f_3 f_4 f_5^{n+1}
+  d_{15}d_{25}d_{34}d_{35}f_3  f_5^{n+2}
+  d_{12}d_{13}d_{14}d_{45}f_1^2 f_4f_5^{n}\\&
+  d_{12}d_{13}d_{15}d_{45}f_1^2 f_5^{n+1}
- d_{12}d_{14}d_{23}d_{45}f_1 f_2 f_4 f_5^{n}
- d_{13}d_{14}d_{23}d_{45}f_1 f_3 f_4 f_5^{n}
- d_{12}d_{15}d_{23}d_{45}f_1 f_2 f_5^{n+1}\\&
- d_{13}d_{15}d_{23}d_{45}f_1 f_3  f_5^{n+1}
+  d_{12}d_{13}d_{24}d_{45}f_1 f_2 f_4 f_5^{n}
- d_{13}d_{14}d_{24}d_{45}f_1 f_4^2  f_5^{n}
- d_{13}d_{15}d_{24}d_{45}f_1 f_4  f_5^{n+1}\\&
+  d_{12}d_{23}d_{24}d_{45}f_2^2 f_4  f_5^{n}
+  d_{13}d_{23}d_{24}d_{45}f_2 f_3 f_4  f_5^{n}
+  d_{14}d_{23}d_{24}d_{45}f_2 f_4^2  f_5^{n}
+  d_{15}d_{23}d_{24}d_{45}f_2 f_4  f_5^{n+1}\\&
+  d_{12}d_{13}d_{25}d_{45}f_1 f_2 f_5^{n+1}
- d_{13}d_{14}d_{25}d_{45}f_1 f_4  f_5^{n+1}
- d_{13}d_{15}d_{25}d_{45}f_1  f_5^{n+2}
+  d_{12}d_{23}d_{25}d_{45}f_2^2  f_5^{n+1}\\&
+  d_{13}d_{23}d_{25}d_{45}f_2 f_3  f_5^{n+1}
+  d_{14}d_{23}d_{25}d_{45}f_2 f_4  f_5^{n+1}
+  d_{15}d_{23}d_{25}d_{45}f_2  f_5^{n+2}
+  d_{12}d_{13}d_{34}d_{45}f_1 f_3 f_4 f_5^{n}\\&
+  d_{12}d_{14}d_{34}d_{45}f_1 f_4^2  f_5^{n}
+  d_{12}d_{15}d_{34}d_{45}f_1 f_4  f_5^{n+1}
+  d_{12}d_{23}d_{34}d_{45}f_2 f_3 f_4  f_5^{n}
+  d_{13}d_{23}d_{34}d_{45}f_3^2 f_4 f_5^{n}\\&
+  d_{14}d_{23}d_{34}d_{45}f_3 f_4^2 f_5^{n}
+  d_{15}d_{23}d_{34}d_{45}f_3 f_4 f_5^{n+1}
+  d_{12}d_{24}d_{34}d_{45}f_2 f_4^2  f_5^{n}
+  d_{13}d_{24}d_{34}d_{45}f_3 f_4^2 f_5^{n}\\&
+  d_{14}d_{24}d_{34}d_{45}f_4^3 f_5^{n}
+  d_{15}d_{24}d_{34}d_{45}f_4^2 f_5^{n+1}
+  d_{12}d_{25}d_{34}d_{45}f_2 f_4  f_5^{n+1}
+  d_{13}d_{25}d_{34}d_{45}f_3 f_4 f_5^{n+1}\\&
+  d_{14}d_{25}d_{34}d_{45}f_4^2 f_5^{n+1}
+  d_{15}d_{25}d_{34}d_{45}f_4  f_5^{n+2}
+  d_{12}d_{13}d_{35}d_{45}f_1 f_3  f_5^{n+1}
+  d_{12}d_{14}d_{35}d_{45}f_1 f_4  f_5^{n+1}\\&
+  d_{12}d_{15}d_{35}d_{45}f_1  f_5^{n+2}
+  d_{12}d_{23}d_{35}d_{45}f_2 f_3  f_5^{n+1}
+  d_{13}d_{23}d_{35}d_{45}f_3^2 f_5^{n+1}
+  d_{14}d_{23}d_{35}d_{45}f_3 f_4 f_5^{n+1}\\&
+  d_{15}d_{23}d_{35}d_{45}f_3  f_5^{n+2}
+  d_{12}d_{24}d_{35}d_{45}f_2 f_4  f_5^{n+1}
+  d_{13}d_{24}d_{35}d_{45}f_3 f_4 f_5^{n+1}
+  d_{14}d_{24}d_{35}d_{45}f_4^2 f_5^{n+1}\\&
+  d_{15}d_{24}d_{35}d_{45}f_4  f_5^{n+2}
+  d_{12}d_{25}d_{35}d_{45}f_2  f_5^{n+2}
+  d_{13}d_{25}d_{35}d_{45}f_3  f_5^{n+2}
+  d_{14}d_{25}d_{35}d_{45}f_4  f_5^{n+2}
\end{align*}
\begin{align*}
&+  d_{15}d_{25}d_{35}d_{45}f_5^{n+3}
+ \de_1 d_{12}d_{13}f_1 f_2 f_3 f_5^{n}
+ \de_1 d_{12}d_{14}f_1 f_2 f_4 f_5^{n}
+ \de_1 d_{12}d_{15}f_1 f_2 f_5^{n+1}\\&
+ \de_1 d_{12}d_{23}f_2^2 f_3  f_5^{n}
+ \de_1 d_{13}d_{23}f_2 f_3^2  f_5^{n}
+ \de_1 d_{14}d_{23}f_2 f_3 f_4  f_5^{n}
+ \de_1 d_{15}d_{23}f_2 f_3  f_5^{n+1}\\&
+ \de_1 d_{12}d_{24}f_2^2 f_4  f_5^{n}
+ \de_1 d_{13}d_{24}f_2 f_3 f_4  f_5^{n}
+ \de_1 d_{14}d_{24}f_2 f_4^2  f_5^{n}
+ \de_1 d_{15}d_{24}f_2 f_4  f_5^{n+1}\\&
+ \de_1 d_{12}d_{25}f_2^2  f_5^{n+1}
+ \de_1 d_{13}d_{25}f_2 f_3  f_5^{n+1}
+ \de_1 d_{14}d_{25}f_2 f_4  f_5^{n+1}
+ \de_1 d_{15}d_{25}f_2  f_5^{n+2}\\&
- \de_2 d_{12}d_{13}f_1^2 f_3  f_5^{n}
- \de_2 d_{12}d_{14}f_1^2 f_4f_5^{n}
- \de_2 d_{12}d_{15}f_1^2 f_5^{n+1}
- \de_2 d_{12}d_{23}f_1 f_2 f_3 f_5^{n}\\&
- \de_2 d_{13}d_{23}f_1 f_3^2 f_5^{n}
- \de_2 d_{14}d_{23}f_1 f_3 f_4 f_5^{n}
- \de_2 d_{15}d_{23}f_1 f_3  f_5^{n+1}
- \de_2 d_{12}d_{24}f_1 f_2 f_4 f_5^{n}\\&
- \de_2 d_{13}d_{24}f_1 f_3 f_4 f_5^{n}
- \de_2 d_{14}d_{24}f_1 f_4^2  f_5^{n}
- \de_2 d_{15}d_{24}f_1 f_4  f_5^{n+1}
- \de_2 d_{12}d_{25}f_1 f_2 f_5^{n+1}\\&
- \de_2 d_{13}d_{25}f_1 f_3  f_5^{n+1}
- \de_2 d_{14}d_{25}f_1 f_4  f_5^{n+1}
- \de_2 d_{15}d_{25}f_1  f_5^{n+2}
+ \de_3 d_{12}d_{13}f_1^2 f_2f_5^{n}\\&
- \de_3 d_{13}d_{14}f_1^2 f_4f_5^{n}
- \de_3 d_{13}d_{15}f_1^2 f_5^{n+1}
+ \de_3 d_{12}d_{23}f_1 f_2^2 f_5^{n}
+ \de_3 d_{13}d_{23}f_1 f_2 f_3 f_5^{n}\\&
+ \de_3 d_{14}d_{23}f_1 f_2 f_4 f_5^{n}
+ \de_3 d_{15}d_{23}f_1 f_2 f_5^{n+1}
- \de_3 d_{12}d_{34}f_1 f_2 f_4 f_5^{n}
- \de_3 d_{13}d_{34}f_1 f_3 f_4 f_5^{n}\\&
- \de_3 d_{14}d_{34}f_1 f_4^2  f_5^{n}
- \de_3 d_{15}d_{34}f_1 f_4  f_5^{n+1}
- \de_3 d_{12}d_{35}f_1 f_2 f_5^{n+1}
- \de_3 d_{13}d_{35}f_1 f_3  f_5^{n+1}\\&
- \de_3 d_{14}d_{35}f_1 f_4  f_5^{n+1}
- \de_3 d_{15}d_{35}f_1  f_5^{n+2}
+ \de_4 d_{12}d_{14}f_1^2 f_2f_5^{n}
+ \de_4 d_{13}d_{14}f_1^2 f_3  f_5^{n}\\&
- \de_4 d_{14}d_{15}f_1^2 f_5^{n+1}
+ \de_4 d_{12}d_{24}f_1 f_2^2 f_5^{n}
+ \de_4 d_{13}d_{24}f_1 f_2 f_3 f_5^{n}
+ \de_4 d_{14}d_{24}f_1 f_2 f_4 f_5^{n}\\&
+ \de_4 d_{15}d_{24}f_1 f_2 f_5^{n+1}
+ \de_4 d_{12}d_{34}f_1 f_2 f_3 f_5^{n}
+ \de_4 d_{13}d_{34}f_1 f_3^2 f_5^{n}
+ \de_4 d_{14}d_{34}f_1 f_3 f_4 f_5^{n}\\&
+ \de_4 d_{15}d_{34}f_1 f_3  f_5^{n+1}
- \de_4 d_{12}d_{45}f_1 f_2 f_5^{n+1}
- \de_4 d_{13}d_{45}f_1 f_3  f_5^{n+1}
- \de_4 d_{14}d_{45}f_1 f_4  f_5^{n+1}\\&
- \de_4 d_{15}d_{45}f_1  f_5^{n+2}
+ \de_5 d_{12}d_{15}f_1^2 f_2f_5^{n}
+ \de_5 d_{13}d_{15}f_1^2 f_3  f_5^{n}
+ \de_5 d_{14}d_{15}f_1^2 f_4f_5^{n}\\&
+ \de_5 d_{12}d_{25}f_1 f_2^2 f_5^{n}
+ \de_5 d_{13}d_{25}f_1 f_2 f_3 f_5^{n}
+ \de_5 d_{14}d_{25}f_1 f_2 f_4 f_5^{n}
+ \de_5 d_{15}d_{25}f_1 f_2 f_5^{n+1}\\&
+ \de_5 d_{12}d_{35}f_1 f_2 f_3 f_5^{n}
+ \de_5 d_{13}d_{35}f_1 f_3^2 f_5^{n}
+ \de_5 d_{14}d_{35}f_1 f_3 f_4 f_5^{n}
+ \de_5 d_{15}d_{35}f_1 f_3  f_5^{n+1}\\&
+ \de_5 d_{12}d_{45}f_1 f_2 f_4 f_5^{n}
+ \de_5 d_{13}d_{45}f_1 f_3 f_4 f_5^{n}
+ \de_5 d_{14}d_{45}f_1 f_4^2  f_5^{n}
+ \de_5 d_{15}d_{45}f_1 f_4  f_5^{n+1}.
\end{align*}

\medskip

We recall that the construction of this vector is also sketched by Rudakov in \cite[\S 4]{R}.

\begin{figure}
	\begin{center}
		\scalebox{1}{	$$
			\begin{tikzpicture}
			
			\draw[fill=none]{(6,1) circle(3pt)};
			\draw[fill=none]{(6,2) circle(3pt)};
			\draw[fill=none]{(6,3) circle(3pt)};
			\draw[fill=none]{(6,4) circle(3pt)};
			\draw[fill=none]{(6,5) circle(3pt)};
			
			\draw[fill=none]{(0,1) circle(3pt)};
			\draw[fill=none]{(1,1) circle(3pt)};
			\draw[fill=none]{(2,1) circle(3pt)};
			\draw[fill=none]{(3,1) circle(3pt)};
			\draw[fill=none]{(4,1) circle(3pt)};
			\draw[fill=none]{(5,1) circle(3pt)};
			\draw[fill=none]{(0,2) circle(3pt)};
			\draw[fill=none]{(1,2) circle(3pt)};
			\draw[fill=none]{(2,2) circle(3pt)};
			\draw[fill=none]{(3,2) circle(3pt)};
			\draw[fill=none]{(4,2) circle(3pt)};
			\draw[fill=none]{(5,2) circle(3pt)};
			\draw[fill=none]{(0,3) circle(3pt)};
			\draw[fill=none]{(1,3) circle(3pt)};
			\draw[fill=none]{(2,3) circle(3pt)};
			\draw[fill=none]{(3,3) circle(3pt)};
			\draw[fill=none]{(4,3) circle(3pt)};
			\draw[fill=none]{(5,3) circle(3pt)};
			\draw[fill=none]{(0,4) circle(3pt)};
			\draw[fill=none]{(1,4) circle(3pt)};
			\draw[fill=none]{(2,4) circle(3pt)};
			\draw[fill=none]{(3,4) circle(3pt)};
			\draw[fill=none]{(4,4) circle(3pt)};
			\draw[fill=none]{(5,4) circle(3pt)};
			\draw[fill=none]{(0,5) circle(3pt)};
			\draw[fill=none]{(1,5) circle(3pt)};
			\draw[fill=none]{(2,5) circle(3pt)};
			\draw[fill=none]{(3,5) circle(3pt)};
			\draw[fill=none]{(4,5) circle(3pt)};
			\draw[fill=none]{(5,5) circle(3pt)};
			
			\draw[fill=none]{(1,-2) circle(3pt)};
			\draw[fill=none]{(2,-2) circle(3pt)};
			\draw[fill=none]{(3,-2) circle(3pt)};
			\draw[fill=none]{(4,-2) circle(3pt)};
			\draw[fill=none]{(5,-2) circle(3pt)};
			\draw[fill=none]{(1,-2) circle(3pt)};
			\draw[fill=none]{(2,-3) circle(3pt)};
			\draw[fill=none]{(3,-3) circle(3pt)};
			\draw[fill=none]{(4,-3) circle(3pt)};
			\draw[fill=none]{(5,-3) circle(3pt)};
			\draw[fill=none]{(1,-3) circle(3pt)};
			\draw[fill=none]{(2,-4) circle(3pt)};
			\draw[fill=none]{(3,-4) circle(3pt)};
			\draw[fill=none]{(4,-4) circle(3pt)};
			\draw[fill=none]{(5,-4) circle(3pt)};
			\draw[fill=none]{(1,-4) circle(3pt)};
			\draw[fill=none]{(2,-5) circle(3pt)};
			\draw[fill=none]{(3,-5) circle(3pt)};
			\draw[fill=none]{(4,-5) circle(3pt)};
			\draw[fill=none]{(5,-5) circle(3pt)};
			\draw[fill=none]{(1,-5) circle(3pt)};
			\draw[fill=none]{(2,-1) circle(3pt)};
			\draw[fill=none]{(3,-1) circle(3pt)};
			\draw[fill=none]{(4,-1) circle(3pt)};
			\draw[fill=none]{(5,-1) circle(3pt)};
			\draw[fill=none]{(1,-1) circle(3pt)};
			\draw[fill=none]{(-1,0) circle(3pt)};
			\draw[fill=none]{(-2,0) circle(3pt)};
			\draw[fill=none]{(-3,0) circle(3pt)};
			\draw[fill=none]{(-4,0) circle(3pt)};
			\draw[fill=none]{(-5,0) circle(3pt)};
			\draw[fill=none]{(-2,-6) circle(3pt)};
			\draw[fill=none]{(-3,-6) circle(3pt)};
			\draw[fill=none]{(-4,-6) circle(3pt)};
			\draw[fill=none]{(-5,-6) circle(3pt)};
			\draw[fill=none]{(-1,-6) circle(3pt)};
			\draw[fill=none]{(-1,-1) circle(3pt)};
			\draw[fill=none]{(-2,-1) circle(3pt)};
			\draw[fill=none]{(-3,-1) circle(3pt)};
			\draw[fill=none]{(-4,-1) circle(3pt)};
			\draw[fill=none]{(-5,-1) circle(3pt)};
			\draw[fill=none]{(-1,-2) circle(3pt)};
			\draw[fill=none]{(-2,-2) circle(3pt)};
			\draw[fill=none]{(-3,-2) circle(3pt)};
			\draw[fill=none]{(-4,-2) circle(3pt)};
			\draw[fill=none]{(-5,-2) circle(3pt)};
			\draw[fill=none]{(-1,-3) circle(3pt)};
			\draw[fill=none]{(-2,-3) circle(3pt)};
			\draw[fill=none]{(-3,-3) circle(3pt)};
			\draw[fill=none]{(-4,-3) circle(3pt)};
			\draw[fill=none]{(-5,-3) circle(3pt)};
			\draw[fill=none]{(-1,-4) circle(3pt)};
			\draw[fill=none]{(-2,-4) circle(3pt)};
			\draw[fill=none]{(-3,-4) circle(3pt)};
			\draw[fill=none]{(-4,-4) circle(3pt)};
			\draw[fill=none]{(-5,-4) circle(3pt)};
			\draw[fill=none]{(-1,-5) circle(3pt)};
			\draw[fill=none]{(-2,-5) circle(3pt)};
			\draw[fill=none]{(-3,-5) circle(3pt)};
			\draw[fill=none]{(-4,-5) circle(3pt)};
			\draw[fill=none]{(-5,-5) circle(3pt)};
			
			\draw[gray,thick](-5.5,0)--(-5.1,0);
			\draw[gray,thick](0,0)--(6.5,0);
			\draw[gray,thick](0,0)--(0,-6);
			\draw[gray, thick](0,0.9)--(0,0);
			\draw[gray, thick](-0.9,0)--(0,0);
			\draw[->, line width=1pt] (2,0.9) to [out=270,in=45] (1.09,-0.91);
			\draw[->, line width=1pt](3,0.9) to [out=270, in=45] (2.09,-0.91);
			\draw[->, line width=1pt](4,0.9) to [out=270, in=45] (3.09,-0.91);
			\draw[->, line width=1pt](5,0.9) to [out=270, in=45] (4.09,-0.91);
			\draw[->, line width=1pt](6,0.9) to [out=270, in=45] (5.09,-0.91);
			\draw[->, line width=1pt](0,1.9)--(0,1.1);
			\draw[->, line width=1pt](1,1.9)--(1,1.1);
			\draw[->, line width=1pt](2,1.9)--(2,1.1);
			\draw[->, line width=1pt](3,1.9)--(3,1.1);
			\draw[->, line width=1pt](4,1.9)--(4,1.1);
			\draw[->, line width=1pt](5,1.9)--(5,1.1);
			\draw[->, line width=1pt](6,1.9)--(6,1.1);
			\draw[->, line width=1pt](0,2.9)--(0,2.1);
			\draw[->, line width=1pt](1,2.9)--(1,2.1);
			\draw[->, line width=1pt](2,2.9)--(2,2.1);
			\draw[->, line width=1pt](3,2.9)--(3,2.1);
			\draw[->, line width=1pt](4,2.9)--(4,2.1);
			\draw[->, line width=1pt](5,2.9)--(5,2.1);
			\draw[->, line width=1pt](6,2.9)--(6,2.1);
			\draw[->, line width=1pt](0,3.9)--(0,3.1);
			\draw[->, line width=1pt](1,3.9)--(1,3.1);
			\draw[->, line width=1pt](2,3.9)--(2,3.1);
			\draw[->, line width=1pt](3,3.9)--(3,3.1);
			\draw[->, line width=1pt](4,3.9)--(4,3.1);
			\draw[->, line width=1pt](5,3.9)--(5,3.1);
			\draw[->, line width=1pt](6,3.9)--(6,3.1);
			\draw[->, line width=1pt](0,4.9)--(0,4.1);
			\draw[->, line width=1pt](1,4.9)--(1,4.1);
			\draw[->, line width=1pt](2,4.9)--(2,4.1);
			\draw[->, line width=1pt](3,4.9)--(3,4.1);
			\draw[->, line width=1pt](4,4.9)--(4,4.1);
			\draw[->, line width=1pt](5,4.9)--(5,4.1);
			\draw[->, line width=1pt](6,4.9)--(6,4.1);
			\draw[->, line width=1pt](0,5.5)--(0,5.1);
			\draw[->, line width=1pt](1,5.5)--(1,5.1);
			\draw[->, line width=1pt](2,5.5)--(2,5.1);
			\draw[->, line width=1pt](3,5.5)--(3,5.1);
			\draw[->, line width=1pt](4,5.5)--(4,5.1);
			\draw[->, line width=1pt](5,5.5)--(5,5.1);
			\draw[->, line width=1pt](6,5.5)--(6,5.1);

			\draw[->, line width=1pt](-1.1,0)--(-1.9,0);
			\draw[->, line width=1pt](-2.1,0)--(-2.9,0);
			\draw[->, line width=1pt](-3.1,0)--(-3.9,0);
			\draw[->, line width=1pt](-4.1,0)--(-4.9,0);
			\draw[line width=1pt](-5.1,-1)--(-5.5,-1);
			\draw[line width=1pt](-5.1,-2)--(-5.5,-2);
			\draw[line width=1pt](-5.1,-3)--(-5.5,-3);
			\draw[line width=1pt](-5.1,-4)--(-5.5,-4);
			\draw[line width=1pt](-5.1,-5)--(-5.5,-5);
			\draw[line width=1pt](-5.1,-6)--(-5.5,-6);
			\draw[->, line width=1pt](-1.1,-1)--(-1.9,-1);
			\draw[->, line width=1pt](-2.1,-1)--(-2.9,-1);
			\draw[->, line width=1pt](-3.1,-1)--(-3.9,-1);
			\draw[->, line width=1pt](-4.1,-1)--(-4.9,-1);
			\draw[->, line width=1pt](-1.1,-2)--(-1.9,-2);
			\draw[->, line width=1pt](-2.1,-2)--(-2.9,-2);
			\draw[->, line width=1pt](-3.1,-2)--(-3.9,-2);
			\draw[->, line width=1pt](-4.1,-2)--(-4.9,-2);
			\draw[->, line width=1pt](-1.1,-3)--(-1.9,-3);
			\draw[->, line width=1pt](-2.1,-3)--(-2.9,-3);
			\draw[->, line width=1pt](-3.1,-3)--(-3.9,-3);
			\draw[->, line width=1pt](-4.1,-3)--(-4.9,-3);
			\draw[->, line width=1pt](-1.1,-4)--(-1.9,-4);
			\draw[->, line width=1pt](-2.1,-4)--(-2.9,-4);
			\draw[->, line width=1pt](-3.1,-4)--(-3.9,-4);
			\draw[->, line width=1pt](-4.1,-4)--(-4.9,-4);
			\draw[->, line width=1pt](-1.1,-5)--(-1.9,-5);
			\draw[->, line width=1pt](-2.1,-5)--(-2.9,-5);
			\draw[->, line width=1pt](-3.1,-5)--(-3.9,-5);
			\draw[->, line width=1pt](-4.1,-5)--(-4.9,-5);
			\draw[->, line width=1pt](-1.1,-6)--(-1.9,-6);
			\draw[->, line width=1pt](-2.1,-6)--(-2.9,-6);
			\draw[->, line width=1pt](-3.1,-6)--(-3.9,-6);
			\draw[->, line width=1pt](-4.1,-6)--(-4.9,-6);
			\draw[->, line width=1pt](0.91,-2.09) to [out=225, in=0](-0.9,-3);
			\draw[->, line width=1pt](0.91,-3.09) to [out=225, in=0](-0.9,-4);
			\draw[->, line width=1pt](0.91,-4.09) to [out=225, in=0](-0.9,-5);
			\draw[->, line width=1pt](0.91,-5.09) to [out=225, in=0](-0.9,-6);
			\draw[->, line width=1pt](0.91,-1.09)  to [out=225, in=0](-0.9,-2);
			
			\draw[->, line width=1pt](1,0.9) to [out=270, in=45] (0.5,-0.5) to [out=225, in=0](-0.9,-1);
			\draw[->, line width=1pt](0,0.9) to [out=270, in=0] (-0.9,0);
			
			\draw[->, line width=1pt](1.9,-2.1)--(1.1,-2.9);
			\draw[->, line width=1pt](2.9,-2.1)--(2.1,-2.9);
			\draw[->, line width=1pt](3.9,-2.1)--(3.1,-2.9);
			\draw[->, line width=1pt](4.9,-2.1)--(4.1,-2.9);
			\draw[->, line width=1pt](5.5,-2.5)--(5.1,-2.9);
			\draw[->, line width=1pt](1.9,-3.1)--(1.1,-3.9);
			\draw[->, line width=1pt](2.9,-3.1)--(2.1,-3.9);
			\draw[->, line width=1pt](3.9,-3.1)--(3.1,-3.9);
			\draw[->, line width=1pt](4.9,-3.1)--(4.1,-3.9);
			\draw[->, line width=1pt](1.9,-3.1)--(1.1,-3.9);
			\draw[->, line width=1pt](5.5,-3.5)--(5.1,-3.9);
			\draw[->, line width=1pt](2.9,-4.1)--(2.1,-4.9);
			\draw[->, line width=1pt](3.9,-4.1)--(3.1,-4.9);
			\draw[->, line width=1pt](4.9,-4.1)--(4.1,-4.9);
			\draw[->, line width=1pt](1.9,-4.1)--(1.1,-4.9);
			\draw[->, line width=1pt](5.5,-4.5)--(5.1,-4.9);
			\draw[->, line width=1pt](2.9,-1.1)--(2.1,-1.9);
			\draw[->, line width=1pt](3.9,-1.1)--(3.1,-1.9);
			\draw[->, line width=1pt](4.9,-1.1)--(4.1,-1.9);
			\draw[->, line width=1pt](1.9,-1.1)--(1.1,-1.9);
			\draw[->, line width=1pt](5.5,-1.5)--(5.1,-1.9);
			\draw[line width=1pt](2.9,-5.1)--(2.6,-5.4);
			\draw[line width=1pt](3.9,-5.1)--(3.6,-5.4);
			\draw[line width=1pt](4.9,-5.1)--(4.6,-5.4);
			\draw[line width=1pt](1.9,-5.1)--(1.6,-5.4);
			

			
			\node at (-6,0) {$m$};
			\node at (7,0) {$m$};
			\node at (0,5.8) {$n$};
			\node at (0,-6.3) {$n$};
			\node at (-0.5,0.4) {$2$};
			\node at (0.4,-0.3) {$3$};
			\node at (-0.3,-2.7) {$2$};
			\node at (-0.3,-3.7) {$2$};
			\node at (-0.3,-4.7) {$2$};
			\node at (-0.3,-5.7) {$2$};
			\node at (-0.3,-1.7) {$2$};
			\node at (1.9,-0.3) {$2$};
			\node at (2.9,-0.3) {$2$};
			\node at (3.9,-0.3) {$2$};
			\node at (4.9,-0.3) {$2$};
			\node at (5.9,-0.3) {$2$};
			\node at (6.6,6){$M(m,n,0,0)$};
			\node at (-5.5,-6.8){$M(0,0,m,n)$};
			\node at (6,-6.8){$M(m,0,0,n)$};
		\end{tikzpicture}$$}
	\caption{\label{E510morphismsbis} All nonzero morphisms between finite generalized Verma modules for $E(5,10)$ of degree 2 and 3 and their infinite bilateral complexes.} 
	\end{center}
	\end{figure}
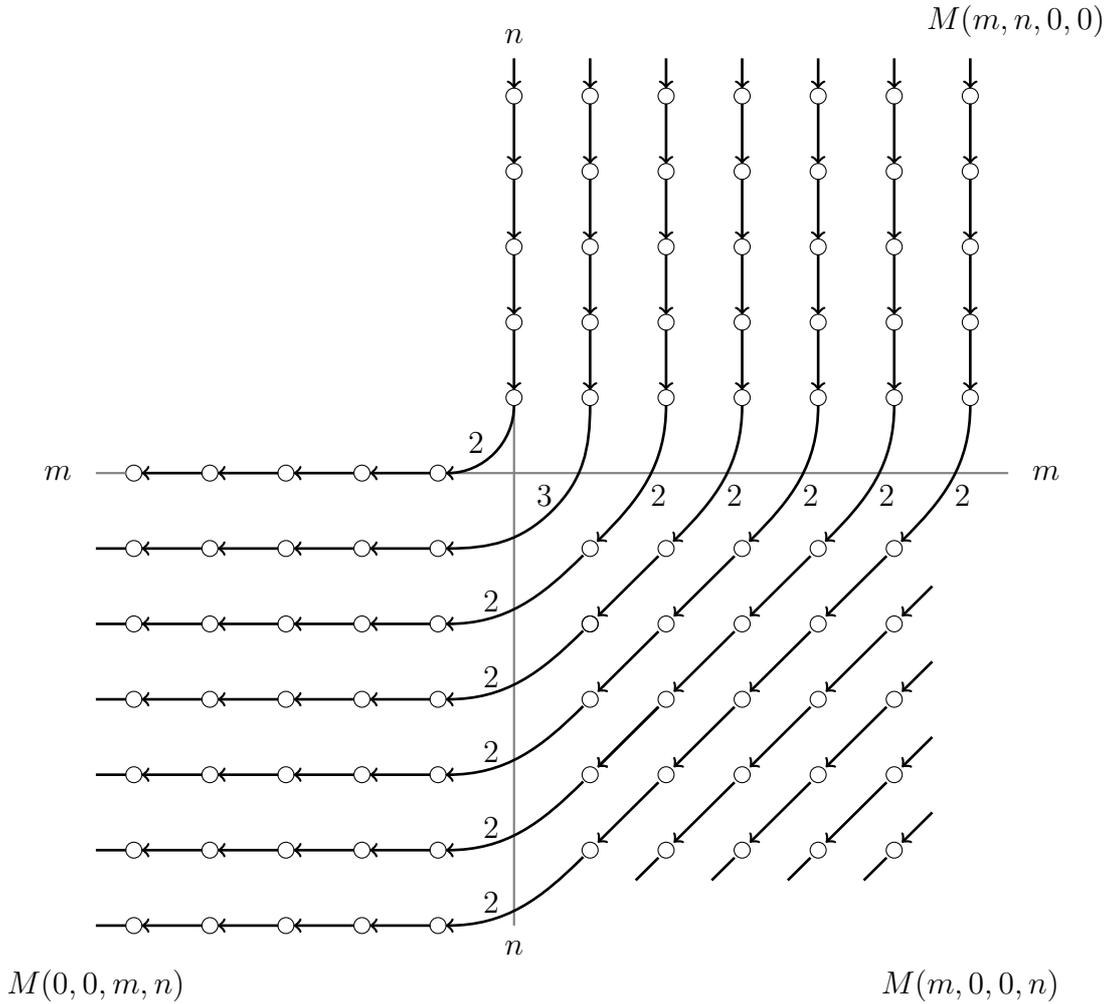
\begin{ack} N.C. and F.C. would like to thank Alessandro D'Andrea for useful conversations. F.C. would like to thank his wife Rossana for her support in the preparation of the computer program that has been used in this work.
\end{ack}

\end{document}